\documentclass{article}

\usepackage{
amsmath,
amsthm,
amscd,
amssymb,
}
\usepackage{comment}
\usepackage{tikz}
\usepackage{wasysym}
\usetikzlibrary{matrix,positioning,arrows,calc,decorations.pathreplacing}

\usepackage{url}
\usepackage{subfig}

\theoremstyle{plain}
\newtheorem{lemma}{Lemma}[section]
\newtheorem{corollary}[lemma]{Corollary}
\newtheorem{proposition}[lemma]{Proposition}
\newtheorem{theorem}[lemma]{Theorem}

\theoremstyle{definition}
\newtheorem{definition}[lemma]{Definition}
\newtheorem{question}[lemma]{Question}

\theoremstyle{remark}
\newtheorem{remark}[lemma]{Remark}
\newtheorem{example}[lemma]{Example}
\newtheorem{claim}[lemma]{Claim}

\newcommand{\R}{\mathbb{R}}

\newcommand{\Z}{\mathbb{Z}}

\newcommand{\N}{\mathbb{N}}

\newcommand{\X}{\mathcal{X}}

\newcommand{\Di}{\mathcal{T}}
\newcommand{\VDi}{\mathcal{R}}
\newcommand{\VVDi}{\mathcal{V}}

\newcommand{\ID}{\mathrm{id}}
\newcommand{\fix}[2]{\mathrm{Fix}_{#1}(#2)}

\newcommand{\fin}[2]{\mathrm{Fin}_{#1}(#2)}
\newcommand{\loc}[2]{\mathrm{Loc}_{#1}(#2)}

\newcommand\xqed[1]{%
  \leavevmode\unskip\penalty9999 \hbox{}\nobreak\hfill
  \quad\hbox{#1}}
\newcommand{\tqed}{\xqed{$\fullmoon$}}

\newcommand{\Aut}{\mathrm{Aut}}

\newcommand{\rig}{\rightarrow}
\newcommand{\upp}{\uparrow}
\newcommand{\lef}{\leftarrow}
\newcommand{\dow}{\downarrow}

\newcommand{\theorempreskipamount}{0.5cm}
\newcommand{\theorempostskipamount}{0.5cm}

\begin{document}

\title{Nilpotent Endomorphisms of Expansive Group Actions}

\author{
  Ville Salo
  \footnote{Author supported by Academy of Finland grant 2608073211.}
  \ \ \ \ \ \ and\ \ \ \ \
  Ilkka T\"orm\"a
  \footnote{Author supported by Academy of Finland grant 295095.}
  \\
  Department of Mathematics and Statistics \\
  University of Turku, Finland \\
  \{\texttt{vosalo}, \texttt{iatorm}\}\texttt{@utu.fi}
}

\maketitle

\begin{abstract}
We consider expansive group actions on a compact metric space containing a special fixed point denoted by $0$, and endomorphisms of such systems whose forward trajectories are attracted toward $0$.
Such endomorphisms are called \emph{asymptotically nilpotent}, and we study the conditions in which they are \emph{nilpotent}, that is, map the entire space to $0$ in a finite number of iterations.
We show that for a large class of discrete groups, this property of \emph{nil-rigidity} holds for all expansive actions that satisfy a natural specification-like property and have dense homoclinic points.
Our main result in particular shows that the class includes all residually finite solvable groups and all groups of polynomial growth.
For expansive actions of the group $\Z$, we show that a very weak gluing property suffices for nil-rigidity.
For $\Z^2$-subshifts of finite type, we show that the block-gluing property suffices.
The study of nil-rigidity is motivated by two aspects of the theory of cellular automata and symbolic dynamics: It can be seen as a finiteness property for groups, which is representative of the theory of cellular automata on groups. Nilpotency also plays a prominent role in the theory of cellular automata as dynamical systems.
As a technical tool of possible independent interest, the proof involves the construction of \emph{tiered dynamical systems} where several groups act on nested subsets of the original space. 
\end{abstract}



\section{Introduction}

\subsection{Finiteness Notions for Groups}

Let $G$ be a discrete group (all groups in this paper are discrete unless otherwise specified), and consider the full $G$-shift $X = \Sigma^G$ over a finite alphabet $\Sigma$ and the action of $G$ on it by translations, defined by $(g \cdot x)_h = x_{g^{-1} h}$.
The endomorphisms of this dynamical system, i.e. the continuous self-maps commuting with the action, are defined by a local rule: for an endomorphism $f : X \to X$ and $x \in X$, the value of a coordinate $f(x)_g$ depends only on the contents of the local pattern at $g$, that is, the restriction $(g^{-1} \cdot x)|_N$, where $N \subset G$ is a fixed finite set, called a \emph{neighborhood} of $f$. 

These endomorphisms are known as cellular automata (CA for short).
As proved by Moore \cite{Mo62} and Myhill \cite{My63}, cellular automata on the grids $G = \Z^d$ are \emph{surjunctive}, meaning that they cannot be both injective and nonsurjective.
In this sense, the dynamical system $(\Z^d, X)$ behaves like a finite set -- indeed, this property can be seen as a version of Dedekind-finiteness in a suitably defined category of shift spaces. 
Groups $G$ such that all injective cellular automata on $\Sigma^G$ are surjective are also called surjunctive, and Gottshalk's conjecture states that every group is surjunctive \cite{Go73}. Since finite groups are obviously surjunctive, and since proofs of surjunctivity are based on approximating infinite groups by finite groups or graphs, surjunctivity is considered a \emph{finiteness notion} for groups.

Gromov \cite{Gr99} proves a general result in an algebraic setting, showing in particular that every sofic group (in particular every amenable and residually finite group) is surjunctive. The term ``sofic'' is coined in \cite{We00} (derived from the Hebrew word for ``finite''), where another proof of the case of cellular automata is given. See \cite{We00} and exercises in \cite{CeCo10} for remarks on surjunctivity of cellular automata on subshifts.
See the paper of Gromov and \cite{CeCo11} for connections with algebra and the Ax-Grothendieck theorem. 

There are many other finiteness notions of discrete groups that can be fully characterized in terms of cellular automata: \emph{amenability} has several characterizations (see in particular \cite{BaKi16}), and \emph{torsion groups} can be characterized by subshifts definable by the dynamics of a certain family of cellular automata \cite{SaTo17}. See also \cite{CeCo10,CaKaTa16,MeSa17}.

The present paper studies another cellular automata based finiteness notion for groups, which we call \emph{nil-rigidity}. This notion is defined with respect to eventually trivial dynamics of cellular automata. Like surjunctivity, the property of nil-rigidity is one that holds trivially for finite groups, less trivially for a large class of infinite groups, and the proof for infinite groups is based on approximability of the group by finite graphs. Unlike surjunctivity where the very weak (and purely geometric) approximability assumption of soficness is enough, our proof requires strict algebraic assumptions -- residual finiteness and a slightly weakened variant of solvability. Nevertheless, like with surjunctivity, we do not know any counterexamples to nil-rigidity in the case of full shifts $\Sigma^G$. 

\subsection{Nilpotency and Nil-rigidity}

Let $0 \in \Sigma$ be a special symbol.
A cellular automaton $f$ on $X = \Sigma^G$ is weakly nilpotent, if for every $x \in X$ there exists $n \in \N$ with $f^n(x) = 0^G$; if there is a global bound on $n$, then $f$ is nilpotent.
It is known that every weakly nilpotent CA is nilpotent on any subshift; 
see \cite{Sa12c} for the case of subshifts over abelian groups, and Proposition~\ref{prop:WeakNilp} for the case of expansive systems over arbitrary groups. This is a strong, but easily proved, finiteness notion for general groups.

As a further weakening of nilpotency, we say the CA is asymptotically nilpotent if $f^n(x)$ converges toward $0^G$ for all $x \in X$ (in the product topology).
In \cite{GuRi08}, it was proved that every asymptotically nilpotent cellular automaton on a full $\Z$-shift is nilpotent.
This was extended to CA on the multidimensional full shifts $\Sigma^{\Z^d}$ by the first author in \cite{Sa12c}. The references \cite{GuRi08,Sa12c,GuRi10,Sa16} also deal with some more general subshifts of finite type, but no such results have been shown for groups that are not (locally virtually) abelian. In \cite{Sa17}, the first author termed a cellular automaton to be \emph{strictly asymptotically nilpotent}, or \emph{SAN} for short, if it is asymptotically nilpotent but not nilpotent, and asked if SAN cellular automata exist on full shifts $\Sigma^G$ on \emph{any} group $G$. In this paper, we call a group $G$ \emph{nil-rigid}, if $\Sigma^G$ does not support any SAN cellular automata for any finite alphabet $\Sigma$.

In addition to being a natural finiteness notion of groups, another strong motivation for a detailed study of nilpotency itself comes from the classical theory of one-dimensional cellular automata, i.e. dynamics of endomorphisms of $\Sigma^\Z$.
Namely, there are several strictly weaker notions of nilpotency of cellular automata, for which the corresponding nil-rigidity result is not true, by counter-examples which are often difficult to construct.
We mention nilpotency on $0$-finite configurations and nilpotency on periodic configurations, which mean that every configuration containing a finite number of nonzero symbols, or every periodic configuration, eventually evolves to $0^G$.
In the case $G = \Z$, neither are equivalent to nilpotency.
In the first case, this is witnessed by the minimum automaton on $\{0, 1\}^\Z$ defined by $f(x)_i = \min(x_i, x_{i+1})$, and in the second case by the construction in \cite{Ka92}.
By a classic example of Gacs, Kurdyumov and Levin \cite{GaKuLe78}, a cellular automaton on $\{0, 1\}^\Z$ can even be nilpotent toward $0^\Z$ on $0$-finite configurations and toward $1^\Z$ on $1$-finite configurations.
Another variant of nilpotency is unique ergodicity, or possessing a unique invariant Borel probability measure.
A cellular automaton $f$ that fixes the uniform point $0^G$ is uniquely ergodic if and only if the asymptotic density of $0$-states $\lim_{n \longrightarrow \infty} \frac{1}{n} | \{ 0 \leq k < n \;|\; f^k(x)_{1_G} = 0 \} |$ is $1$ for all $x \in \Sigma^G$.
By a result of the second author \cite{To15}, unique ergodicity does not imply nilpotency even on full $\Z$-shifts.

Another important aspect of nilpotency are the connections to recursion theory. Given the local rule of a cellular automaton, it is algorithmically undecidable whether it defines a nilpotent cellular automaton, even in the case of full $\Z$-shifts \cite{AaLe74,Ka92}. By the above results, the same holds for weak and asymptotic nilpotency. In other words, while asymptotic nilpotency is equivalent to nilpotency for cellular automata, there is no decidable characterization of local rules giving rise to this property. The proof of this result implies also that for any computable function $\phi : \N \to \N$, and any finite alphabet $\Sigma$ with more than one element, for any large enough $n \in \N$ there is a nilpotent cellular automaton $f : \Sigma^\Z \to \Sigma^\Z$ with neighborhood $\{-n,-n+1,...,n-1,n\}$, such that $f^{\phi(n)}(x) \neq 0^\Z$ for some configuration $x \in \Sigma^\Z$. 


\subsection{Our Results}

In this article, we provide a large class of examples of nil-rigid groups, and more generally prove nil-rigidity results for a large class of expansive dynamical systems.
In our formalism, the special symbol is replaced by a special fixed point $0 \in X$, making $(G, X, 0)$ a pointed dynamical system.
Our main result states that for a large class of such systems and all endomorphisms $f$ with $f^n(x) \longrightarrow 0$ for all $x \in X$, the convergence toward $0$ is uniform (which in the case that $G$ acts expansively implies $0$ is reached in finite time).
In this case, we say that the system $(G, X, 0)$ is \emph{nil-rigid}, so in terms of our earlier terminology, $G$ is nil-rigid if all pointed full shifts $(G, \Sigma^G, 0^G)$ are nil-rigid.


Since most groups are not nil-rigid on all expansive systems (even $\Z$ is not), it is necessary to impose some additional constraints on them.
As a natural generalization of full shifts, we consider expansive pointed systems where the set of \emph{homoclinic points} (the points asymptotic to $0$ in the group action) is dense, and which are \emph{$0$-gluing}.
This notion is related to specification, and states that any collection of orbit segments can be `glued' together into a single point, provided that the borders of the segments are close enough to the special point $0$.
It has been used in the context of subshifts at least in the preprint \cite{Sa16}.
We prove in Proposition~\ref{prop:ShadowingIsGluing} that all systems with the shadowing property are $0$-gluing, which provides a large class of examples.
In particular, this holds for all shifts of finite type.
By considering different weakenings of the $0$-gluing property, we can prove that some well known one-dimensional subshifts are nil-rigid.
The following result shows the principal content of our main technical results, Proposition~\ref{prop:OneTierLocal} and Theorem~\ref{thm:BigClass}, which are proved in sections \ref{sec:ClosureProps} through \ref{sec:GroupsWeGet}.

\vspace{\theorempreskipamount}
\noindent
\textsc{Theorem.}
\textit{
Let $G$ be a discrete group such that every finitely generated subgroup $H \leq G$ is residually finite and has a subnormal series $\{1\} = H_0 \triangleleft H_1 \triangleleft \cdots \triangleleft H_n = H$, where each quotient $H_{i+1} / H_i$ is locally virtually abelian.
Let $(G, X, 0)$ be an expansive $0$-gluing pointed dynamical system with dense homoclinic points.
Then every asymptotically nilpotent endomorphism $f : X \to X$ is nilpotent.
}
\vspace{\theorempostskipamount}

This class of groups includes all polycyclic groups such as the discrete Heisenberg group, all groups of polynomial growth, and all metabelian groups such as the lamplighter group $\Z_2 \wr \Z$. 
We show that it also contains several groups which are not locally virtually solvable.
It does not include free groups on two or more generators, whose nil-rigidity is left as an open question.
As a less technical corollary, we mention the following.

\vspace{\theorempreskipamount}
\noindent 
\textsc{Corollary.}
\textit{
Let $G$ be a residually finite virtually solvable group.
Let $(G, X, 0)$ be an expansive pointed dynamical system with the shadowing property and dense homoclinic points.
Then every asymptotically nilpotent endomorphism $f : X \to X$ is nilpotent.
}
\vspace{\theorempostskipamount}

In the zero-dimensional setting, expansivity of the group action means that $X$ is a subshift, the shadowing property means it is of finite type (defined by a finite set of forbidden patterns), and dense homoclinic points means the set of points with finite support is dense. In particular, the above result implies that if $G$ is residually finite and virtually solvable, and $X \subset \Sigma^G$ is a subshift of finite type where points with finite support are dense, then every cellular automaton on $X$ is nilpotent. This covers in particular the full shift $X = \Sigma^G$.

On the group $\Z$, much weaker gluing properties suffice for the result.
The following is a corollary of the rather technical Lemma~\ref{lem:Z}.

\vspace{\theorempreskipamount}
\noindent
\textsc{Theorem.}
\textit{
Let $(\Z, X, 0)$ be an expansive pointed $\Z$-dynamical system such that for each homoclinic point $x \in X$ and all $\epsilon > 0$, there exist two points $x^1, x^2 \in X$ such that for each $i \in \{1, 2\}$ and $n \in \Z$ there exists $k \in \Z$ with $d(n \cdot x^i, k \cdot x) < \epsilon$, and we can choose $k = 0$ for infinitely many positive $n$ for $i = 1$, and infinitely many negative $n$ for $i = 2$.
Then every asymptotically nilpotent endomorphism $f : X \to X$ is nilpotent.
}
\vspace{\theorempostskipamount}

We also show that in the case of two-dimensional SFTs, the so-called block-gluing property -- that any two valid rectangular patterns whose domains are separated by a constant can be glued to a common pattern -- suffices for nil-rigidity.

\vspace{\theorempreskipamount}
\noindent
\textsc{Theorem.}
\textit{
Let $X \subset \Sigma^{\Z^2}$ be a block-gluing subshift of finite type. Then every asymptotically nilpotent endomorphism $f : X \to X$ is nilpotent.
}
\vspace{\theorempostskipamount}


\subsection{Outline of Methods}

The result on one-dimensional cellular automata was proved in \cite{GuRi08} by exploiting the existence of a blocking word and the two-endedness of $\Z$ to prevent information flow under the dynamics of asymptotically nilpotent cellular automata.
The authors 
use the $0$-gluing property of the full shift, as they construct a configuration containing an infinite sequence of nested patterns that live for some number of steps under the action of the CA, send the central cell into a nonzero state and then disappear.
Such a configuration does not converge toward $0^\Z$, so the CA cannot be asymptotically nilpotent.

The new idea in \cite{Sa12c}, which proved the multidimensional case, was \emph{periodization}: if a finite pattern is repeated periodically in some direction on $\Z^d$, the result can be seen as a configuration on $\Z^{d-1}$ over a larger alphabet.
Since a cellular automaton preserves the set of configurations of a given period, one sees by induction on $d$ that an symptotically nilpotent CA will eventually send any given periodic configuration to $0^{\Z^d}$, and use this fact to construct the sequence of nested patterns.
The proof uses the density of homoclinic points and the $0$-gluing property of $\Sigma^{\Z^d}$, as well as the fact that the groups $\Z^d$ are residually finite, to transform an arbitrary finite pattern into a periodic configuration.

On a high level, our main arguments are somewhat similar to the aforementioned ones.
We first extend the result of \cite{GuRi08} to expansive $\Z$-dynamical systems satisfying a very weak variant of the $0$-gluing property using an auxiliary result from the preprint \cite{Sa16} of the first author.
This proof does not generalize beyond virtually cyclic groups. 
Then we show that the property of a group being nil-rigid on certain more well-behaved classes of systems is preserved under group extensions, given that the extension is residually finite.
The proofs rely on periodization and the construction of nested patterns.

The main new idea in this article is \emph{tiered dynamical systems}.
In our construction, we would like to handle the set of periodic points or homoclinic points as subsystems, and proceed inductively along a subnormal series of the group. For this, we need the systems constructed to inherit the properties of $0$-gluing and dense homoclinic points from the parent system, but they typically only inherit these properties in a weaker sense, as tiered dynamical systems.
This problem was not present in \cite{Sa12c} due to the simple direct product structure of $\Z^d$ and the especially strong gluing properties of full shifts.

Tiered dynamical systems consist of several nested subsets of an ambient state space, each of which may have a slightly different group acting on it.
In this model, orbit segments can be glued together and homoclinic points can be extracted, but the resulting point usually lies on a higher tier than the original. We prove that tiered dynamical systems have better closure properties for our purposes, and indeed extracting periodic or homoclinic points from a well-behaved tiered dynamical system still has the same good properties. This allows us to extend the periodization argument to a larger class of residually finite groups.

\section{Preliminaries}

In this article, $\Z_n = \Z/n\Z$ is the integers modulo $n$, and $\N = \{0, 1, 2, \ldots\}$.
For bi-infinite sequences, we may use the notation $x = \cdots x_{-3} x_{-2} x_{-1} . x_0 x_1 x_2 \cdots$, where the coordinate $0$ is immediately to the right of the period.
For an exact sequence of groups $1 \to K \to G \to H \to 1$, we say $G$ is $K$-by-$H$, and that $G$ is an extension of $H$ by $K$.
If $G$ has a finite-index subgroup isomorphic to $H$, we say it is virtually $H$.
These notions are extended to classes of groups in the natural way.
If $\mathcal{G}$ is a class of groups, we say $G$ is locally $\mathcal{G}$ if its finitely generated subgroups are in $\mathcal{G}$.

\subsection{Expansive Systems}

A \emph{pointed dynamical system} is a quadruple $(G, X, d, \rho, 0)$, where $G$ is a discrete group (not necessarily infinite, and not necessarily finitely generated), $(X, d)$ is a compact metric space, $\rho : G \times X \to X$ is a continuous left action of $G$ on $X$, and $0 \in X$ is a special point with $\rho(g, 0) = 0$ for all $g \in G$.
The action is denoted by $\rho(g, x) = g \cdot x$ for $g \in G$ and $x \in X$ whenever this creates no ambiguity. 
If the action, metric and special point are clear from the context, we denote the system simply by $(G, X)$.
We may also denote the special point by $0_X$ and the metric by $d_X$.
A \emph{morphism} between two systems $(G, X)$ and $(G, Y)$ with the same group is a continuous function $f : X \to Y$ with $f(0_X) = 0_Y$ and $f(g \cdot x) = g \cdot f(x)$ for all $g \in G$ and $x \in X$.

For a subset $K \subseteq G$, the set $\{ x \in X \;|\; \forall g \in K : g \cdot x = x \}$ of fixed points under $K$ is denoted $\fix{K}{X}$.
It is a closed, and thus compact, subset of $X$.
If $K \trianglelefteq G$ is a normal subgroup, the quotient group $G / K$ acts continuously on $\fix{K}{X}$ by $g K \cdot x = g \cdot x$, so $(G/K, \fix{K}{X})$ is a pointed dynamical system.

A pointed dynamical system $(G, X)$ is \emph{expansive} if there exists \emph{an expansivity constant} $0 < \Delta \leq 1$ such that $d_X(g \cdot x, g \cdot y) \leq \Delta$ for all $g \in G$ implies $x = y$. If $(G, X)$ is expansive, we define \emph{the canonical expansivity constant of $(G, X)$} by
\begin{equation}
\label{eq:Expansivity}
\Delta_X = \sup \{\Delta/2 \;|\; \Delta \mbox{ is an expansivity constant of } (G, X) \},
\end{equation}
and observe that $\Delta_X$ is an expansivity constant. Having a canonical expansivity constant is useful, since we will consider multiple dynamical systems on subspaces of the same metric space. We divide by $2$ since the supremum itself is not obviously an expansivity constant due to the choice of the half-open interval used to define expansivity, and this choice makes the expansivity constant behave better under limits, for our purposes.

A point $x \in X$ is \emph{homoclinic} (sometimes called \emph{$0$-finite} in symbolic dynamics) if for all $\epsilon > 0$ there exists a finite set $E \subseteq G$ such that $d_X(g \cdot x, 0_X) < \epsilon$ for all $g \in G \setminus E$.
For $\delta > 0$, the \emph{$\delta$-shadow} of a point $x \in X$ is the set $\Theta(\delta, x) = \{ g \in G \;|\; d_X(g \cdot x, 0_X) > \delta \}$.

The following lemmas will be used frequently in the course of this paper, sometimes without explicit mention.

\begin{lemma}
\label{lem:Nhood}
Let $(G, X)$ be an expansive pointed dynamical system, $Y$ a metric space and $f : X \to Y$ a continuous function.
Let $\Delta$ be the canonical expansivity constant of $(G, X)$.
For all $\epsilon > 0$ and $0 < \delta \leq \Delta$, there exists a finite set $N \subseteq G$ such that for all $x, y \in X$, the condition $d_X(g \cdot x, g \cdot y) \leq \delta$ for all $g \in N$ implies $d_Y(f(x), f(y)) < \epsilon$.
\end{lemma}

\begin{proof}
Suppose on the contrary that for each finite set $N \subseteq G$ there exist $x_N, y_N \in X$ with $d_X(g \cdot x_N, g \cdot y_N) \leq \delta$ for all $g \in N$, but $d_Y(f(x_N), f(y_N)) \geq \epsilon$.
Then we have a net $(x_N, y_N)_{N \subseteq G_d}$ in $X \times X$, where $N$ ranges over the finite subsets of $G$, and since $X \times X$ is compact, the net has a limit point $(x, y)$.
It satisfies $d_Y(f(x), f(y)) \geq \epsilon$ and $d_X(g \cdot x, g \cdot y) \leq \delta$ for all $g \in G$.
The latter condition implies $x = y$ since $\delta \leq \Delta$, a contradiction.
\end{proof}

In particular, for any $\delta \leq \Delta$, a point $x \in X$ is homoclinic if and only if it has a finite $\delta$-shadow, and only $0_X$ has an empty $\delta$-shadow.

\begin{lemma}
\label{lem:BoundedShadow}
Let $(G, X)$ be an expansive pointed dynamical system with expansivity constant $\Delta$.
For each finite $F \subseteq G$ and $0 < \delta \leq \Delta/2$, the set $Y = \{ x \in X \;|\; \Theta(\delta, x) \subseteq F \}$ is finite.
\end{lemma}

\begin{proof}
For each $y \in Y$, let $B_y = \{ x \in Y \;|\; \forall g \in F : d_X(g \cdot x, g \cdot y) < \delta \}$.
These sets form an open cover of $Y$.
Since $Y$ is closed in $X$, it is compact, hence there is a finite set $Z \subseteq Y$ with $Y = \bigcup_{z \in Z} B_z$.
On the other hand, for each $z \in Z$ and $y \in B_z$ we have $d_X(g \cdot z, g \cdot y) < \Delta/2$ for $g \in F$, and $d_X(g \cdot z, g \cdot y) \leq d_X(g \cdot z, 0_X) + d_X(0_X, g \cdot y) \leq \Delta$ for $g \in G \setminus F$.
Thus $y = z$, so $Y = Z$ is finite.
\end{proof}

An endomorphism $f : X \to X$ of a pointed dynamical system $(G, X)$ is
\begin{itemize}
\item \emph{nilpotent}, if $f^n(X) = \{0_X\}$ for some $n \in \N$,
\item \emph{weakly nilpotent}, if for all $x \in X$ there exists $n \in \N$ with $f^n(x) = 0_X$,
\item \emph{asymptotically nilpotent}, if $\lim_{n \longrightarrow \infty} f^n(x) = 0_X$ for all $x \in X$, and
\item \emph{uniformly asymptotically nilpotent}, if it is asymptotically nilpotent and the convergence toward $0_X$ is uniform.
\end{itemize}
Asymptotic nilpotency is the weakest of these notions.
Note that for expansive systems, uniform asymptotic nilpotency is equivalent to nilpotency.
A point $x \in X$ is called \emph{mortal} if $f^n(x) = 0_X$ for some $n \in \N$.
Weak nilpotency means exactly that all points are mortal, and nilpotency is its uniform version.

\begin{definition}
\label{def:NRsystem}
A pointed dynamical system $(G, X)$ is \emph{nil-rigid} if every asymptotically nilpotent endomorphism of $(G, X)$ is uniformly asymptotically nilpotent.
\end{definition}

\subsection{Symbolic Dynamics}
\label{sec:SymbDyn}

As far as we know, the most interesting applications of our theorems are to cellular automata (CA) on subshifts on groups, so we give some definitions here, see also standard references such as \cite{LiMa95,CeCo10}. Subshifts are a large family of dynamical systems where we obtain nil-rigidity results, and examples of them are presented in Section~\ref{sec:FinalExamples} of the paper.

If $G$ is a discrete group and $\Sigma$ a finite set, called an \emph{alphabet}, then $\Sigma^G$ is the \emph{full $G$-shift over $\Sigma$}, and its elements are called \emph{configurations}.
For $s \in \Sigma$, the \emph{all-$s$ configuration} $x = s^G$ is defined by $x_g = s$ for all $g \in G$.
We define a left $G$-action on $\Sigma^G$ by $(g \cdot x)_h = x_{g^{-1} h}$, which is called the \emph{left shift}.
We also equip $\Sigma^G$ with the product topology, which makes it a compact metrizable space.
The left shift action is continuous in this topology.

A \emph{$G$-subshift} is a topologically closed set $X \subset \Sigma^G$ satisfying $G \cdot X = X$.
A \emph{$G$-block code} between two $G$-subshifts $Y \subset \Sigma^G$ and $Y \subset \Gamma^G$ is a continuous function $f : X \to Y$ that commutes with the left shifts, in the sense that $f(g \cdot x) = g \cdot f(x)$ holds for all $x \in X$ and $g \in G$.
A \emph{$G$-cellular automaton} (CA for short) is a $G$-block code from a $G$-subshift to itself. A subshift $X \subset \Sigma^G$ is a \emph{shift of finite type}, or \emph{SFT} for short, if there exists a clopen set $C \subset \Sigma^G$ such that for $x \in \Sigma^G$, we have $x \in X$ if and only if $g \cdot x \in C$ for all $g \in G$.
A \emph{sofic shift} is an image of an SFT under a $G$-block code.

Every $G$-subshift containing a uniform configuration $0^G$ is an expansive pointed $G$-dynamical system.
Up to topological conjugacy, they are exactly the zero-dimensional expansive pointed dynamical systems.
The morphisms between $G$-subshifts as dynamical systems are exactly the block codes.
A cellular automaton $f$ on $\Sigma^G$ is asymptotically nilpotent iff for each $x \in X$ and $g \in G$ there exists $n \in \N$ with $f^m(x)_g = 0$ for all $m \geq n$.

A \emph{pattern (on $G$)} is a function $P \in \Sigma^D$, where $D = D(P)$ is a finite subset of $G$, called the \emph{domain} of $P$.
Each pattern $P$ defines a \emph{cylinder set} $[P] = \{x \in \Sigma^G \;|\; x|_D = P\}$.
The clopen sets in $\Sigma^G$ are precisely the finite unions of cylinders, and form a basis for the topology.
Every subshift $X$ has a (possibly infinite) set $P_X$ of \emph{forbidden patterns} such that $X = \bigcap_{P \in P_X, g \in G} \{x \in \Sigma^G \;|\; x \notin g \cdot [P]\}$.
If $P_X$ can be chosen finite, then $X$ is an SFT.
Every cellular automaton $f : X \to X$ has a finite \emph{neighborhood} $N \subset G$ and a \emph{local function} $F : \Sigma^N \to \Sigma$ such that $f(x)_g = F((g^{-1} \cdot x)|_N)$ holds for all $x \in X$ and $g \in G$.
Note that we while do not require $G$ to be finitely generated, having a local rule means a CA $f$ can only ``see'' a finitely generated subgroup $H \leq G$, and its dynamics splits into a product system consisting of independent copies (one for each left coset of $H$) of the ``restriction'' of $f$ to $\Sigma^H$. 

If $G$ is a group and $X \subset \Sigma^G$ is a subshift, then a pattern $P \subset \Sigma^A$ is \emph{valid for $X$} (or simply \emph{valid} if $X$ is clear from the context) if there exists a point $x \in X$ such that $x|_A = P$, and in this case \emph{$P$ appears in $x$} or \emph{$x$ contains $P$}. We say \emph{$P$ appears in $x$ at $g \in G$} if $(g^{-1} \cdot x)|_A = P$, and then \emph{a translated copy of $P$ appears in $x$}, or \emph{$x$ contains a translated copy of $P$}. The \emph{support} of $x \in X$ is the set $\{ g \in G \;|\; x_g \neq 0 \}$.

If $S \subset G$ is a generating set, then we may define the \emph{(right) Cayley graph} with respect to $S$ as the labeled directed graph $(V, E, L)$, where $V = G$ is the set of vertices, $L = S$ is the set of edge labels, and $E = \{(g, g s, s) \;|\; g \in G, s \in S\}$ is the set of edges.
In other words, for each $g \in G$ and $s \in S$, there is an edge from $g$ to $g s$ with label $s$.
When the local rule of a cellular automaton looks for the neighbors of a cell $g \in G$, it looks at the cells with index $g h$, where $h \in G$ comes from a finite set.
This means that it is looking at the neighbors of $g$ in the Cayley graph, if the generating set $S$ is chosen properly.

Note that the characterization of cellular automata by local rules implies in particular that for any group $G$, the full shift $\Sigma^G$ has many endomorphisms -- every local rule defines a cellular automaton $f : \Sigma^G \to \Sigma^G$. In the case of subshifts of finite type (and sometimes more general classes of subshifts), the \emph{marker method} can be used to construct cellular automata, see \cite{BoLiRu88,Ho10} for some examples.






\subsection{Gluing}
\label{sec:Gluing}

In this section we define the notion of $0$-gluing, compare it to some other notions in the literature, and state some preliminary results about it and its variants. The weaker `variably' variant of the definition is only used in the case of virtually $\Z$ groups, but we give the definition for arbitrary groups. In the case of other groups, we only have results for systems that are $0$-gluing and have dense homoclinic points.

For two sets $A, E \subseteq G$, the \emph{$E$-border of $A$} is the set $\partial_E A = \{g \in A \;|\; E g \not\subseteq A\}$.

\begin{definition}
\label{def:Gluing}
Let $(G, X)$ be a pointed dynamical system.
A \emph{designation} in $(G, X)$ is a collection $\mathcal{A} = (A_i, x_i)_{i \in I}$, where $x_i \in X$ for all $i \in I$ and $(A_i)_{i \in I}$ is a partition of $G$.
If $(A_i)_{i \in I}$ is a disjoint family but not a partition, then we use the shorthand $(A_i, x_i)_{i \in I}$ for the designation $(A_i, x_i)_{i \in I} \cup (G \setminus \bigcup_{i \in I} A_i, 0_X)$.
We say $\mathcal{A}$ is a \emph{$(\delta, E)$-designation} for $\delta > 0$ and $E \subseteq G$, if for each $i \in I$ we have $d_X(g \cdot x_i, 0) < \delta$ for all $g \in \partial_E A_i$.

Let $\Gamma$ be a set of self-homeomorphisms of $X$ that fix $0_X$.
For $\epsilon > 0$, an \emph{$(\epsilon, \Gamma)$-realization of $(A_i, x_i)_{i \in I}$} is a pair $(y, W)$ containing a point $y \in X$ and a collection $W = (\gamma_i)_{i \in I}$ of elements of $\Gamma$ such that $d_X(g \cdot y, \gamma_i(g \cdot x_i)) < \epsilon$ whenever $g \in A_i$.
For a set $D \subseteq G$, an $(\epsilon, D)$-realization is an $(\epsilon, \Gamma)$-realization with $\Gamma$ being the set of functions $x \mapsto g \cdot x$ for $g \in D$; each $\gamma \in \Gamma$ is identified with the corresponding element of $D$.
An $\epsilon$-realization is an $(\epsilon, \Gamma)$-realization with $\Gamma = \{\mathrm{id}_X\}$, and consists only of the point $y$.

Given a set $\Lambda$ of self-homeomorphisms of $X$ that fix $0_X$, we say $(G, X)$ is \emph{$\Lambda$-$0$-gluing}, if there exists a finite set $\Gamma \subseteq \Lambda$ such that for all $\epsilon > 0$ there exist $\delta > 0$ and $E \subseteq G$ finite such that every $(\delta, E)$-designation in $(G, X)$ has an $(\epsilon, \Gamma)$-realization.
We say $(G, X)$ is \emph{variably $0$-gluing}, if we can choose $\Lambda$ to be the set of translations by all elements of $G$.
If we can choose $\Gamma = \{\mathrm{id}_X\}$, then $X$ is \emph{$0$-gluing}.
\end{definition}

In the case of a one-dimensional pointed subshift $X \subseteq \Sigma^\Z$ (whose special point is $0^\Z \in X$), $0$-gluing is equivalent to the existence of $r \geq 0$ such that if $v 0^r, 0^r w \in \Sigma^*$ are words occurring in $X$, then $v 0^r w$ also occurs in $X$; in other words, $0^r$ is a \emph{synchronizing word} for $X$.
The intuition is that the two words can be glued together by the parts that contain only $0$-symbols.



\emph{Specification} originates in \cite{Bo71} (see also \cite{Si74}), and states that any two orbit segments can be glued together into a periodic point (with some constant gluing time where we cannot control the orbit). There are many variants of this notion. While the definition of $0$-gluing bears some resemblance to specification, the two properties are incomparable even for expansive zero-dimensional $\Z$-systems with dense homoclinic points, and for all variants of specification that we know.

Consider first the subsystem of the ternary full shift $\{0, 1, 2\}^\Z$ defined by forbidding the patterns $1 0^n 1$ for all $n \geq 0$, where the special point is the all-$0$ configuration.
This system has a very strong specification property, since any two allowed patterns can be joined by inserting a $2$-symbol between them, but it is not $0$-gluing, since two occurrences of $1$ cannot be glued together using only $0$-symbols.
Conversely, consider the subsystem $X \subset \{0, 1\}^\Z$ defined by forbidding all patterns $0 1^n 0$ where $n \geq 1$ is not a power of 2.
This system does not have the specification property, since the patterns $0 1^{2^n + 1}$ and $0$ cannot be joined by a third pattern shorter than $2^n - 1$, but it is $0$-gluing, since all valid patterns of the form $0 w 0$ for $w \in \{0, 1\}^*$ can be concatenated freely. As for weaker specification notions, this system does not have non-uniform specification in the sense of \cite{Pa16} (for any gap function $f(n) < n$), and does not have almost specification with a constant mistake function $g(n)$.

For a group $G$ generated by a finite set $S$, the \emph{shadowing property} for a dynamical system $(G, X)$ was defined in \cite{OsTi14} as follows.
For all $\epsilon > 0$ there exists $\delta > 0$ such that for all functions $x : G \to X$ with $d_X(s \cdot x(g), x(s g)) < \delta$ for all $g \in G$ and $s \in S$, there exists $y \in X$ such that $d_X(x(g), g \cdot y) < \epsilon$ for all $g \in G$.
The function $x$ is called a \emph{$\delta$-pseudotrajectory} or \emph{pseudo-orbit}, and we say that the orbit of $y$ \emph{traces} the pseudo-orbit $x$. For this reason, shadowing is also known as the pseudo-orbit tracing property.
All $G$-subshifts of finite type have the shadowing property.
We now show that systems with the shadowing property are $0$-gluing.

\begin{proposition}
\label{prop:ShadowingIsGluing}
Let $(G, X)$ be a pointed dynamical system where $G$ is finitely generated.
If $(G, X)$ has the shadowing property, then it is $0$-gluing.
\end{proposition}

\begin{proof}
We need to show that for all $\epsilon > 0$ there exists $\delta' > 0$ and a finite set $E \subset G$ such that all $(\delta', E)$-designations have an $\epsilon$-realization. We pick $E = \{1\} \cup S \cup S^{-1}$. Let $\epsilon > 0$ be arbitrary, and let $\delta > 0$ be as in the definition of shadowing. Pick $\delta' > 0$ so that $d_X(x, y) < \delta'$ implies $d_X(g \cdot x, g \cdot y) < \delta/2$ for all $g \in E$. Let $(A_i, x_i)_{i \in I}$ be a $(\delta', E)$-designation. Since $(A_i, x_i)_{i \in I}$ is a $(\delta', E)$-designation, we have $d_X(g \cdot x_i, 0) < \delta'$ for all $i$ and $g \in \partial_E A_i$.

Define $x : G \to X$ by $x(g) = g \cdot x_i$ when $g \in A_i$. We claim that $x$ is a $\delta$-pseudo-trajectory. To see this, let $g \in G$ be arbitrary. Let $i \in I$ be such that $g \in A_i$. First, suppose that $s g \in A_i$. Then $d_X(s \cdot x(g), x(s g)) = d_X(s g \cdot x_i, s g \cdot x_i) = 0$. If on the other hand $s g \notin A_i$, then $s g \in A_j$ for some $j \neq i$, and $g \in \partial_E A_i$ and $s g \in \partial_E A_j$ (because $E$ is symmetric). Then $d_X(s \cdot x(g), x(s g)) = d_X(s g \cdot x_i, s g \cdot x_j)$. We have $d_X(s g \cdot x_j, 0) < \delta' < \delta/2$ because $s g \in \partial_E A_j$. We have $d_X(g \cdot x_i, 0) < \delta'$ because $g \in \partial A_i$, and thus $d_X(s g \cdot x_i, 0) = d_X(s g \cdot x_i, s \cdot 0) < \delta/2$ by the choice of $\delta'$. By the triangle inequality,
\[ d_X(s g \cdot x_i, sg \cdot x_j) \leq d_X(s g \cdot x_i, 0) + d_X(0, s g \cdot x_j) < \delta \]

By the shadowing property, there is a point $y$ such that $d_X(x(g), g \cdot y) < \epsilon$ for all $g \in G$. Then $d_X(g \cdot y, g \cdot x_i) < \epsilon$ for all $g \in A_i$ because by the definition of $x$, $x(g) = g \cdot x_i$ whenever $g \in A_i$. Thus $y$ is an $\epsilon$-realization of $(A_i, x_i)_{i \in I}$.
\end{proof}

The converse direction, that $0$-gluing implies the shadowing property, is false in general. For example, if $X \subset \{1, \ldots, k\}^\Z$ is any subshift, then $0^\Z \cup X$ is a pointed subshift which is $0$-gluing, but it has the shadowing property if and only if $X$ does.

\begin{lemma}
  \label{lem:GluingImage}
  Let $(G, X)$ and $(G, Y)$ be pointed dynamical systems and $\pi : X \to Y$ a surjective morphism with $\pi^{-1}(0_Y) = \{ 0_X \}$.
  If $(G, X)$ is $0$-gluing, then so is $(G, Y)$.
\end{lemma}

\begin{proof}
  Let $\epsilon > 0$ be arbitrary, and let $\eta > 0$ be such that $d_X(x, y) < \eta$ implies $d_Y(\pi(x), \pi(y)) < \epsilon$ for all $x, y \in X$.
  Let $\delta > 0$ and $E \subseteq G$ be given by the $0$-gluing property of $(G, X)$ for $\eta$.
  By compactness and the assumption $\pi^{-1}(0_Y) = \{ 0_X \}$, there exists $\theta > 0$ such that $d_Y(\pi(x), 0_Y) < \theta$ implies $d_X(x, 0_X) < \delta$ for $x \in X$.
  Let $\mathcal{A} = (A_i, y_i)_{i \in I}$ be a $(\theta, E)$-designation on $(G, Y)$.
  For each $i \in I$, choose a point $x_i \in \pi^{-1}(y_i)$, and denote $\mathcal{B} = (A_i, x_i)_{i \in I}$, which is a designation on $(G, X)$.
  For each $i \in I$ and $g \in \partial_E A_i$, we have $d_Y(g \cdot y_i, 0_Y) < \theta$, which implies $d_X(g \cdot x_i, 0_X) < \delta$, so $\mathcal{B}$ is a $(\delta, E)$-designation, and hence has an $\eta$-realization $x \in X$.
  For each $i \in I$ and $g \in A_i$, we have $d_X(g \cdot x_i, g \cdot x) < \eta$, and then $d_Y(g \cdot y_i, g \cdot \pi(x)) < \epsilon$.
  Thus $\pi(x)$ is an $\epsilon$-realization of $\mathcal{A}$.
\end{proof}

A pointed $G$-subshift $X$ is called \emph{$0$-to-$0$ sofic}, if there exists a $G$-subshift of finite type $Y$ and a $G$-block code $\pi : Y \to X$ with $\pi^{-1}(0_X) = \{ 0_Y \}$.
By Proposition~\ref{prop:ShadowingIsGluing} and Lemma~\ref{lem:GluingImage}, every $0$-to-$0$ sofic shift is $0$-gluing.
We will briefly return to this class of subshifts in Section~\ref{sec:BlockGluing}.

We will now prove some auxiliary lemmas about designations and realizations in expansive systems. In the lemmas, `$\epsilon$ is small enough' implies that it is a function of $G$, $X$ and $\Gamma$.

\begin{lemma}
\label{lem:UniqGlue}
Let $(G, X)$ be an expansive pointed dynamical system, and let $\Gamma$ be a finite set of self-homeomorphisms of $X$.
Let $\epsilon > 0$, let $\mathcal{A} = (A_i, x_i)_{i \in I}$ be a designation, and let $x, y \in X$ and $W = (\gamma_i \in \Gamma)_{i \in I}$ be such that both $(y, W)$ and $(z, W)$ are $(\epsilon, \Gamma)$-realizations of $\mathcal{A}$.
If $\epsilon$ is small enough, then $y = z$.
\end{lemma}

\begin{proof}
Let $\Delta > 0$ be the canonical expansivity constant of $X$, and suppose $\epsilon \leq \Delta / 2$.
For each $i \in I$ and $g \in A_i$, we now have
\[
d_X(g \cdot y, g \cdot z) \leq d_X(g \cdot y, \gamma_i (g \cdot x_i)) + d_X(\gamma_i (g \cdot x_i), g \cdot z) \leq \Delta
\]
Since the $A_i$ form a partition of $G$, this implies $y = z$.
\end{proof}

From now on, unless otherwise noted, we will assume that all realizations in expansive systems satisfy the conditions of Lemma~\ref{lem:UniqGlue} when the $\Gamma$-component $W$ is fixed.
This justifies the notation $R^W_\epsilon(\mathcal{A}) \in X$ for the $(\epsilon, \Gamma)$-realization with $W$ as the second component, and $R_\epsilon(\mathcal{A}) \in X$ for the $\epsilon$-realization, of a designation $\mathcal{A}$, provided that they exist.

\begin{lemma}
\label{lem:RegionsGlue}
Let $(G, X)$ be an expansive pointed dynamical system, and let $\Gamma$ be a finite set of self-homeomorphisms of $X$.
Let $\epsilon > 0$, and let $\mathcal{A} = (A_i, x_i)_{i \in I}$ and $\mathcal{B} = (B_i, x_i)_{i \in I}$ be two designations with the same points.
Assume that for each $i \in I$, we have $d_X(g \cdot x_i, 0_X) < \epsilon$ for all $g \in G \setminus (A_i \cap B_i)$.
If $\epsilon$ is small enough and the realizations $R^W_\epsilon(\mathcal{A})$ and $R^W_\epsilon(\mathcal{B})$ both exist, then they are equal.
\end{lemma}

\begin{proof}
Let $\Delta > 0$ be the canonical expansivity constant of $X$, and suppose $\epsilon \leq \Delta / 4$ is small enough that $d_X(x, 0_X) < \epsilon$ implies $d_X(\gamma(x), 0_X) \leq \Delta / 4$ for all $\gamma \in \Gamma$.
Denote $R^W_\epsilon(\mathcal{A}) = y$ and $R^W_\epsilon(\mathcal{B}) = z$, and let $g \in G$ be arbitrary.
If $g \in A_i \cap B_i$ for some $i \in I$, then $d_X(g \cdot y, g \cdot z) \leq d_X(g \cdot y, \gamma_i(g \cdot x_i)) + d_X(\gamma_i(g \cdot x_i), g \cdot z) \leq \Delta$.
Otherwise, we have $g \in A_i \cap B_j$ for some $i \neq j$, and $d_X(g \cdot y, g \cdot z) \leq d_X(g \cdot y, \gamma_i(g \cdot x_i)) + d_X(\gamma_i(g \cdot x_i), 0_X) + d_X(0_X, \gamma_j(g \cdot x_j)) + d_X(\gamma_j(g \cdot x_j), g \cdot z) \leq \Delta$.
This implies $y = z$.
\end{proof}

%


\begin{lemma}
\label{lem:FuncGlue}
Let $(G, X)$ be an expansive pointed dynamical system, let $f : X \to X$ be a morphism, and let $\Gamma$ be a finite set of self-homeomorphisms of $X$ that commute with $f$.
Let $\epsilon > 0$, let $\mathcal{A} = (A_i, x_i)_{i \in I}$ be a designation, and let $W = (\gamma_i \in \Gamma)_{i \in I}$.
If $\epsilon$ is small enough and $y = R^W_\epsilon(\mathcal{A})$ and $z = R^W_\epsilon((A_i, f(x_i))_{i \in I})$ both exist, then $z = f(y)$.
\end{lemma}

\begin{proof}
Let $0 < \epsilon < \delta$ be such that $d_X(x, x') < \epsilon$ implies $d_X(f(x), f(x')) < \delta$ for all $x, x' \in X$.
For each $i \in I$ and $g \in A_i$, we have $d_X(g \cdot y, \gamma_i (g \cdot x_i)) < \epsilon$, which implies $d_X(g \cdot f(y), \gamma_i (g \cdot f(x_i))) = d_X(f(g \cdot y), f(\gamma_i (g \cdot x_i))) < \delta$.
Then $(f(y), W)$ is a $(\delta, \Gamma)$-realization of $(A_i, f(x_i))_{i \in I}$.
For small enough $\delta$, Lemma~\ref{lem:UniqGlue} implies that the realization is unique, and since $(z, W)$ is also a $(\delta, \Gamma)$-realization of the same designation, we have $f(y) = z$.
\end{proof}

\section{Tiered Dynamical Systems}
\label{sec:Tiers}

In this section, we define the notion of tiered dynamical systems, which is used in the proof of our main result. We also extend the definitions of $0$-gluing, nilpotency and nil-rigidity into this context.

\begin{definition}
\label{def:Tiers}
A \emph{tiered dynamical system}, or TDS for short, is defined as a triple $\X = (X, (G_t, X_t)_{t \in \Di}, (\phi^{t'}_t)_{t \leq t' \in \Di})$, where
\begin{itemize}
\item $X$ is a compact metric space, called the \emph{ambient space}, containing a special point $0 \in X$,
\item $\Di$ is a 
directed set with least element $t_0 \in \Di$,
\item for each $t \in \Di$, $(G_t, X_t)$ is an expansive pointed dynamical system with $0 \in X_t \subseteq X$,
\item for each pair $t \leq t' \in \Di$ we have $X_t \subseteq X_{t'}$,
\item for all $t \leq t' \in \Di$, $\phi^{t'}_t : G_{t'} \to G_t$ is a finite-to-one surjective homomorphism with $\phi^t_t = \ID_{G_t}$ and $\phi^{t'}_t \circ \phi^{t''}_{t'} = \phi^{t''}_t$ for $t'' \geq t'$, and
\item for all $t \leq t'$, $x \in X_t$ and $g \in G_{t'}$ we have $\phi^{t'}_t(g) \cdot x = g \cdot x$.
\end{itemize}
We denote $G_{t_0} = G_0$, and call it the \emph{base group} of $\X$.
Similarly, $X_{t_0} = X_0$.
\end{definition}

The word \emph{tier} refers to either an element $t \in \Di$ of the directed set, the associated set $X_t$, or the system $(G_t, X_t)$.
The correct interpretation should always be clear from the context.
We note that the union of tiers $\bigcup_{t \in \Di} X_t$ may not be compact.
In most of our constructions, it is dense in the ambient space $X$, but even this is not guaranteed by the definition.

We usually denote by $\Delta_t$ the canonical expansivity constant (see \eqref{eq:Expansivity}) of the tier $(G_t, X_t)$.
We automatically have that $t \leq t'$ implies $\Delta_{t'} \leq \Delta_t$, since $X_{t'}$ contains $X_t$ and restricted to $X_t$, the action on $X_{t'}$ agrees with that on $X_t$ (through the homomorphism $\phi^{t'}_t$).

\begin{definition}
\label{def:TierEndo}
An \emph{evolution map}, or simply $\emph{evolution}$, of the system $\X$ in Definition~\ref{def:Tiers} is a continuous function $f : X \to X$ such that
for all $t \leq t' \in \Di$ and $x \in X_t$, the condition $f(x) \in X_{t'}$ implies $f(\phi^{t'}_t(g) \cdot x) = g \cdot f(x)$ for all $g \in G_{t'}$.
We say that a tier $t \in \Di$ is \emph{$f$-stabilizing}, if there exists $s(t) \geq t$ such that $f^k(X_t) \subseteq X_{s(t)}$ for all $k \in \N$.
The TDS $\X$ is $f$-stabilizing if each of its tiers is.
The quadruple $(X, (G_t, X_t)_{t \in \Di}, (\phi^{t'}_t)_{t \leq t'}, f)$, which we may denote by $(\X, f)$, is called an \emph{evolution of a tiered dynamical system}, or ETDS for short.


If $(G, X)$ is a pointed dynamical system, the single-tier system $(X, (G, X))$ is denoted by $\mathcal{U}(G, X)$.
If $f : X \to X$ is an endomorphism of $(G, X)$, we denote $\mathcal{U}(G, X, f) = (X, (G, X), f)$
\end{definition}

Of course, all endomorphisms of a single-tier system are evolution maps, which justifies our definition of $\mathcal{U}(G, X, f)$.
The main reason for our choice of the term ``evolution'' over ``endomorphism'' is that the set of evolution maps is in general not closed under composition.
It is even possible that $f \circ f$ fails to be an evolution map even though $f$ is.
Namely, if $(G, X)$ is any expansive dynamical system, take the single-tier system $\X = (X \cup Y, (G, X))$ with $Y$ a disjoint copy of $X$.
Then every continuous function $f$ that maps $X$ to $Y$ and vice versa is an evolution of $\X$, even though $(f \circ f)|_X$ can be any continuous function from $X$ to itself.
However, it can be shown that if two evolution maps preserve the union of tiers $\bigcup_{t \in \Di} X_t$ (in particular if each tier is stabilizing with respect to them), then their composition is an evolution map.
It would seem natural for an endomorphism of a tiered system to have this property.
We do not require it, as it would restrict our ability to carry out the constructions of Section~\ref{sec:Local}, and closure with respect to composition is not essential for the purposes of this article.
The second reason for the term ``evolution'' is that we see the iteration of an endomorphism as the space $X$ evolving in time, while the group action represents a spatial dimension, which is the standard interpretation in the context of cellular automata.

We introduce versions of $0$-gluing and dense homoclinic points for tiered systems.


\begin{definition}
\label{def:TierProperties}
  Let $\X = (X, (G_t, X_t)_{t \in \Di}, (\phi^{t'}_t)_{t \leq t'})$ be a TDS.
  We say that $t \in \Di$ is \emph{weakly $0$-gluing} if there exists $w(t) \geq t$ such that for all $\epsilon > 0$, there exist $\delta > 0$ and $E > 0$ finite such that every $(\delta, E)$-designation $(A_i, x_i)_{i \in I}$ on $(G_{w(t)}, X_{w(t)})$ with $x_i \in X_t$ for each $i \in I$ has an $\epsilon$-realization in $X_{w(t)}$.
  
  Let $f : X \to X$ be an evolution of $\X$.
  For each $t \in \Di$, let $\Lambda^t_f$ be the set of homeomorphisms $\gamma : X_t \to X_t$ that commute with $f$ and satisfy $\gamma(X_r) = X_r$ for all $r \leq t$ and $\gamma(0) = 0$.
  We say $t \in \Di$ is \emph{weakly $f$-variably $0$-gluing}, if there exists $w(t) \geq t$ and a finite set $\Gamma \subseteq \Lambda^{w(t)}_f$ such that for all $\epsilon > 0$, there exist $\delta > 0$ and $E > 0$ finite such that every $(\delta, E)$-designation $(A_i, x_i)_{i \in I}$ on $(G_{w(t)}, X_{w(t)})$ with $x_i \in X_t$ for each $i \in I$ has an $(\epsilon, \Gamma)$-realization in $X_{w(t)}$.
  
  We say that the system $\X$ is weakly ($f$-variably) $0$-gluing, if each of its tiers is.

  We say that $\X$ has \emph{weakly dense homoclinic points}, if for all $t \in \Di$, $x \in X_t$ and $\epsilon > 0$, there exists $t' \geq t$ and $y \in X_{t'}$ such that $d_X(x, y) < \epsilon$ and $y$ is homoclinic in $(G_{t'}, X_{t'})$.
\end{definition}

In general, we treat $s$ and $w$ as nondecreasing partial functions from the set of tiers $\Di$ to itself.
Statements like ``the tier $w(s(t))$ exists'' are interpreted as the tier $t$ being $f$-stabilizing and $s(t)$ having one of the weak $0$-gluing properties (the exact flavor should be clear from the context).

Definitions~\ref{def:Tiers},~\ref{def:TierEndo} and~\ref{def:TierProperties} are somewhat technical, but there is a motivation behind them.
Namely, we wish to define a class of objects that contains all expansive pointed dynamical systems, and allows us to extract certain subsystems that consist of periodic and homoclinic points.
The main issue is that depending on the exact definitions, the subsystems of periodic and homoclinic points tend to lose the important properties of $0$-gluing or dense homoclinic points, are not invariant under a given evolution map, or are nonexpansive or even noncompact.
Tiered systems are a solution to this problem: we can glue points together, extract homoclinic points and apply evolution maps, but it may require passing to a higher tier.
Examples of these constructions are given in Section~\ref{sec:ConstrExamples}.

\begin{definition}
\label{def:Shadows}
Let $\X = (X, (G_t, X_t)_{t \in \Di})$ be a TDS.
For $t \in \Di$, the \emph{$t$-shadow} of a point $x \in X_t$ is defined as $\Theta_t(x) = \Theta(\Delta_t, x) \subseteq G_t$, where $\Delta_t > 0$ is the canonical expansivity constant of $(G_t, X_t)$.
\end{definition}


\begin{remark}
\label{rem:D}
  Recall that $t \leq t' \in \Di$ implies $\Delta_{t'} \leq \Delta_t$.
  Then there exists a finite set $1_{G_t} \in D^{t'}_t \subseteq G_t$ such that $d_X(g^{-1} \cdot x, 0) \leq \Delta_t$ for all $g \in D^{t'}_t$ implies $d_X(x, 0) \leq \Delta_{t'}$ for $x \in X_t$.
  This means that $\Theta_t(x) \subseteq \phi^{t'}_t(\Theta_{t'}(x)) \subseteq D^{t'}_t \cdot \Theta_t(x)$.
  In particular, if one of $\Theta_t(x)$ or $\Theta_{t'}(x)$ is finite, then so is the other.
\end{remark}

We now define a few convenient classes of ETDSs.

\begin{definition}
\label{def:Classes}
Let $(\X, f)$ be an ETDS.
\begin{itemize}
\item If the TDS $\X$ is $f$-stabilizing and weakly $f$-variably $0$-gluing, then $(\X, f) \in \mathfrak{C}_1$.
\item If $\X$ is $f$-stabilizing, weakly $0$-gluing and has weakly dense homoclinic points, then $(\X, f) \in \mathfrak{C}_2$.
\item If $(\X, f) \in \mathfrak{C}_2$ and the groups of $\X$ are residually finite, then $(\X, f) \in \mathfrak{C}_3$.
\item If $(\X, f) \in \mathfrak{C}_2$ and contains a single tier, then $(\X, f) \in \mathfrak{C}_4$.
\end{itemize}
\end{definition}

Note that $\mathfrak{C}_3 \subset \mathfrak{C}_2 \subset \mathfrak{C}_1$ and $\mathfrak{C}_4 \subset \mathfrak{C}_2$.
Our primary interests are the class $\mathfrak{C}_4$, which consists of ordinary dynamical systems, and the class $\mathfrak{C}_1$, which contains systems with weaker gluing properties.
The other classes are used as auxiliary objects in the course of proving our main theorems.

We also extend the different definitions of nilpotency and nil-rigidity for evolution maps of tiered systems.

\begin{definition}
Let $(X, (G_t, X_t)_{t \in \Di}, f)$ be an ETDS.
Let $Y \subset X$ be a set of points.
We say that $f$ is
\begin{itemize}
\item \emph{nilpotent on $Y$}, if there exists $n \in \N$ such that $f^n(x) = 0$ for all $x \in Y$.
\item \emph{weakly nilpotent on $Y$}, if for all $x \in Y$ there exists $n \in \N$ such that $f^n(x) = 0$.
\item \emph{asymptotically nilpotent on $Y$}, if $\lim_{n \longrightarrow \infty} f^n(x) = 0$ for all $x \in Y$.
\item \emph{uniformly asymptotically nilpotent on $Y$}, if it is asymptotically nilpotent on $Y$ and the convergence toward $0$ is uniform.
\end{itemize}
If $Y$ is not mentioned, it is assumed to be the ambient space $X$.
\end{definition}

Note that an endomorphism $f$ of a pointed dynamical system $(G, X)$ is nilpotent if and only if $f$ is nilpotent on $X$ with respect to the tiered system $\mathcal{U}(G, X)$, and similarly for the other three properties.
Our goal is to prove that on a large class of groups $G$, all sufficiently well-behaved asymptotically nilpotent evolution maps of tiered $G$-systems exhibit uniform convergence.

\begin{definition}
  \label{def:NRtiergroup}  
  Let $\mathfrak{C}$ be a class of ETDSs, and let $\X = (X, (G_t, X_t)_{t \in \Di})$ be a tiered dynamical system.
  If every asymptotically nilpotent evolution map $f : X \to X$ with $(\X, f) \in \mathfrak{C}$ is nilpotent on each tier $X_t$, then $\X$ is \emph{nil-rigid on $\mathfrak{C}$}.
  A group $G$ is \emph{nil-rigid on $\mathfrak{C}$}, if every tiered system $\X$ with base group $G$ is nil-rigid on $\mathfrak{C}$.
\end{definition}

We note that each asymptotically nilpotent evolution map $f$ is only required to be nilpotent on each tier separately, and may not be nilpotent on the union of all tiers (although it must be weakly nilpotent on this set).
Also, the nil-rigidity of a single-tier system $\mathcal{U}(G, X)$ on any of the classes $\mathfrak{C}_i$ is equivalent to the nil-rigidity of $(G, X)$ according to Definition~\ref{def:NRsystem}.

%

We show that nilpotency and weak nilpotency are equivalent when $Y$ is a subsystem of one of the tiers. See \cite{MeSa17} for a similar result for semigroup actions by endomorphisms.

\begin{proposition}
\label{prop:WeakNilp}
Let $(X, (G_t, X_t)_{t \in \Di}, f)$ be an ETDS, and let $Y \subseteq X_t$ be a $G_t$-invariant closed subset of some $f$-stabilizing tier $X_t$.
If $f$ is weakly nilpotent on $Y$, then it is nilpotent on $Y$.
\end{proposition}

\begin{proof}
Let $\Delta > 0$ be an expansivity constant for $(G_{s(t)}, X_{s(t)})$.
Denote by $Z = \bigcap_{k \geq 0} f^{-k}(X_{s(t)})$ the set of points whose forward orbit does not leave $X_{s(t)}$.
Note that $Y \subseteq Z$.
By Lemma~\ref{lem:Nhood}, there exists a finite set $E \subseteq G_{s(t)}$ with the following property: if $x \in Z$ is such that $d_X(f(x), 0) \geq \Delta$, then $d_X(g \cdot x, 0) > \Delta$ for some $g \in E$.

Suppose that $f$ is not nilpotent on $Y$.
Then for all $k \in \N$, there exists $x_k \in Y$ such that $d_X(f^k(x_k), 0) > \Delta$.
For all $0 \leq m \leq k$ there exists $g^k_m \in G_{s(t)}$ such that $d_X(g^k_m \cdot f^m(x_k), 0) > \Delta$ and $g^k_m (g^k_{m+1})^{-1} \in E$, and $g^k_k = 1_{G_{s(t)}}$.
Let $y_k = g^k_0 \cdot x_k \in Y$, and let $y \in Y$ be a limit point of the $y_k$.
Then there is a sequence $(g_m)_{m \in \N}$ of elements of $G_{s(t)}$ such that $g_0 = 1_{G_{s(t)}}$ and $g_m g_{m+1}^{-1} \in E$ and $d_X(g_m \cdot f^m(y), 0) \geq \Delta$ for all $m \in \N$.
In particular, this implies $f^m(y) \neq 0$.
Thus $f$ is not weakly nilpotent on $Y$.
\end{proof}

\section{Examples of constructions of tiered systems}
\label{sec:ConstrExamples}

In this section we show how tiered systems arise from normal systems when the acting group changes. The examples are meant to show why the different aspects of tiered systems are needed, but also illustrate the simple symbolic dynamical ideas behind the rather technical proofs later. The first is the $\mathrm{Fix}$-construction, which constructs a tiered system corresponding to the periodic points of the system along a normal subgroup, elaborated in Section~\ref{sec:Periods}. The second is the $\mathrm{Fin}$-construction, which constructs a tiered system corresponding to the points whose shadow is finite when projected to a quotient group. It is formally introduced in Section~\ref{sec:Finites}. The third is the the $\mathrm{Loc}$-construction, which extracts finitely generated subgroups of the acting group, and is defined in Section~\ref{sec:Local}.

\begin{example}[Example of the $\mathrm{Fix}$-construction]
\label{ex:Fix}
Let $X'$ be the set of homoclinic configurations in $\{0, 1\}^{\Z^2}$ where all connected components of $1$s are filled squares, and let $X$ be its topological closure, where we add degenerate squares -- quadrants, half-planes and the all-$1$ configuration. This subshift $X$ is sofic (and we could also use its SFT cover in the proof, by constructing it suitably). The dynamical system $(\Z^2, X)$ is $0$-gluing and has dense homoclinic points.

Let $p \geq 1$ be arbitrary, and consider the subshift $X_p = \{x \in X \;|\; (0,p) \cdot x = x\}$ of configurations of $X$ with vertical period $p$. With the notation introduced after Lemma~\ref{lem:BoundedShadow}, we have $X_p = \fix{\langle (0,p) \rangle}{X}$, and the quotient group $\Z^2 / \langle (0,p) \rangle \simeq \Z \times \Z_p$ acts on this set. The system $(\Z \times \Z_p, X_p)$ is $0$-gluing, but does not have dense homoclinic points. Namely, the all-$1$ point is in $X_p$, but any close enough $(0,p)$-periodic point contains a column full of $1$s, and thus infinitely many columns full of $1$s. We now construct a tiered system consisting of all $X_p$, which is both weakly $0$-gluing and has weakly dense homoclinic points. This is a special case of Corollary~\ref{cor:PeriodicC2}.

Let $\Di = \{1, 2, 3, \ldots\}$ with the partial order given by divisibility. For $p \in \Di$, let $G_p = \Z \times \Z_p$ with maps $\phi^{p'}_p(m, n + p' \Z) = (m, n + p \Z)$ for $p \vert p'$, and $X_p$ as above. The tiered system
\[ \X = (X, (G_p, X_p)_{p \in \Di}, (\phi^{p'}_p)_{p \vert p'}) \]
is weakly $0$-gluing with $w : \Di \to \Di$ the identity map.
We claim that it also has weakly dense homoclinic points.
Since $X$ has dense homoclinic points, given any $p \in \Di$ and $x \in X_p$ we can approximate $x$ by a homoclinic point $y \in X$. If we pick a large enough vertical period $p'$ with $p \vert p'$, then the point $z \in \{0, 1\}^{\Z^2}$ defined by
\[
z_{(i, j)} = \begin{cases}
  1, & \text{if $y_{(i, j + k p')} = 1$ for some $k \in \Z$,} \\
  0, & \text{otherwise}
\end{cases}
\]
is in $X_{p'}$, is homoclinic for the action of $\Z \times \Z_{p'}$, and can be made arbitrarily close to $x$.
\tqed
\end{example}



\begin{example}[Example of the $\mathrm{Fin}$-construction]
Let $X$ be the $\Z^2$-SFT over the alphabet $\{0, {\rig}, {\upp}, {\lef}, {\dow}\}$ consisting of configurations $x$ such that the directed graph with vertex set $\{\vec v \in \Z^2 \;|\; x_{\vec v} \neq 0\}$ and edge set $\{(\vec u, \vec v) \;|\; \vec u + x_{\vec u} = \vec v\}$ is a disjoint union of paths, where each element of $\{{\rig}, {\upp}, {\lef}, {\dow}\}$ is interpreted as an element of $\Z^2$ in the obvious way (e.g. ${\upp} = (0,1)$). This subshift is an SFT, so it has the shadowing property, and by Proposition~\ref{prop:ShadowingIsGluing} it is $0$-gluing.
It also has dense homoclinic points.

Let $t \in \N$ and denote $X_t = \{x \in X \;|\; \forall |m| > t, n \in \Z : x_{(m,n)} = 0\}$. In other words, $X_t$ consists of those configurations whose shadow -- the set of nonzero coordinates -- is contained in $[-t, t] \times \Z$. The subgroup $\{0\} \times \Z \simeq \Z$ acts expansively on each $X_t$, and the system $(\Z, X_t)$ is $0$-gluing, but does not have dense homoclinic points. Namely, if $x \in X_t$ is the point containing ${\upp}$s in the central column and $0$ elsewhere else, then any homoclinic point approximating $x$ must complete the path to a cycle. The completion must go through the central row, and thus the point is not in $X_t$ if it is close enough to $x$.

Consider now the tiered system
\[ \X = (X, (G_t, X_t)_{t \in \Di}, (\phi^{t'}_t)_{t \leq t'}) \]
where $\Di = \N$, $t \leq t'$ is the usual ordering of $\N$, $G_t = \{0\} \times \Z$ for all $t \in \Di$ and $\phi_t^{t'} = \ID$.
This system has weakly dense homoclinic points because $X$ does (every homoclinic point of $X$ is on some tier $X_t$), and is weakly $0$-gluing since every tier is $0$-gluing.
\tqed
\end{example}

\begin{example}[Example of the $\mathrm{Loc}$-construction]
\label{ex:Loc}
Let $G = \Z^\omega$ be the direct sum of a countably infinite number of copies of $\Z$ with generators $\vec e_1, \vec e_2, \ldots$.
Let $\Sigma = \N \cup \{\infty\}$ be the one-point compactification of $\N$, and for $t \in \N$, denote $\Sigma_t = \{0, \ldots, t\}$.
Define a set of configurations $X \subset \Sigma^G$ by the following rules.
\begin{enumerate}
\item For all $x \in X$, $1 < n < \infty$ and $\vec v \in G$, if $x_{\vec v} = n$, then $x_{\vec v + \vec e_n} > 0$.
\item For all $x \in X$, $n > 0$ and $\vec v \in G$, if $x_{\vec v + k e_1} = 1$ for all $1 \leq k \leq n$ and $x_{\vec v} \neq 1$, then $x_{\vec v} \geq n$.
\end{enumerate}
Denote $X_t = X \cap \Sigma_t^G$, so that the $X_t$ form an increasing chain of $G$-subshifts on larger and larger finite alphabets.
Then $\X = (X, (G, X_t)_{t \geq 1}, (\phi^{t'}_t)_{t \leq t'})$ is a tiered dynamical system, with each $\phi^{t'}_t$ being the identity map.
It can be verified that each tier $(G, X_t)$ is $0$-gluing, so the system as a whole is weakly $0$-gluing.
It also has weakly dense homoclinic points, since $X$ has dense homoclinic points and each homoclinic point that does not contain the symbol $\infty$ is found on some tier.

The idea of the $\mathrm{Loc}$-construction is to extract a subsystem whose acting group is finitely generated and which contains a specified homoclinic point.
Fix $t \geq 1$ and a homoclinic point $x \in X_t$, and let $H = \langle \{ \vec v \in G \;|\; x_{\vec v} > 0 \} \rangle$.
For $r \geq 1$, consider $Y_r = \{ y \in X_r \;|\; \forall \vec v \notin H : y_{\vec v} = 0 \}$.
Then $H$ acts expansively on each $Y_r$, and we may form the tiered system $(X, (H, Y_r)_{r \geq 1})$.
Also, we have $Y_r \subset \Sigma_t^G$ for all $r$ by rule 1, which implies $Y_r = Y_t$ for all $r \geq t$, so we only have finitely many distinct tiers.
Thus it makes sense to restrict our attention to the single-tier system $(H, Y_t)$, which contains the point $x$.

It can be shown that $(H, Y_t)$ is also $0$-gluing, but it does not necessarily have dense homoclinic points.
For example, if $x_{\vec e_1} > 0$, then $\langle \vec e_1 \rangle \leq H$, so in particular $Y_t$ contains the point $y \in X_1$ with $y_{k \vec e_1} = 1$ for all $k \in \Z$ and $y_{\vec v} = 0$ for all other $\vec v \in G$.
Since $x$ is homoclinic, $H$ is finitely generated, hence $H \leq \langle \vec e_1, \ldots, \vec e_m \rangle$ for some $m \geq 1$.
Any $H$-homoclinic point $z \in X$ sufficiently close to $y$ contains a finite run of $1$s on the first component of $G$ of length more than $m$, and thus satisfies $z_{k \vec e_1} > m$ for some $k \in \Z$ by rule 2, which implies $z \notin Y_t$.
\tqed
\end{example}




\section{Nil-Rigidity and Finiteness Properties}
\label{sec:ClosureProps}

The classes of nil-rigid groups enjoy several closure properties, some of which depend on the class of tiered systems under consideration.
In this section, we prove the nil-rigidity of finite groups, and show that virtually nil-rigid groups, finite-by-nil-rigid groups and subgroups of nil-rigid groups (on $\mathfrak{C}_4$) are themselves nil-rigid.

\begin{proposition}
  \label{prop:FiniteGroups}
  Let $(X, (G_t, X_t)_{t \in \Di}, f)$ be an ETDS with $G_0 = G$ finite.
  If $f$ is asymptotically nilpotent and $t \in \Di$ is $f$-stabilizing, then $f$ is nilpotent on $X_t$.
\end{proposition}

\begin{proof}
  Since the group $G_{s(t)}$ is finite and acts expansively on $X_{s(t)}$, the topology of $X_{s(t)}$ is discrete, and as it is also compact, it is finite.
  Then $f$ is clearly nilpotent on $X_t$.
\end{proof}

\begin{corollary}
  Every finite group is nil-rigid on the class of stabilizing ETDSs.
\end{corollary}

\begin{lemma}
\label{lem:VirtualProp}
  Let $(\X, f)$ be an ETDS with $\X = (X, (G_t, X_t)_{t \in \Di})$.
  For each $t \in \Di$, let $H_t \leq G_t$ be a subgroup of finite index with $\phi^{t'}_t(H_{t'}) = H_t$ whenever $t \leq t'$.
  Then $\mathcal{Y} = (X, (H_t, X_t)_{t \in \Di})$ is a TDS and $f$ is its evolution.
  If $\X$ or one of its tiers is weakly $f$-variably $0$-gluing, weakly $0$-gluing and/or has weakly dense homoclinic points, then $\mathcal{Y}$ or the corresponding tier has the same properties.
\end{lemma}

\begin{proof}
  For $t \in \Di$, let $1_{G_t} \in L_t \subset G_t$ be a set of representatives for the left cosets of $H_t$ in $G_t$.
  Fix now a single tier $t \in \Di$, and let $\Delta > 0$ be the canonical expansivity constant for $(G_t, X_t)$. 
  Let $\epsilon > 0$ be such that $d_X(x, x') \leq \epsilon$ implies $d_X(k \cdot x, k \cdot x') \leq \Delta$ for all $k \in L_t$.
  Now, if $d_X(h \cdot x, h \cdot x') \leq \epsilon$ for all $h \in H_t$, then for each $g \in G_t$ we have $d_X(g \cdot x, g \cdot x') = d_X(k h \cdot x, k h \cdot x') \leq \Delta$, where $h \in H_t$ and $k \in L_t$ are such that $k h = g$.
  Thus $\epsilon$ is an expansivity constant for $(H_t, X_t)$.
  Since $t \in \Di$ was arbitrary, each system $(H_t, X_t)$ is expansive.
  
  The property $\phi^{t'}_t(H_{t'}) = H_t$ guarantees that the restriction of each $\phi^{t'}_t$ to $H_{t'}$ is a surjective homomorphism onto $H_t$.
  The other conditions of Definition~\ref{def:Tiers} are clear, so $\mathcal{Y}$ is a tiered dynamical system and $f$ is its evolution.
    
    Let $t \in \Di$, and suppose that $x \in X_t$ is homoclinic for the action of $G_t$.
    Then for all $\epsilon > 0$, there exists a finite set $E \subseteq G_t$ such that $d_X(g \cdot x, 0) < \epsilon$ for all $g \in G_t \setminus E$.
    We can then choose $E' = E \cap H_t$ as a witness for the homoclinicity of $x$ under $H_t$.
    Conversely, suppose that $x$ is homoclinic under the action of $H_t$.
    Given $\epsilon > 0$, let $\delta > 0$ be such that $d_X(y, 0_X) < \delta$ implies $d_X(k \cdot y, 0_X) < \epsilon$ for all $k \in L_t$.
    Let $E \subseteq H_t$ be a finite set with $d_X(h \cdot x, 0_X) < \delta$ for all $h \in H_t \setminus E$.
    Then $d_X(g \cdot x, 0_X) < \epsilon$ for all $g \in G_t \setminus L_t^{-1} E$.
    This shows that $x$ is homoclinic under the action of $G_t$ as well.
    In particular, if $\X$ has weakly dense homoclinic points, then so does $\mathcal{Y}$.
    
    Suppose that a tier $t \in \Di$ of $\X$ is weakly $f$-variably $0$-gluing, and let $w(t) \geq t$ and $\Gamma \subseteq \Lambda^{w(t)}_f$ be the higher tier and set of homeomorphisms given by this property.
    Let $\epsilon > 0$, and let $\delta > 0$ and $E \subseteq G_{w(t)}$ finite be given by the weak $f$-variable $0$-gluing property of $t$ for it.
    Let $\eta > 0$ be such that $d_X(x, 0) < \eta$ implies $d_X(k \cdot x, 0) < \delta$ for all $k \in L_{w(t)}$ and $x \in X_{w(t)}$, and define $E' = H_{w(t)} \cap L_{w(t)}^{-1} E L_{w(t)}$, which is also finite.
    
    Let $\mathcal{A} = (A_i, x_i)_{i \in I}$ be an $(\eta, E')$-designation on the system $(H_{w(t)}, X_{w(t)})$.
    Then the collection $(L_{w(t)} A_i)_{i \in I}$ is a partition of the group $G_{w(t)}$.
    We claim that $(L_{w(t)} A_i, x_i)_{i \in I}$ is a $(\delta, E)$-designation on $(G_{w(t)}, X_{w(t)})$.
    Let $i \in I$ and $k h \in \partial_E (L_{w(t)} A_i)$, where $k \in L_{w(t)}$ and $h \in A_i$.
    Then there exists $e \in E$ with $e k h \notin L_{w(t)} A_i$.
    By the definition of $E'$, there exists $q \in L_{w(t)}$ with $q^{-1} e k \in E'$.
    Then $q q^{-1} e k h \notin L_{w(t)} A_i$, which implies $q^{-1} e k h \notin A_i$, or $h \in \partial_{E'} (A_i)$.
    Since $\mathcal{A}$ is an $(\eta, E')$-designation, we have $d_X(h \cdot x_i, 0) < \eta$.
    It follows that $d_X(k h \cdot x_i, 0) < \delta$, and thus $(L_{w(t)} A_i, x_i)_{i \in I}$ is a $(\delta, E)$-designation on $(G_{w(t)}, X_{w(t)})$.
    
    
    By the definition of $\delta$ and $E$, the designation has an $(\epsilon, \Gamma)$-realization $(y, W)$, where $y \in X_{w(t)}$ and $W = (\gamma_i \in \Gamma)_{i \in I}$.
    Note that since $1_{G_{w(t)}} \in L_{w(t)}$, for each $i \in I$ we have $A_i \subseteq L_{w(t)} A_i$.
    This implies $d_X(h \cdot y, \gamma_i(h \cdot x_i)) < \epsilon$ for all $h \in A_i$.
    Thus $(y, W)$ is also an $(\epsilon, \Gamma)$-realization of $\mathcal{A}$.
    This means that $t$ is weakly $f$-variably $0$-gluing in $\mathcal{Y}$ as well, with the same choice of $w(t)$ and $\Gamma$.
  
    If $t$ is weakly $0$-gluing in $\X$, we can repeat the above argument with $\Gamma = \{\ID_{X_t}\}$, showing that it is weakly $0$-gluing also in $\mathcal{Y}$.
\end{proof}

\begin{lemma}
\label{lem:virtual}
  Let $\mathfrak{C}$ be one of the classes $\mathfrak{C}_1$, $\mathfrak{C}_2$, $\mathfrak{C}_3$ or $\mathfrak{C}_4$.
  If a group $G$ is virtually nil-rigid on $\mathfrak{C}$ (meaning it has a finite-index subgroup with this property), then it is nil-rigid on $\mathfrak{C}$.
\end{lemma}

\begin{proof}
  Let $H \leq G$ be a subgroup of finite index that is nil-rigid on $\mathfrak{C}$.
  Let $(\X, f) \in \mathfrak{C}$, where $\X = (X, (G_t, X_t)_{t \in \Di})$ with $G_0 = G$ and $f$ is asymptotically nilpotent.
  For $t \in \Di$, define $H_t = (\phi^t_{t_0})^{-1}(H)$, which is a subgroup of $G_t$ of finite index, and residually finite whenever $G_t$ is.
  Denote $\mathcal{Y} = (X, (H_t, X_t)_{t \in \Di})$.
  Then $(\mathcal{Y}, f) \in \mathfrak{C}$ by Lemma~\ref{lem:VirtualProp}.
  Since $H$ is nil-rigid on $\mathfrak{C}$, the map $f$ is nilpotent on each tier of $\mathcal{Y}$, and thus on each tier of $\X$.
\end{proof}

Lemma~\ref{lem:virtual} is the main reason for the definition of the `$f$-variable' flavor of the $0$-gluing property.
Namely, we would like to prove, as we do in Section~\ref{sec:Z}, that $\Z$ is nil-rigid for expansive variably $0$-gluing tiered systems, since this class includes many well-known subshifts like the even shift (see Example~\ref{ex:EvenShift} in Section~\ref{sec:FinalExamples}).
Once this result has been established, we would like to extend it to all groups that are virtually $\Z$.
Expansive actions of virtually $\Z$ groups include one-dimensional subshifts with local symmetries, and the endomorphisms of such systems are cellular automata that respect those symmetries; see Example~\ref{ex:VirtuallyZ} in Section~\ref{sec:FinalExamples}.
Unfortunately, the proof that we give for Lemma~\ref{lem:VirtualProp} is not valid for expansive variably $0$-gluing systems, since a translation of a group $G_t$ is not necessarily a translation of a given finite-index subgroup $H_t$.

\begin{lemma}
\label{lem:FinExt}
  Let $\mathfrak{C}$ be one of the classes $\mathfrak{C}_1$, $\mathfrak{C}_2$ or $\mathfrak{C}_3$, and let $G$ and $H$ be groups.
  If $G$ is nil-rigid on $\mathfrak{C}$ and there is a finite-to-one surjective homomorphism $\psi : H \to G$, then $H$ is nil-rigid on $\mathfrak{C}$.
\end{lemma}

\begin{proof}
  Let $(\X, f) \in \mathfrak{C}$, where $\X = (X, (H_t, X_t)_{t \in \Di})$ is a TDS with $H_0 = H$ and $f$ is asymptotically nilpotent.
  Let $r \notin \Di$ be a new element, and let $\VDi = \Di \cup \{r\}$ be a new directed set with $r < t$ for each $t \in \Di$.
  Define $H_r = G$ and $X_r = \{0_X\}$, and let the system $(H_r, X_r)$ have the trivial dynamics.
  Then $X_r \subseteq X_t$ for all $t \in \Di$.
  Define also $\phi^t_r = \psi \circ \phi^t_{t_0}$ for $t \in \Di$.
  Then $\mathcal{Y} = (X, (H_t, X_t)_{t \in \VDi})$ is a TDS and $f$ is its evolution.
  It is easy to check that $(\mathcal{Y}, f) \in \mathfrak{C}$, and since the bottom group is $H_r = G$, $f$ is nilpotent on each tier of $\mathcal{Y}$.
  Then $f$ is also nilpotent on each tier of $\X$.
\end{proof}

Finally, we show that nil-rigidity on single-tier systems is inherited by subgroups using a standard coset construction.
The countability of $G/H$ is a technical restriction that could be lifted if we considered uniform spaces instead of metric spaces.
We are not able to easily extend the result beyond $\mathfrak{C}_4$ due to the larger groups $H_t$, which should have counterparts in the new system with base group $G$.

\begin{lemma}
\label{lem:Subgroups}
  Let $H \leq G$ be groups such that the number of cosets of $H$ in $G$ is at most countable.
  If $G$ is nil-rigid on $\mathfrak{C}_4$, then so is $H$.
\end{lemma}

\begin{proof}
  Let $\mathcal{U}(H, X, f) \in \mathfrak{C}_4$ be an asymptotically nilpotent ETDS.
  Choose a set of representatives $L = \{k_0, k_1, k_2, \ldots\} \subset G$ for the left cosets of $H$ with $k_0 = 1_G$.
  We assume that $d_X(x, y) < 1$ for all $x, y \in X$, using the metric $d_X(x, y) / (1 + d_X(x, y))$ in place of $d_X$ if this is not the case.
  Let $Y = X^{G/H}$ be the $G/H$-fold direct product of $X$, whose points are denoted $x = (x_{g H})_{g H \in G/H}$ or $x = (x_{k H})_{k \in L}$, equipped with the metric
  \[
    d_Y(x, y) = \sum_{k_n \in L} 2^{-n} d_X(x_{k_n H}, y_{k_n H})
  \]
  This metric defines the product topology on $Y$; note that $d_Y(x, y) \leq d_X(x_H, y_H)$.
  Let $\rho : H \times X \to X$ be the action of $H$ on $X$.
  We define an action $\tilde \rho : G \times Y \to Y$ by $\tilde \rho(g, x)_{q H} = \rho(h, x_{k H})$ for $g \in G$, $x \in Y$ and $q \in L$, where $h \in H$ and $k \in L$ are such that $g^{-1} q = k h^{-1}$.
  It can be checked that this indeed defines a left action, and we denote it by ${\cdot}$ as we did with $\rho$.
  In particular, we have $(h k^{-1} \cdot x)_H = h \cdot x_{k H}$ for all $h \in H$, $k \in L$ and $x \in Y$.
  
  Define a function $\tilde f : Y \to Y$ by $\tilde f ((x_{k H})_{k \in L}) = (f(x_{k H}))_{k \in L}$.
  Then $(G, Y)$ is a dynamical system and $\tilde f$ is its asymptotically nilpotent endomorphism.
  Any expansivity constant of $(H, X)$ is also an expansivity constant of $(G, Y)$.
  
  We show that $(G, Y)$ is $0$-gluing.
  Let $\epsilon > 0$ be arbitrary, let $N \geq 0$ be such that $\sum_{n > N} 2^{-n} \leq \epsilon/4$, and define $F = \{k_0, k_0^{-1}, \ldots, k_N, k_N^{-1}\}$.
  Let $0 < \delta < \epsilon/4$ and $1_H \in E \subseteq H$ be given for $\epsilon / 8$ by the $0$-gluing property of $(H, X)$.
  We may assume $E$ is symmetric, that is, $E = E^{-1}$, and that $\delta$ is small enough that $d_Y(x, 0_Y) < \delta$ implies $d_Y(k_n^{-1} \cdot x, 0_Y) < \epsilon / 8$ for all $n \leq N$.
  Let $\mathcal{A} = (A_j, x_j)_{j \in I}$ be a $(\delta, E \cup F)$-designation in $(G, Y)$.
  For $n \in \N$, denote $B^n_j = A_j k_n \cap H$, so that $(B^n_j)_{j \in I}$ forms a partition of $H$.
  We claim that $(B^n_j, (x_j)_{k_n H})_{j \in I}$ is a $(\delta, E)$-designation in $(H, X)$.
  Let $j \in I$ and $h \in \partial_E B^n_j$.
  Then there exists $e \in E$ with $e h \notin B^n_j$.
  Since $e h \in H$ and $h k_n^{-1} \in A_j$, this implies $h k_n^{-1} \in \partial_E A_j \subseteq \partial_{E \cup F} A_j$.
  Then we have $d_Y(h k_n^{-1} \cdot x_j, 0_Y) < \delta$, which implies $d_X(h \cdot (x_j)_{k_n H}, 0_X) = d_X((h k_n^{-1} \cdot x_j)_H, 0_X) < \delta$.
  
  Since $(B^n_j, (x_j)_{k_n H})_{j \in I}$ is a $(\delta, E)$-designation on the system $(H, X)$, it has an $\epsilon / 8$-realization $y^n \in X$.
  Define $y = (y^n)_{k_n \in L} \in Y$; we claim that $y$ is an $\epsilon$-realization of $\mathcal{A}$.
  Let $j \in I$ and $g \in A_j$.
  Let $n \leq N$, and let $k_{m(n)} \in L$ and $h_n \in H$ be such that $g^{-1} k_n = k_{m(n)} h_n^{-1}$.
  Suppose first that $k_n^{-1} g \in A_j$.
  Then $h_n = k_n^{-1} g k_{m(n)} \in B^{m(n)}_j$, and we have $d_X((g \cdot x_j)_{k_n H}, (g \cdot y)_{k_n H}) = d_X(h_n \cdot (x_j)_{k_{m(n)} H}, h_n \cdot y^{m(n)}) < \epsilon/8$ by the definition of $y^{m(n)}$.
  
  Suppose then that $k_n^{-1} g \notin A_j$.
  Since $k_n^{-1} \in F$, this implies $g \in \partial_{E \cup F} A_j$, so that $d_Y(g \cdot x_j, 0_Y) < \delta$.
  This implies $d_X((g \cdot x_j)_{k_n H}, 0_X) = d_X((k_n^{-1} g \cdot x_j)_H, 0_X) \leq d_Y(k_n^{-1} g \cdot x_j, 0_Y) < \epsilon/8$ by our choice of $\delta$.
  Let $i \in I$ be such that $k_n^{-1} g \in A_i$.
  Since $E \cup F$ is symmetric, we have $k_n^{-1} g \in \partial_{E \cup F} A_i$, and hence $d_X((g \cdot x_i)_{k_n H}, 0_X) < \epsilon / 8$ with a similar computation as for $x_j$.
  We also have $h_n \in B^{m(n)}_i$, and from this (as in the case $k_n^{-1} g \in A_j$) we obtain $d_X((g \cdot x_i)_{k_n H}, (g \cdot y)_{k_n H}) < \epsilon/8$.
  This implies $d_X((g \cdot x_j)_{k_n H}, (g \cdot y)_{k_n H}) < \frac{3}{8} \epsilon$.
  
  We have now shown $d_X((g \cdot x_j)_{k_n H}, (g \cdot y)_{k_n H}) < \frac{3}{8} \epsilon$ for each $n \leq N$, from which we compute $d_Y(g \cdot y, g \cdot x_j) < \sum_{n \leq N} 2^{-n} d_X((g \cdot y)_{k_n H}, (g \cdot x_j)_{k_n H}) + \epsilon / 4 < \epsilon$ using the definition of $N$.
  Thus $y$ is an $\epsilon$-realization of $\mathcal{A}$, and $(G, Y)$ is $0$-gluing.
  
  We show that $(G, Y)$ has dense homoclinic points.
  Let $x \in Y$ and $\epsilon > 0$ be arbitrary.
  Let $N \geq 0$ be such that $\sum_{n > N} 2^{-n} < \epsilon / 2$.
  For each $n \leq N$, let $y^n \in X$ be a homoclinic point with $d_X(x_{k_n H}, y^n) < \epsilon / 4$, and for $n > N$, let $y^n = 0_X$.
  Denoting $y = (y^n)_{k_n \in L} \in Y$, we then have $d_Y(x, y) < \epsilon$.
  We claim that $y$ is a homoclinic point in $(G, Y)$.
  Let $\delta > 0$, and choose $M \geq 0$ be such that $\sum_{n > M} 2^{-n} < \delta / 2$.
  For each $n \leq N$, we have a finite set $E_n \subseteq H$ with $d_X(h \cdot y^n, 0_X) < \delta / 4$ whenever $h \in H \setminus E_n$.
  Define the set $E = \{ k_m e k_n^{-1} \;|\; m \leq M, n \leq N, e \in E_n \} \subseteq G$, which is finite since each $E_n$ is.
  Let $g \in G \setminus E$.
  For $m \leq M$, we have $(g \cdot y)_{k_m H} = h_m \cdot y_{k_{p(m)} H}$, where $p(m) \geq 0$ and $h_m \in H$ are such that $g^{-1} k_m = k_{p(m)} h_m^{-1}$.
  If $p(m) \leq N$, we have $g = k_m h_m k_{p(m)}^{-1}$, which implies $h_m \notin E_{p(m)}$, and then $d_X(h_m \cdot y_{k_{p(m)} H}, 0_X) < \delta / 4$.
  On the other hand, $p(m) > N$ implies $h_m \cdot x_{k_{p(m)} H} = 0_X$.
  Then $d_Y(g \cdot y, 0_Y) < \sum_{m \leq M} 2^{-m} d_X((g \cdot y)_{k_m H}, 0_X) + \delta / 2 < \delta$, and thus $y$ is homoclinic.

  Since $\mathcal{U}(G, Y) \in \mathfrak{C}_4$ and $G$ is nil-rigid on $\mathfrak{C}_4$, the map $\tilde f$ is nilpotent on $Y$.
  Thus $f$ is nilpotent on $X$.
\end{proof}

\section{The Case of $\Z$}
\label{sec:Z}

It was shown in \cite{GuRi08} that on full $\Z$-shifts, asymptotically nilpotent cellular automata are nilpotent.
We generalize this result to tiered systems with at least three tiers, one of which is $f$-stabilizing and another satisfies a very weak $0$-gluing property.
This section is one of the places where we do not assume the entire system to be $f$-stabilizing.
For the purposes of our main theorem it would be enough to show the nil-rigidity of $\Z$ on $\mathfrak{C}_2$, but for the sake of proving the nil-rigidity of as many one-dimensional subshifts as possible we give a much stronger result, and for this reason the proof technique of \cite{GuRi08} cannot be directly applied.
Instead, we use a result from the preprint \cite{Sa16} by the first author.
The exact statement of the result is somewhat technical and involves several auxiliary concepts whose definitions we will not repeat here.

\begin{lemma}[Corollary of Theorem 3 in \cite{Sa16}]
\label{lem:Sparse}
Let $X \subset \{0, 1\}^{\Z^2}$ be a binary two-dimensional subshift with at least two points such that $\{ n \geq 0 \;|\; x_{(0,n)} = 1 \}$ is finite for all $x \in X$.
Then one of the following conditions holds for some $x \in X$.
\begin{itemize}
\item $x_{(0,0)} = 1$ and $x_{(m, n)} = 0$ for all $m \in \Z$ and $n < 0$.
\item There exist $r \geq 0$ and a sequence of coordinates $(m_k, n_k)_{k \in \Z}$ such that $m_k < m_{k+1} < m_k + r$, $|n_k - n_{k+1}| \leq r$ and $x_{(m_k, n_k)} = 1$ for all $k \in \Z$, and $x_{(m, n)} = 1$ implies $|m - m_k|, |n - n_k| \leq r$ for some $k \in \Z$.
\end{itemize}
\end{lemma}

The second condition intuitively states that in some configuration of $X$, there is a two-way infinite `path' of nonzero cells that spans the entire horizontal space, and every nonzero cell occurs near this path.
Note also that we apply Theorem 3 of~\cite{Sa16} ``sideways'', as the original result concerns subshifts where $\{ n \geq 0 \;|\; x_{(n,0)} = 1 \}$ is always finite.


\begin{definition}
  \label{def:HomoRec}
  Let $(X, (\Z, X_t)_{t \in \Di}, (\ID_\Z)_{t \leq t'}, f)$ be an ETDS whose groups are all equal to $\Z$.
  For $t \in \Di$, let $\Lambda^t_f$ be the set of homeomorphisms $\gamma : X_t \to X_t$ that commute with $f$ and satisfy $\gamma(X_r) = X_r$ for all $r \leq t$ and $\gamma(0_X) = 0_X$.
  A tier $t \in \Di$ has the \emph{$f$-variable homoclinic recurrence property}, if there exist a tier $w(t) \geq t$ and a finite set $\Gamma \subseteq \Lambda^{w(t)}_f$ such that for each homoclinic point $x \in X_t$, all $\epsilon > 0$ and $c \in \{1, -1\}$, there exists a designation $(A_i, -a_i \cdot x)_{i \in I}$ that has an $(\epsilon, \Gamma)$-realization in $X_{w(t)}$, and satisfies $a_i \in A_i \subseteq \Z$ for all $i \in I$ and $\mathrm{sign}(a_i) = c$ for infinitely many $i \in I$.
  
  Let $(X, (G_t, X_t)_{t \in \Di}, (\phi^{t'}_t)_{t \leq t'} f)$ be an ETDS whose base group is virtually $\Z$.
  We say a tier $t \in \Di$ has the $f$-variable homoclinic recurrence property, if there exists $w(t) \geq t$ and a finite index subgroup $\Z \simeq H \leq G_{w(t)}$ such that tier $t$ of the restricted ETDS $(X, (H, X_r)_{r \leq w(t)}, f)$ has that property with the same choice of $w(t)$.
\end{definition}

In the restricted system, the acting group of each tier is $H$, and the action on lower tiers $X_r$ is defined through the homomorphisms $\phi^{w(t)}_r$.
The fact that the restricted system is an ETDS follows from Lemma~\ref{lem:virtual}.

The idea of this property is that if a homoclinic point $x \in X_t$ satisfies $f^n(x) = q \cdot x$ for some $n \geq 1$ and $q \in \Z$, i.e. is a \emph{spaceship} in cellular automata terminology, then we can glue perturbed copies of $x$ into a single point $y \in X_{w(t)}$ that also satisfies $f^n(y) = q \cdot y$, and where infinitely many of the perturbed copies travel across the origin in the $f$-trajectory.
This means that $f^k(y)$ does not approach $0_X$ as $k$ grows, so $f$ is not asymptotically nilpotent.
Using Lemma~\ref{lem:Sparse}, we can always find such a homoclinic point in an asymptotically nilpotent $\Z$-system that is not nilpotent, and thus reach a contradiction.
We define the property in the language of designations and realizations so that the lemmas of Section~\ref{sec:Gluing} can be used in the proof of the following result.

\begin{lemma}
\label{lem:Z}
  Let $(X, (\Z, X_t)_{t \in \Di}, (\mathrm{ID}_\Z)_{t \leq t'}, f)$ be an ETDS, and suppose that $t \in \Di$ is $f$-stabilizing and $s(t)$ has the $f$-variable homoclinic recurrence property.
  If $f$ is asymptotically nilpotent on $X_{w(s(t))}$, then it is nilpotent on $X_t$.
\end{lemma}

\begin{proof}
  We assume for contradiction that $f$ is not nilpotent on $X_t$.
  Let $\Gamma \subseteq \Lambda^{w(s(t))}_f$ be the finite set of homeomorphisms given by the $f$-variable homoclinic recurrence property of $X_{s(t)}$.
  Let $\Delta = \Delta_{w(s(t))}$ be the canonical expansivity constant of $X_{w(s(t))}$, which is also an expansivity constant for $X_t$ and $X_{s(t)}$.
  Denote by $\Theta = \Theta_{w(s(t))}$ the shadow function of $X_{w(s(t))}$.
  Let
  \[
    T = \{ (x_n)_{n \in \Z} \;|\; \forall n \in \Z: x_n \in X_{s(t)}, x_{n+1} = f(x_n) \}
  \]
  be the set of two-way trajectories of $f$ within $X_{s(t)}$.
  For $x \in X_{w(s(t))}$, we interpret $\Theta(x)$ as an element of $\{0, 1\}^\Z$ by its indicator function, and for $(x_n) \in T$, we define a two-dimensional configuration $y = \Theta((x_n)) \in \{0, 1\}^{\Z^2}$ by the formula $y_{(m,n)} = \Theta(x_n)_m$.
  Define also $Y = \overline{ \Theta(T) } \subseteq \{0, 1\}^{\Z^2}$.
  
  \begin{claim}
  \label{cl:Subshift}
    The set $Y$ is a $\Z^2$-subshift with at least two points.
    For all $y \in Y$ there exists $(x_n)_{n \in \Z} \in T$ with $\Theta((x_n)) \leq y$ in the cellwise order, and $d_X(m \cdot x_n, 0_X) \geq \Delta$ whenever $y_{(m, n)} = 1$.
    The set $\{ n \geq 0 \;|\; y_{(0,n)} = 1 \}$ is finite for all $y \in Y$.
    There exists $L \geq 0$ such that $y_{(m, n)} = 1$ implies $y_{(m + \ell, n-1)} = 1$ for some $\ell \in \{-L, \ldots, L\}$.
  \end{claim}
  
  In the second claim, we say $(x_n)$ is an \emph{associated trajectory} of $y$.
  
  \begin{proof}
    First, $\Theta(T)$ is invariant under translations, since they correspond to shifting a trajectory and/or acting by an element of $\Z$ to each point of a trajectory, which results in another trajectory.
    As the closure of a translation-invariant set, $Y$ is a subshift.
    
    Since $f$ is not nilpotent on $X_t$, for all $k \in \N$ there exists $x_k \in X_t$ such that $f^k(x_k) \neq 0$, and we may assume $d(f^k(x_k), 0) > \Delta$.
    We then have $f^k(x_k) \in X_{s(t)}$ for all $k \in \N$ by $f$-stability.
    The sequence $(f^k(x_k))_{k \in \N}$ has a limit point $x \in X_{s(t)}$, which is the central point $z_0$ of some trajectory $(z_n)_{n \in \Z} \in T$.
    Then $d_X(x, 0) \geq \Delta$, $y = \Theta((z_n)) \in Y$ and $y_{(i, 0)} = 1$ for some $i \in \Z$.
    Since $0^{\Z^2} \in Y$, this implies that $Y$ contains at least two configurations.
    
    The second claim is clear if $y \in \Theta(T)$.
    If $y \notin \Theta(T)$, then there is a sequence $((x^k_n)_{n \in \Z})_{k \in \N}$ of trajectories with $\Theta((x^k_n)) \longrightarrow y$ as $k \longrightarrow \infty$, and by passing to a subsequence we may assume that the sequence itself converges in $T$.
    By the properties of the shadow function, $\Theta(\lim_k (x_n^k)) \leq \lim_k \Theta((x_n^k)) = y$ holds in the cell-wise order, and $y_{(m, n)} = 1$ implies $d_X(m \cdot x_n, 0_X) \geq \Delta$.
    We may choose $(x_n) = \lim_k (x_n^k)$.
    
    Let $y \in Y$ be arbitrary, and let $(x_n) \in T$ be an associated trajectory.
    Since $x_0 \in X_{s(t)}$ and $f$ is asymptotically nilpotent on $X_{s(t)}$, there exists $N \in \N$ such that $d(x_n, 0) = d(f^n(x_0), 0) \leq \Delta$ for all $n \geq N$.
    This means that $y_{(0, n)} = 0$ for all $n \geq N$.
    
    Finally, suppose that for all $L \geq 0$ there exists $y^L \in Y$ such that $y^L_{(0,1)} = 1$ but $y^L_{(\ell, 0)} = 0$ for all $\ell \in \{-L, \ldots, L\}$.
    Let $y \in Y$ be a limit point of the $y^L$, which then satisfies $y_{(0,1)} = 1$ and $y^L_{(\ell, 0)} = 0$ for all $\ell \in \Z$.
    If $(x_n) \in T$ is an associated trajectory, this implies $x_0 = 0$ by expansivity, and then $x_1 = f(x_0) = 0$, a contradiction.
  \end{proof}
  
  We can now apply Lemma~\ref{lem:Sparse} and obtain a configuration $y \in Y$ satisfying one of the two conditions.
  The existence of $L$ in Claim~\ref{cl:Subshift} implies that $y$ does not satisfy the first condition, so it satisfies the second with some radius $r \geq 0$ and path $(m_k, n_k)_{k \in \Z}$ that is increasing in the first coordinate.
  We now extract another path that is increasing in the second coordinate.
  
  \begin{claim}
  \label{cl:Path}
    For all $k \in \Z$ there exists $p \in \{-(r + 2) L, \ldots, (r + 2) L\}$ such that $n_{k + p} < n_k$.
  \end{claim}
  
  \begin{proof}
    Assume on the contrary that we have $n_{k + p} \geq n_k$ wherever $|p| \leq (r+2)L$.
    Since $y_{(m_k, n_k)} = 1$, Claim~\ref{cl:Subshift} implies that some $\ell_1 \in \{-L, \ldots, L\}$ satisfies $y_{(m_k + \ell_1, n_k - 1)} = 1$.
    Inductively we find a sequence $(\ell_q)_{q \geq 1}$ of integers $\ell_q \in \{-L, \ldots, L\}$ satisfying $y_{(m_k + K_q, n_k - q)} = 1$, where $K_q = \sum_{i = 1}^q \ell_i$.
    Note that $|K_q| \leq q L$.
    Choosing $q = r+1$, we have $y_{(m_k + K, n_k - r - 1)} = 1$ where $|K| \leq (r+1) L$.
    
    There now exists $k' \in \Z$ with $|m_k + K - m_{k'}| \leq r$ and $|n_k - r - 1 - n_{k'}| \leq r$, and the latter inequality implies $n_k > n_{k'}$.
    Since $n_{k + p} \geq n_k$ whenever $|p| \leq (r + 2) L$, we must have $|k - k'| > (r+2)L$.
    But then $|m_k - m_{k'}| > (r+2)L$, which implies $|K| > (r+2)L - r = (r+1)L$, a contradiction.
  \end{proof}
  
  Using Claim~\ref{cl:Path}, we construct a sequence $(k_i)_{i \in \N}$ by setting $k_0 = 0$ and defining $k_{i+1} = k_i + p$, where $|p| \leq (r + 2)L$ is such that $n_{k_i + p} < n_{k_i}$.
  We obtain another path $(m_{k_i}, n_{k_i})_{i \in \N}$ where both sequences $(m_{k_i})_i$ and $(n_{k_i})_i$ have bounded differences, and the latter is strictly decreasing.
  By shifting $(m_{k_i}, n_{k_i})$ to the origin and taking a limit in $Y$, we obtain a two-way infinite path $(m_{k_i}, n_{k_i})_{i \in \Z}$ with these properties.
  Then the reversed path $(m_{-k_i}, n_{-k_i})_{i \in \Z}$ is strictly increasing in the second coordinate.
  We may even assume $n_{-k_i} = i$, since the limit configuration cannot contain empty rows.
  Finally, since the set $\{ -k_i \;|\; i \in \Z \}$ is syndetic in $\Z$ by construction, there exists $s \in \N$ such that $y_{(m, i)} = 1$ implies $|m - m_{-k_i}| \leq s$.
  In a slight abuse of notation, we redefine $(m_k)_{k \in \Z} = (m_{-k_i})_{i \in \Z}$, so that $(m_k)$ has bounded differences and $y_{(m_k, k)} = 1$ for all $k \in \Z$.
  We may also assume $m_0 = 0$.
  Note that the sequence $(m_k)$ may not be monotonic anymore.
  
  Let $(x_k)_{k \in \Z} \in T$ be associated with $y$.
  Since the $1$-symbols of each row of $y$ are within a bounded distance from each other, Lemma~\ref{lem:BoundedShadow} implies that the points $x_k$ come from translates of a finite set.
  Thus there exist $n \geq 1$ and $q \in \Z$ with $x_{k + n} = q \cdot x_k$ for all $k \in \N$.
  We may also assume $m_{k + n} = m_k - q$ for all $k \in \Z$.
  If $q = 0$, then $x_0$ is an $f$-periodic point distinct from $0_X$, which contradicts the asymptotic nilpotency of $f$, so we assume $q \neq 0$.
  In the language of cellular automata, $x_0$ is now a spaceship.
  
  
  Let $\epsilon > 0$ be small, and let $0 < \delta \leq \epsilon$ be such that $d_X(z, z') < \delta$ implies $d_X(f^n(z), f^n(z')) < \epsilon$ for $z, z' \in X_{w(s(t))}$.
  By the $f$-variable homoclinic recurrence property of $s(t)$ applied to the homoclinic point $x_0$, $\delta$ and $c = \mathrm{sign}(q)$, there exist a designation $\mathcal{A} = (A_i, -a_i \cdot x_0)_{i \in I}$ and a $(\delta, \Gamma)$-realization $(z, W)$ in $X_{w(s(t))}$.
  Furthermore, $a_i$ has the same sign as $q$ for infinitely many $i \in \N$.
  The situation is visualized in Figure~\ref{fig:GlueTraj}.
  
  Consider now the designation $(-q + A_i, (q - a_i) \cdot x_0)_{i \in I}$.
  We claim that $(f^n(z), W)$ and $(q \cdot z, W)$ are both $(\epsilon, \Gamma)$-realizations for it, which implies $f^n(z) = q \cdot z$ by Lemma~\ref{lem:FuncGlue}, as long as $\epsilon$ is small enough.
  For this, let $m \in \Z$ be arbitrary.
  If $m \in -q + A_i$ for some $i \in I$, then $m + q \in A_i$, and we have $d_X((m + q) \cdot z, \gamma_i((m + q - a_i) \cdot x_0)) < \delta$.
  Since $\delta \leq \epsilon$, this directly implies that $(q \cdot z, W)$ is an $(\epsilon, \Gamma)$-realization.
  On the other hand, we have $d_X(m \cdot f^n(z), \gamma_i(m \cdot ((q - a_i) \cdot x_0))) = d_X(m \cdot f^n(z), \gamma_i((m - a_i) \cdot x_n)) = d_X(f^n(m \cdot z), f^n(\gamma_i((m - a_i) \cdot x_0))) < \epsilon$ by the definition of $\delta$.
  Thus $(f^n(z), W)$ is also an $(\epsilon, \Gamma)$-realization, and we have shown $f^n(z) = q \cdot z$.
  
  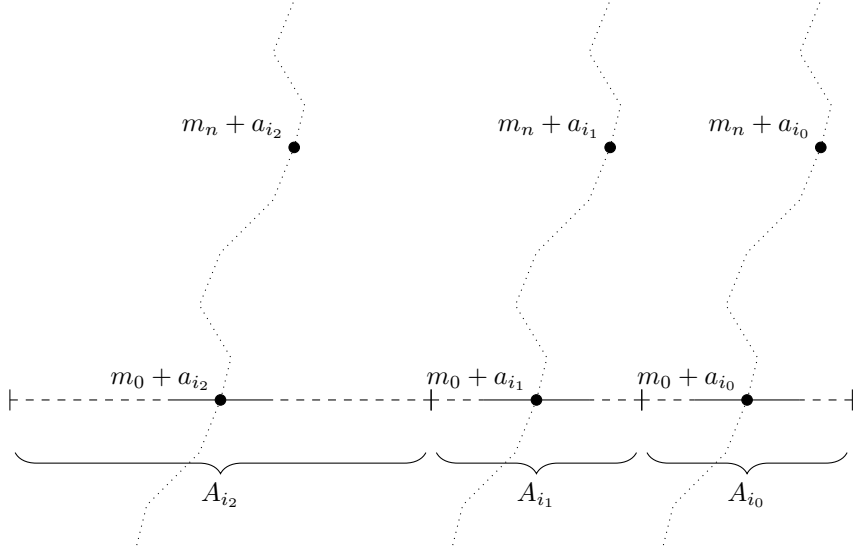
\begin{figure}[ht]
  \begin{center}
  \begin{tikzpicture}[scale=1.4]
  
    \foreach \p/\w in {-3/1.5,0/0.5,2/0.5}{
      \coordinate (coordm\p) at (\p,0);
      \draw [fill] (coordm\p) circle (0.05cm);
      \draw [fill] ($(coordm\p)+(0.7,2.4)$) circle (0.05cm);
      \draw (\p,0) -- ++(0.5,0);
      \draw (\p,0) -- ++(-0.5,0);
      \pgfmathsetmacro{\dl}{\p-\w-0.5}
      \pgfmathsetmacro{\dr}{\p+\w+0.5}
      \draw (\dl,-0.1) -- ++(0,0.2);
      \draw (\dr,-0.1) -- ++(0,0.2);
      \draw [dashed] (\dl,0) -- ++(\w,0);
      \draw [dashed] (\dr,0) -- ++(-\w,0);
      \draw [dotted] (coordm\p) -- ++(-.2,-.5) -- ++(-.5,-.5) -- ++(-.1,-.4);
      \draw [dotted] (coordm\p) -- ++(.1,.4) -- ++(-.3,.5) -- ++(.2,.5) -- ++(.5,.5) -- ++(.2,.5) -- ++(.1,.4) -- ++(-.3,.5) -- ++(.2,.5);
    }
    
    \node [above left] at (coordm-3) {$m_0 + a_{i_2}$};
    \node [above left] at (coordm0) {$m_0 + a_{i_1}$};
    \node [above left] at (coordm2) {$m_0 + a_{i_0}$};
    
    \node [above left] at ($(coordm-3)+(0.7,2.4)$) {$m_n + a_{i_2}$};
    \node [above left] at ($(coordm0)+(0.7,2.4)$) {$m_n + a_{i_1}$};
    \node [above left] at ($(coordm2)+(0.7,2.4)$) {$m_n + a_{i_0}$};
    
    \node [below] at (-3, -0.7) {$A_{i_2}$};
    \draw [decorate,decoration={brace,amplitude=8pt}] (-1.05,-0.5) -- (-4.95,-0.5);
    \node [below] at (0, -0.7) {$A_{i_1}$};
    \draw [decorate,decoration={brace,amplitude=8pt}] (0.95,-0.5) -- (-0.95,-0.5);
    \node [below] at (2, -0.7) {$A_{i_0}$};
    \draw [decorate,decoration={brace,amplitude=8pt}] (2.95,-0.5) -- (1.05,-0.5);
    
    
  
  \end{tikzpicture}
  \end{center}
  \caption{The construction of the realization $(z, W)$. The dotted lines are translates of the path $(m_k, k)_{k \in \Z}$. The solid line in $A_i$ marks the shadow of $x_0$. The indices $i_0$, $i_1$ and $i_2$ are arbitrary elements of $I$. Depending on $c$, there are guaranteed to be infinitely many coordinates $a_i$ either to the left or to the right.}
  \label{fig:GlueTraj}
  \end{figure}
  
  Take an arbitrary index $i \in I$ with $a_i$ having the same sign as $q$, and let $a_i = k q + r$, where $k \geq 0$ and $0 \leq r < |q|$.
  We have $f^{k n}(z) = k q \cdot z = (a_i - r) \cdot z$, and hence $d_X(r \cdot f^{k n}(z), 0_X) = d_X(a_i \cdot z, 0_X) \geq d_X(\gamma_i(x), 0_X) - d_X(a_i \cdot z, \gamma_i(x)) > d_X(\gamma_i(x), 0_X) - \epsilon$ since $a_i \in A_i$.
  If $\epsilon$ is small enough, then $d_X(r \cdot f^{k n}(z), 0_X)$ is bounded away from zero for all such $k$, and then so is $d_X(f^{k n}(z), 0_X)$.
  As there are infinitely many choices for $i$, we can also choose $k$ from infinitely many possibilities, and hence $f$ is not asymptotically nilpotent.
\end{proof}

We now extend this result to all groups that are virtually $\Z$.

\begin{corollary}
  \label{cor:Z}
  Let $(X, (G_t, X_t)_{t \in \Di}, (\phi^{t'}_t)_{t \leq t'}, f)$ be an ETDS whose base group is virtually $\Z$.
  If a tier $t \in \Di$ is $f$-stabilizing, $s(t)$ has the $f$-variable homoclinic recurrence property, and $f$ is asymptotically nilpotent on $X_{w(s(t))}$, then it is nilpotent on $X_t$.
\end{corollary}

\begin{proof}
  Let $\Z \simeq H \leq G_{w(s(t))}$ be a finite index subgroup such that tier $s(t)$ of the restricted ETDS $\mathcal{Y} = (X, (H, X_r)_{r \leq w(s(t))}, f)$ has the $f$-variable homoclinic recurrence property.
  Then the tier $t$ is $f$-stabilizing also in $\mathcal{Y}$.
  The claim follows from Lemma~\ref{lem:Z} applied to the system $\mathcal{Y}$.
\end{proof}

Finally, we show that systems in the class $\mathfrak{C}_1$ with a virtually $\Z$ base group have the $f$-variable homoclinic recurrence property, tying it more closely to our general framework.

\begin{lemma}
  \label{lem:HomoRecC1}
  Let $(X, (G_t, X_t)_{t \in \Di}, f) \in \mathfrak{C}_1$ be an ETDS with a virtually $\Z$ base group.
  Then each tier $t \in \Di$ has the $f$-variable homoclinic recurrence property.
\end{lemma}

\begin{proof}
  Let $t \in \Di$ be an arbitrary tier, and let $w(t) \geq t$ be given by the weak $f$-variable $0$-gluing property.
  The base group $G_0$ has a finite index subgroup $F$ isomrphic to $\Z$.
  Let $K = (\phi^{w(t)}_{t_0})^{-1}(F)$ be the preimage of this group in $G_{w(t)}$.
  The restriction $\phi^{w(t)}_{t_0} : K \to F$ is a homomorphism onto a free group, so it splits, and we have an injection $\psi : F \to G_{w(t)}$ that is its right inverse.
  The image $\psi(F) \simeq F \simeq \Z$ has finite index in $G_{w(t)}$.
  
  Consider the restricted system $(X, (\Z, X_r)_{r \leq w(t)}, f)$, where we have identified $\psi(F)$ with $\Z$.
  By Lemma~\ref{lem:virtual}, it is an ETDS, and the tier $t$ is weakly $f$-variably $0$-gluing with the same choice of $w(t)$ as in the original system.
  We claim that tier $t$ of this system has the $f$-variable homoclinic recurrence property.
  For this, let $\Gamma \subseteq \Lambda^{w(t)}_f$ be the finite set of homeomorphisms given by the $0$-gluing property of $t$.
  Let $x \in X_t$ be a homoclinic point, and let $\epsilon > 0$ and $c \in \{1, -1\}$ be arbitrary.
  Let $\delta > 0$ and $E \subset \Z$ be given for $\epsilon$ by the $0$-gluing property of $t$.
  We may assume $E = \{-M, \ldots, M\}$ for some $M \in \N$.
  Let $N \in \N$ be such that $d_X(n \cdot x, 0_X) < \delta$ whenever $|n| > N$, and define $B = \{-M-N, \ldots, M+N\}$.
  Let $I = \Z$, and for $i \in I$, let $a_i = i (2 (M + N) + 1)$ and $A_i = a_i + B$.
  Consider the designation $\mathcal{A} = (A_i, -a_i \cdot x)_{i \in I}$.
  For $i \in I$ and $n \in \partial_E A_i$, we have $|n - a_i| > N$, and hence $d_X(n \cdot (-a_i \cdot x), 0_X) = d_X((n - a_i) \cdot x, 0_X) < \delta$.
  This means $\mathcal{A}$ has an $(\epsilon, \Gamma)$-realization.
  Furthermore, we have $a_i \in A_i$ for all $i \in I$, and infinitely many of the $a_i$ have the same sign as $c$.
  Hence tier $t$ has the $f$-variable homoclinic recurrence property in the restricted system, and by definition also in the original system.
\end{proof}

The nil-rigidity of all virtually $\Z$ groups on $\mathfrak{C}_1$ now directly follows.

\begin{theorem}
  \label{thm:Z}
  Every virtually $\Z$ group is nil-rigid on $\mathfrak{C}_1$.
\end{theorem}

Note that virtually $\Z$ groups are also nil-rigid on $\mathfrak{C}_2$, $\mathfrak{C}_3$ and $\mathfrak{C}_4$, as these classes are contained in $\mathfrak{C}_1$.
For example, this result holds for the infinite dihedral group, and all groups of the form $\Z \times F$, where $F$ is finite.

\section{Periodic Points and Mortality}
\label{sec:Periods}

This section contains two key ideas: nested gluing (Lemma~\ref{lem:Towers} and Lemma~\ref{lem:Mortal}) and mortalization by periodization (Lemma~\ref{lem:Periodization} and Corollary~\ref{cor:PeriodicC2}).
The former technique was used in \cite{GuRi08} to prove that if $f$ is an asymptotically nilpotent cellular automaton on a $\Z$-full shift, then $f$ sends all mortal homoclinic points to $0_X$ in a bounded number of steps.
The idea is to glue together infinitely many mortal configurations, each of which wanders away from $0_X$ at some point in its $f$-orbit.
The resulting configuration wanders away from $0_X$ infinitely often, contradicting the asymptotic nilpotency of $f$.
The latter technique was used in \cite{Sa12c} in the context of $\Z^d$-full shifts to prove that under the iteration of an asymptotically nilpotent cellular automaton $f$, the shadow of a point can only spread a finite distance in the direction of a basis vector of $\Z^d$.
The idea is that a configuration that is periodic along all basis vectors except one is essentially a configuration of a $\Z$-full shift on a larger alphabet, and $f$ simulates a one-dimensional cellular automaton on such configurations.
Since asymptotically nilpotent $\Z$-cellular automata are nilpotent by the results of \cite{GuRi08}, such configurations are mortal under $f$.
The technique of nested gluing can be modified to prove the desired result.

We begin by formalizing the idea of periodization in the setting of tiered dynamical systems.
Let $K \trianglelefteq G$ be groups.
We study the periodic points of a TDS $\X = (X, (G_t, X_t)_{t \in \Di}, (\phi^{t'}_t)_{t \leq t'})$ with $G_0 = G$.
Our aim is to construct a new tiered system consisting of the $K$-periodic points of the original system, and show that it retains the relevant properties of $\X$.
In particular, we show that if $G$ is residually finite and $\X \in \mathfrak{C}_3$, then the system of periodic points is in $\mathfrak{C}_2$.

For $t \in \Di$, denote $K_t = (\phi^t_{t_0})^{-1}(K)$, and let $\VVDi_t$ be the set of finite-index normal subgroups of $K_t$ that are normal also in $G_t$, ordered by reverse inclusion, so that the bottom element is $K_t$. Denote $\VVDi = \{ (t, F) \;|\; t \in \Di, F \in \VVDi_t \}$, ordered by the relation defined by $(t, F) \leq (t', F')$ iff $t \leq t'$ and $\phi^{t'}_t(F') \subseteq F$.
Then $\VVDi$ is a directed set with bottom element $(t_0, K)$.
For $(t, F) \leq (t', F') \in \VVDi$, define the finite-to-one homomorphism $\phi^{(t', F')}_{(t, F)} : G_{t'} / F' \to G_t / F$ by $g F' \mapsto \phi^{t'}_t(g) F$.
Then the triple $\mathcal{Y} = (X, (G_t / F, \fix{F}{X_t})_{(t, F) \in \VVDi}, (\phi^{(t', F')}_{(t, F)}))$ is a TDS with bottom tier $(G / K, \fix{K}{X_0})$.
Namely, if $(t, F) \leq (t', F') \in \VVDi$, then $\phi^{t'}_t(F') \subseteq F$, so that $\fix{F}{X_t} \subseteq \fix{\phi^{t'}_t(F')}{X_t} \subseteq \fix{F'}{X_{t'}}$.
Each tier $(K_t, \fix{F}{X_t})$ has the same expansivity constant as $(G_t, X_t)$.
The other conditions of Definition~\ref{def:Tiers} are easy to verify.

\begin{definition}
Let $\X = (X, (G_t, X_t)_{t \in \Di}, (\phi^{t'}_t)_{t \leq t'})$ be a TDS and $K \trianglelefteq G_0$ a normal subgroup. Then the TDS
\[ \mathcal{Y} = (X, (G_t / F, \fix{F}{X_t})_{(t, F) \in \VVDi}, (\phi^{(t', F')}_{(t, F)})) \]
defined in the two paragraphs above is denoted by $\fix{K}{\X}$.
\end{definition}

We now prove that the system of $K$-periodic points respects evolution maps and the sets of stabilizing and weakly $0$-gluing tiers.
For this, we define the notion of a shadow projected onto a quotient group.
Recall the definition of the $t$-shadow function $\Theta_t$ from Definition~\ref{def:Shadows}.


\begin{definition}
\label{def:KShadow}
Let $\X = (X, (G_t, X_t)_{t \in \Di}, (\phi^{t'}_t)_{t \leq t'})$ be a TDS, and let $K \trianglelefteq G$ be a normal subgroup.
Let $\pi_K : G \to G/K$ be the natural projection.
Denote $\psi^K_t = \pi_K \circ \phi^t_{t_0} : G_t \to G/K$.
For $x \in X_t$, define $\Theta^K_t(x) = \psi^K_t(\Theta_t(x)) \subseteq G/K$.
This set is called the \emph{$(K,t)$-shadow} of $x$.
\end{definition}

These functions are clarified in Figure~\ref{fig:diagram}.

\begin{figure}[htp]
  
  \begin{center}
    \begin{tikzpicture}
      
      \matrix (m) [matrix of math nodes,
                   row sep=3em,
                   column sep=4em,
                   minimum width=2em,
                   nodes={anchor=center}] {

        X_t & G_t & \\
        X_0 & G   & G/K \\
      };
      
      \path [->]
      (m-1-1) edge [double] node [above] {$\Theta_t$} (m-1-2)
              edge [double,bend left=65] node [above] {$\Theta^K_t$} (m-2-3)
      (m-1-2) edge node [left] {$\phi^t_{t_0}$} (m-2-2)
              edge node [above right] {$\psi^K_t$} (m-2-3)
      (m-2-1) edge [double] node [below] {$\Theta_{t_0}$} (m-2-2)
      (m-2-2) edge node [below] {$\pi_K = \psi^K_{t_0}$} (m-2-3)
      ;
    
    \end{tikzpicture}
  \end{center}
  
  \caption{The functions defined in Definition~\ref{def:KShadow}. Double arrows denote functions to subsets of the target.}
  \label{fig:diagram}
  
\end{figure}
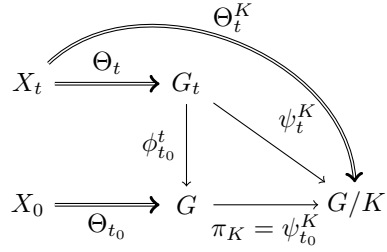

\begin{lemma}
\label{lem:PeriodicProps}
Let $\X = (X, (G_t, X_t)_{t \in \Di})$ be a TDS and $f : X \to X$ its evolution map, and let $K \trianglelefteq G_0$.
Then $f$ is also an evolution map of $\fix{K}{\X}$.
If a tier $t \in \Di$ is $f$-stabilizing or weakly $0$-gluing, then $(t, F)$ has the same property in the system $\fix{K}{\X}$ for each $F \in \VVDi_t$ with $s(t, F) = (s(t), (\phi^{s(t)}_t)^{-1}(F))$, and similarly for $w$.
\end{lemma}

\begin{proof}
Let $(t, F) \leq (t', F') \in \VVDi$, and suppose a point $x \in \fix{F}{X_t}$ satisfies $f(x) \in \fix{F'}{X_{t'}}$.
For each $g \in F'$ we have $\phi^{t'}_t(g) \in F$, which implies $g \cdot f(x) = f(\phi^{t'}_t(g) \cdot x) = f(x)$, and then $f(x) \in \fix{F'}{X_{t'}}$.
Hence $f$ is an evolution of $\fix{K}{\X}$.
If $t$ is $f$-stabilizing, we can choose $t' = s(t)$ to verify the claim about $s(t, F)$.

We then assume $t$ is weakly $0$-gluing, and prove that $(t, F)$ has the same property.
Let $\Delta > 0$ be an expansivity constant for $(G_{w(t)}, X_{w(t)})$.
Let $0 < \epsilon < \Delta / 2$ be arbitrary, and let $\delta > 0$ and $E \subseteq G_{w(t)}$ be given by the weak $0$-gluing property of $\X$ for $t$ and $\epsilon$.
Denote $F' = (\phi^{w(t)}_t)^{-1}(F)$.
Then $E' = \{e F' \;|\; e \in E\} \subset G_{w(t)} / F'$ is a finite set.
The tier of $\VVDi$ in which we will construct realizations is $w(t, F) = (w(t), F')$.
Let $\mathcal{A} = (A_i, x_i)_{i \in I}$ be a $(\delta, E')$-designation in $(G_{w(t)} / F', \fix{F'}{X_{w(t)}})$ with $x_i \in \fix{F}{X_t}$ for all $i \in I$, so that $(A_i)_{i \in I}$ forms a partition of $G_{w(t)} / F'$.

Denote $B_i = \bigcup A_i \subset G_{w(t)}$.
Then $(B_i)_{i \in I}$ forms a partition of $G_{w(t)}$.
Let $i \in I$ and $g \in \partial_E B_i$.
Then $g F' \in A_i$, and $e g \notin B_i$ for some $e \in E$, which implies $e g F' = (e F')(g F') \notin A_i$, and since $e F' \in E'$, we have $g F' \in \partial_{E'} A_i$.
Then $d_X(g \cdot x_i, 0) = d_X(g F' \cdot x_i) < \delta$.
Thus $(B_i, x_i)_{i \in I}$ is a $(\delta, E)$-designation in $(G_{w(t)}, X_{w(t)})$, and since $x_i \in X_t$ for all $i \in I$, it has an $\epsilon$-realization $y \in X_{w(t)}$.
We claim that $y \in \fix{F'}{X_{w(t)}}$, and for that, let $h \in F'$ and $g \in G_{w(t)}$ be arbitrary, and let $i \in I$ be such that $g h \in B_i$.
Then $d_X(g h \cdot y, g \cdot y) \leq d_X(g h \cdot y, g h \cdot x_i) + d_X(g \cdot x_i, g \cdot y) < \Delta$, and since $g$ was arbitrary, we have $h \cdot y = y$.
We have shown that $\fix{K}{\X}$ is weakly $0$-gluing.
\end{proof}

\begin{lemma}
\label{lem:Periodization}
Let $\X = (X, (G_t, X_t)_{t \in \Di})$ be a TDS, and let $K \trianglelefteq G_0$.
Suppose a tier $t \in \Di$ is weakly $0$-gluing and $G_{w(t)}$ is residually finite.
There exists a finite set $N \subseteq G / K$ such that for each homoclinic point $y \in X_t$ and $\epsilon > 0$, there exists $F \in \VVDi_{w(t)}$ and a point $z \in \fix{F}{X_{w(t)}}$ that satisfies $d_X(y, z) < \epsilon$ and $\Theta^K_{w(t)}(z) \subseteq N \cdot \Theta^K_{w(t)}(y)$.
\end{lemma}

In particular, the last condition implies that $z$ is homoclinic in the system $(G_{w(t)} / F, \fix{F}{X_{w(t)}})$.

\begin{proof}
The idea is to use the residual finiteness of $G_{w(t)}$ and the weak $0$-gluing property of $t$ to glue infinitely many shifted copies of $y$ into a single point that is $F$-periodic for some sparse enough finite-index subgroup $F \trianglelefteq K_{w(t)}$.
This argument is visualized in the right half of Figure~\ref{fig:PeriodDenseFinites}.

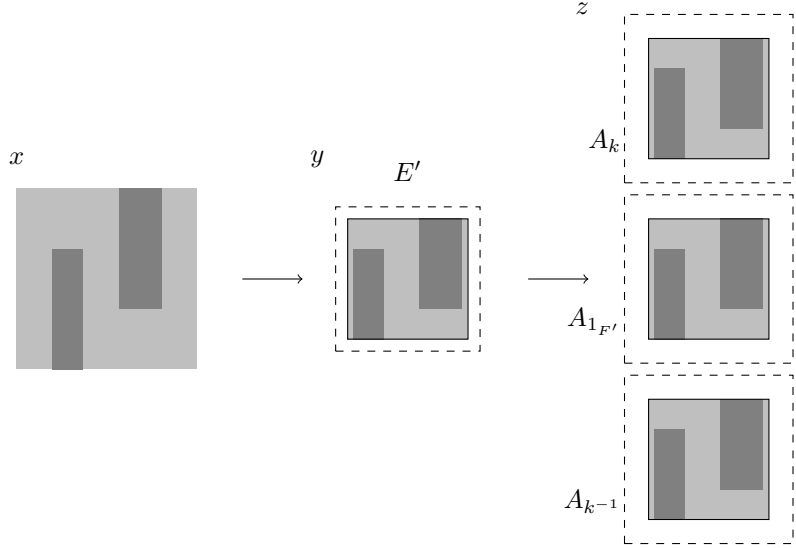
\begin{figure}

  \begin{center}
    \begin{tikzpicture}[scale=0.8]

      \coordinate (xc) at (0,0);
      \coordinate (yc) at (5,0);
      \coordinate (zc) at (10,0);
      
      \fill [black!25] ($(xc)-(1.5,1.5)$) rectangle ++(3,3);
      \fill [black!50] ($(xc)-(0.9,1.5)$) rectangle ++(0.5,2);
      \fill [black!50] ($(xc)-(-0.2,0.5)$) rectangle ++(0.7,2);

      \node at ($(xc)+(-1.5,2)$) {$x$};

      \path [->] (2.25,0) edge (3.25,0);

      \fill [black!25] ($(yc)-(1,1)$) rectangle ++(2,2);
      \fill [black!50] ($(yc)-(0.9,1)$) rectangle ++(0.5,1.5);
      \fill [black!50] ($(yc)-(-0.2,0.5)$) rectangle ++(0.7,1.5);
      \draw ($(yc)-(1,1)$) rectangle ++(2,2);
      \draw [dashed] ($(yc)-(1.2,1.2)$) rectangle ++(2.4,2.4);

      \node at ($(yc)+(-1.5,2)$) {$y$};
      \node at ($(yc)+(0,1.8)$) {$E'$};

      \path [->] (7,0) edge (8,0);

      \foreach \y in {-3,0,3}{
        \fill [black!25] ($(zc)-(1,1+\y)$) rectangle ++(2,2);
        \fill [black!50] ($(zc)-(0.9,1+\y)$) rectangle ++(0.5,1.5);
        \fill [black!50] ($(zc)-(-0.2,0.5+\y)$) rectangle ++(0.7,1.5);
        \draw ($(zc)-(1,1+\y)$) rectangle ++(2,2);
        \draw [dashed] ($(zc)-(1.4,1.4+\y)$) rectangle ++(2.8,2.8);
      }

      \node at ($(zc)+(-2.1,4.5)$) {$z$};
      \node [left] at ($(zc)+(-1.3,-3.7)$) {$A_{k^{-1}}$};
      \node [left] at ($(zc)+(-1.3,-0.7)$) {$A_{1_{F'}}$};
      \node [left] at ($(zc)+(-1.3,2.3)$) {$A_k$};

    \end{tikzpicture}
  \end{center}

  \caption{Visualization of the construction of homoclinic points in $\fix{K}{\X}$. The vertical axis represents $F$, and the horizontal axis represents $G_t / F$.}
  \label{fig:PeriodDenseFinites}
\end{figure}

Let $\Delta > 0$ be the canonical expansivity constant for $(G_{w(t)}, X_{w(t)})$, and let $N' \subseteq G_{t(w)}$ be finite such that for $x \in X_{w(t)}$, the condition $d_X(b^{-1} \cdot x, 0) \leq \frac{3}{2} \Delta$ for all $b \in N'$ implies $d_X(x, 0) \leq \Delta$, which is possible since $\frac{3}{2} \Delta$ is also an expansivity constant for $(G_{w(t)}, X_{w(t)})$.
We define $N = \psi^K_{w(t)}(N')$.
We may also assume $\epsilon < \Delta / 2$.
Let $0 < \delta < \epsilon$ and $E \subseteq G_{w(t)}$ be given by the weak $0$-gluing property of $\X$ for $t$ and $\epsilon$.
Let $E' \subseteq G_{w(t)}$ be a finite set with $d_X(g \cdot y, 0) < \delta$ whenever $g \in G_{w(t)} \setminus E'$.
We may assume $1_{G_{w(t)}} \in E \subseteq E'$.

Denote $M = (E')^{-1} E^{-1} E E'$, which is a finite set.
Let $F' \trianglelefteq G_{w(t)}$ be a finite-index subgroup such that $F' \cap M \subseteq \{1_{G_{w(t)}}\}$.
Such an $F'$ exists, since $G_{w(t)}$ is residually finite.
Define $F = F' \cap K_{w(t)} \in \VVDi_{w(t)}$.
For $k \in F$, let $A_k = E E' k^{-1}$.
We define $\mathcal{A} = (A_k, k \cdot y)_{k \in F}$, and claim that it is a $(\delta, E)$-designation in $(G_{w(t)}, X_{w(t)})$.
First, if $k, k' \in F$ are such that $A_k \cap A_{k'} \neq \emptyset$, then $k^{-1} k' \in M$, which implies $k' = k$ by our choice of $F'$.
Thus the sets $A_k$ are disjoint.
Let then $k \in F$ and $g \in \partial_E A_k$.
Then $e g \notin A_k = E E' k^{-1}$ for some $e \in E$, so in particular $g k \notin E'$.
This implies $d_X(g k \cdot y, 0) < \delta$ by the definition of $E'$.

Since $\mathcal{A}$ is a $(\delta, E)$-designation whose points belong to $X_t$, it has an $\epsilon$-realization $z \in X_{w(t)}$.
For all $e \in E E'$ and $k \in F$, we have $e k \in A_{k^{-1}}$, which implies
\begin{equation}
  \label{eq:EEk}
  d_X(e k \cdot z, e \cdot y) < \epsilon
\end{equation}
We now claim that $z \in \fix{F}{X_{w(t)}}$, and for that, let $k \in F$ and $g \in G_{w(t)}$ be arbitrary.
If $g \in A_h$ for some $h \in F$, then $g = e h^{-1}$ for some $e \in E E'$, and we have $d_X(e h^{-1} k \cdot z, e h^{-1} \cdot z) \leq d_X(e h^{-1} k \cdot z, e \cdot y) + d_X(e \cdot y, e h^{-1} \cdot z) < \Delta$ by~\eqref{eq:EEk} and the fact that $\epsilon < \Delta / 2$.
On the other hand, if $g \notin A_h$ for all $h \in F'$, then $g k \notin A_h$ for all $h \in F$, and we have $d_X(g k \cdot z, g \cdot z) \leq d_X(g k \cdot z, 0) + d_X(0, g \cdot z) < \Delta$.
Since $g \in G_{w(t)}$ was arbitrary, we have $k \cdot z = z$ by the expansivity of $(G_{w(t)}, X_{w(t)})$, and thus $z \in \fix{F}{X_{w(t)}}$.


Next, we prove $\Theta^K_{w(t)}(z) \subseteq N \cdot \Theta^K_{w(t)}(y)$, with the definition $N = \psi^K_{w(t)}(N')$.
For that, let $g \in G_{w(t)}$ be such that $g K \notin N \cdot \Theta^K_{w(t)}(y)$, which is equivalent to $d_X(b^{-1} g h \cdot y, 0) \leq \Delta$ for all $b \in N'$ and $h \in K_{w(t)}$.
We claim that this implies $d_X(b^{-1} g h \cdot z, 0) \leq \frac{3}{2} \Delta$ for all $b \in N'$ and $h \in K_{w(t)}$.
Namely, if $g h \notin A_k$ for all $k \in F$, then $d_X(b^{-1} g h \cdot z, 0) < \epsilon < \Delta/2$, since $z$ is an $\epsilon$-realization of $\mathcal{A}$.
On the other hand, if $b^{-1} g h \in A_k$ for some $k \in F$, then $b^{-1} g h = e k^{-1}$ for some $e \in E E'$.
Then $b^{-1} g h k = e \in A_1$, and since $h k \in K_{w(t)}$, we have $d_X(b^{-1} g h \cdot z, 0) \leq d_X(b^{-1} g h \cdot z, b^{-1} g h k \cdot y) + d_X(b^{-1} g h k \cdot y, 0) < \Delta + \epsilon < \frac{3}{2} \Delta$.
By the definition of $N'$, we now obtain $d_X(g h \cdot z, 0) \leq \Delta$ for all $h \in K_{w(t)}$, which implies $g K \notin \Theta^K_{w(t)}(z)$.
Thus we have shown $\Theta^K_{w(t)}(z) \subseteq N \cdot \Theta^K_{w(t)}(y)$.

Finally, we note that $z$ is $\epsilon$-close to $y$: since $1_{G_{t(e)}} \in E E'$, we have $d_X(z, y) < \epsilon$ by \eqref{eq:EEk}.
\end{proof}

\begin{corollary}
\label{cor:PeriodicC2}
  Let $(\X, f) \in \mathfrak{C}_3$ be an ETDS with bottom group $G$, and let $K \trianglelefteq G_0$.
  Then $(\fix{K}{\X}, f) \in \mathfrak{C}_2$.
\end{corollary}

\begin{proof}
  Denote $\X = (X, (G_t, X_t)_{t \in \Di})$.
  Let $(t, F) \in \VVDi$ be arbitrary.
  By Lemma~\ref{lem:PeriodicProps}, $(t, F)$ is $f$-stabilizing and weakly $0$-gluing.
  Let then $x \in \fix{F}{X_t}$ be arbitrary, and let $\epsilon > 0$.
  By the weakly dense homoclinic points of $\X$, there exist $t' \geq t$ and a homoclinic point $y \in X_{t'}$ with $d_X(x, y) < \epsilon / 2$.
  This is visualized in the left half of Figure~\ref{fig:PeriodDenseFinites}.
  Since $(\X, f) \in \mathfrak{C}_3$, the group $G_{w(t')}$ is residually finite.
  By Lemma~\ref{lem:Periodization}, there exists $F' \in \VVDi_{w(t')}$ and a homoclinic point $z \in \fix{F'}{X_{w(t')}}$ with $d_X(y, z) < \epsilon / 2$, and we have $d_X(x, z) < \epsilon$.
  Thus $\fix{K}{\X}$ has weakly dense homoclinic points.
\end{proof}

Recall Example~\ref{ex:Fix}, where we applied the $\mathrm{Fix}$-construction to a $\Z^2$-subshift, obtaining a tier with acting group $\Z \times \Z_p$ for each $p \geq 1$.
In that example, we could also have used the horizontal translation action of $\Z$ on every tier, because the normal subgroup $\{0\} \times \Z$ along which we periodize is a direct summand. When dealing with group extensions that do not split, using a chain of groups seems unavoidable. We illustrate this by an example.

\begin{example}[Need for multiple groups]
Consider the discrete Heisenberg group $G = \langle x, y \;|\; x [x,y] = [x,y] x, y [x,y] = [x,y] y\rangle$, let $X = \{0, 1\}^G$ be the binary full shift on $G$, and let $\X = \mathcal{U}(G, X) \in \mathfrak{C}_3$ be the associated single-tier system.
Denote the commutator by $z = [x,y]$, and let $K = \langle z \rangle \simeq \Z$.
Then $K = Z(G)$ is normal in $G$, so we can periodize along it.
Corollary~\ref{cor:PeriodicC2} implies that $\fix{K}{\X} \in \mathfrak{C}_2$.
In this example, we have $G/K \simeq \Z^2$, and the extension is not split.
The finite-index subgroups of $K$ are exactly $K_p = \langle z^p \rangle$ for $p \geq 1$, each of which is normal in $G$, so that the directed set $\VVDi$ is isomorphic to $\{1, 2, \ldots\}$ with the divisibility relation.
The group acting on the tier $\fix{K_p}{X}$ is $G/K_p \simeq \Z_p \rtimes \Z^2$, and in fact $(G/K_p, \fix{K_p}{X})$ is isomorphic to the binary full shift on $\Z_p \rtimes \Z^2$ as a pointed dynamical system, since $X$ is the full shift on $G$.
\tqed
\end{example}

We will now use the above properties of the system $\fix{K}{\X}$ to prove that if the quotient $G/K$ is nil-rigid on $\mathfrak{C}_2$, then an asymptotically nilpotent evolution map cannot transfer information arbitrarily far, except along the subgroup $K$.

\begin{lemma}
\label{lem:Towers}
Let $K \trianglelefteq G$ be two groups such that $G/K$ is nil-rigid on $\mathfrak{C}_2$, and let $(X, (G_t, X_t)_{t \in \Di}, f) \in \mathfrak{C}_3$ be an asymptotically nilpotent ETDS with $G_0 = G$.
Then for all $\epsilon > 0$ and arbitrarily large $t \in \Di$, there exists a finite set $R \subseteq G/K$ such that for all $x \in X_t$ with $\Theta^K_t(x) \cap R = \emptyset$, we have $d_X(f^n(x), 0) \leq \epsilon$ for every $n \in \N$.
\end{lemma}

\begin{proof}
  Denote $\X = (X, (G_t, X_t)_{t \in \Di})$.
  Assume on the contrary that the conclusion does not hold: for some $\epsilon > 0$, all large enough $t \in \Di$ and all finite $R \subseteq G/K$, there exists $x \in X_t$ and $n \in \N$ with $\Theta^K_t(x) \cap R = \emptyset$ and $d_X(f^n(x), 0) > \epsilon$.
  We will inductively construct a sequence $(t_k, x_k, n_k)_{k \geq 1}$, where $t_k \leq t_{k+1} \in \Di$, $x_k \in X_{t_k}$ is mortal, $\Theta^K_{t_k}(x_k)$ is finite, $n_k < n_{k+1} \in \N$ and $d_X(f^{n_m}(x_k), 0) > (1/2 + 2^{-k}) \epsilon$ holds for all $k \geq 1$ and $1 \leq m < k$.
  Then any limit point $x \in X$ of the $x_k$ will satisfy $d_X(f^{n_m}(x), 0) \geq \epsilon/2$ for all $m \in \N$, contradicting the asymptotic nilpotency of $f$ on the ambient space $X$.
  
  The idea of the construction is this.
  The shadow of each point $x_k$ spreads near the origin at the designated time steps $n_1, \ldots, n_{k-1}$, and $x_k$ dies off at some later time step $n$, since it is assumed mortal.
  Using our assumption, we find a point $y$ whose shadow stays far from that of $x_k$ for the first $n$ steps of its $f$-orbit and then spreads near the origin at some later time step $n_k$.
  We glue $x_k$ to $y$, extract a homoclinic point very close to the result, and periodize it using Lemma~\ref{lem:Periodization}.
  Since $G/K$ is nil-rigid on $\mathfrak{C}_2$, the periodized point is mortal, and we choose it as $x_{k+1}$.
  
  We start by defining $t_1$ large enough that the counterassumption holds for all $t \geq t_1$, and $x_1 = 0_X \in X_{t_1}$, which is trivially mortal and has a finite shadow.
  Suppose then that $t_k \in \Di$, $x_k \in X_t$ and $n_1 < \cdots < n_{k-1} \in \N$ have already been defined.
  Since $x_k$ is mortal, there exists $n > n_{k-1}$ such that $f^n(x_k) = 0$.
  Let $t = w(s(t_k))$, and let $\Delta_t$ be the canonical expansivity constant of $(G_t, X_t)$.
  Let $\eta \leq \min (\epsilon / 2^{k+2}, \Delta_t / 2)$, and let $\delta > 0$ and $E \subseteq G_t$ be given by the weak $0$-gluing property of $\X$ for $s(t_k)$ and $\eta$.
  Let $M \subseteq G_t$ be a finite set such that for $x \in X_t$, the condition $d_X(g \cdot x, 0) \leq \Delta_t$ for all $g \in M$ implies $d_X(x, 0) < \delta$.
  We may assume $1_{G_t} \in M$.
  
  For $x \in X_{t_k}$ and $m \in \N$, define
  \[
    A_m(x) = \bigcup_{q \leq m} E M^{-1} \cdot \Theta_t(f^q(x)) \subseteq G_t
  \]
  Intuitively, $A_m(x)$ is the set of coordinates in $G_t$ that the $f$-trajectory of $x$ affects in its first $m$ steps, plus the finite-width border $E M^{-1}$.
  Let $g \in \partial_E A_m(x)$, so that we have $e g \notin A_m(x)$ for some $e \in E$. 
  We claim that $d_X(g \cdot f^q(x), 0) < \delta$ for all $q \leq m$.
  Namely, we have $g \notin M^{-1} \cdot \Theta_t(f^q(x))$.
  For all $h \in M$, we then have $h g \notin \Theta_t(f^q(x))$, which implies $d_X(h g \cdot f^q(x), 0) \leq \Delta_t$.
  Then $d_X(g \cdot f^q(x), 0) < \delta$ by the definition of $M$, which is what we claimed.
  
  Let $N \subseteq G_t$ be a finite set such that for $x \in X_t$, the condition $d_X(f(x), 0) > \Delta_t$ implies $d_X(g \cdot x, 0) > \Delta_t$ for some $g \in N$; then we have $\Theta_t(f(x)) \subseteq (N^{-1}) \Theta_t(x)$.
  Let also $V \subseteq G_t$ be a finite set such that for $x \in X_t$, the condition $d_X(x, 0) > \epsilon$ implies $d_X(g \cdot x, 0) > \Delta_t$ for some $g \in V$.
  Denote
  \[
    R = \bigcup_{m \leq n} \psi_t(N^m \cdot (M E^{-1} A_n(x_k) \cup V)) \subseteq G/K
  \]
  which is a finite set.
  By our assumptions, there exists a point $y \in X_{t_k}$ and $p \in \N$ such that $d_X(f^p(y), 0) > \epsilon$ and $R \cap \Theta^K_t(y) = \emptyset$ (the assumptions concern the set $\Theta^K_{t_k}(y)$, but we can consider $\Theta^K_t(y)$ instead by Remark~\ref{rem:D}).
  
  Since $d_X(f^p(y), 0) > \epsilon$, there exists $g \in V$ with $d_X(g \cdot f^p(y), 0) > \Delta_t$, hence $g \in \Theta_t(f^p(y))$.
  Applying $\psi_t$, we have $\psi_t(g) \in \Theta^K_t(f^p(y)) \subseteq \psi_t((N^{-1})^p) \cdot \Theta^K_t(y)$, which implies $\psi_t(N^p \cdot V) \cap \Theta^K_t(y) \neq \emptyset$.
  From this and $R \cap \Theta^K_t(y) = \emptyset$ we deduce $p > n$.
  
  We now claim that $A_n(x_k) \cap A_n(y) = \emptyset$.
  Assume the contrary.
  Applying $\psi_t$, we obtain
  \[
    \psi_t(A_n(x_k)) \cap \psi_t(E M^{-1}) \cdot \Theta^K_t(f^m(y)) \neq \emptyset
  \]
  for some $m \leq n$.
  Since $\Theta^K_t(f^m(y)) \subset \psi_t((N^{-1})^m) \cdot \Theta^K_t(y)$, we have
  \[
    \psi_t(N^m M E^{-1} A_n(x_k)) \cap \Theta^K_t(y) \neq \emptyset
  \]
  which contradicts the condition $R \cap \Theta^K_t(y) = \emptyset$.
  
  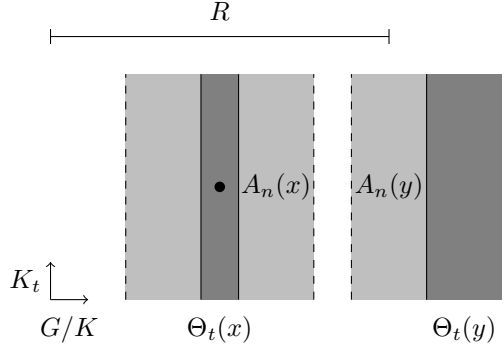
\begin{figure}
  \begin{center}
  \begin{tikzpicture}
  
    \fill [black!25] (1,0) rectangle (2,3);
    \fill [black!50] (2,0) rectangle (2.5,3);
    \fill [black!25] (2.5,0) rectangle (3.5,3);
    
    \fill [black!25] (4,0) rectangle (5,3);
    \fill [black!50] (5,0) rectangle (6,3);
    
    \draw [dashed] (1,0) -- ++(0,3);
    \draw (2,0) -- ++(0,3);
    \draw (2.5,0) -- ++(0,3);
    \draw [dashed] (3.5,0) -- ++(0,3);
    \draw [dashed] (4,0) -- ++(0,3);
    \draw (5,0) -- ++(0,3);
    
    \draw (0,3.5) -- (4.5,3.5);
    \draw (0,3.4) -- ++ (0,0.2);
    \draw (4.5,3.4) -- ++ (0,0.2);
    
    \fill (2.25,1.5) circle (0.07cm);
    
    \node [below] at (2.25,-0.1) {$\Theta_t(x)$};
    \node [below] at (5.5,-0.1) {$\Theta_t(y)$};
    \node at (3,1.5) {$A_n(x)$};
    \node at (4.5,1.5) {$A_n(y)$};
    \node [above] at (2.25,3.6) {$R$};
    
    \draw [->] (0,0) -- (0.5,0);
    \draw [->] (0,0) -- (0,0.5);
    \node [left] at (0,0.25) {$K_t$};
    \node [below] at (0.25,-0.1) {$G/K$};
  
  \end{tikzpicture}
  \end{center}
  \caption{The construction of the points $y$ and $z_0$ in Lemma~\ref{lem:Towers}. The dark regions are $\Theta_t(x)$ and $\Theta_t(y)$, and the light regions are $A_n(x)$ and $A_n(y)$. The black circle is $1_{G_t}$.}
  \label{fig:PeriodTower}
  \end{figure}
  
  We have now shown that $((A_n(x_k), f^m(x_k)), (A_n(y), f^m(y)))$ is a two-element $(\delta, E)$-designation in $(G_t, X_t)$ for each $m \leq n$, and both points belong to $X_{s(t_k)}$.
  By our choice of $\delta$ and $E$, the designation has an $\eta$-realization $z_m \in X_t$.
  Furthermore, if $\eta$ is chosen small enough, Lemma~\ref{lem:FuncGlue} implies $f(z_m) = z_{m+1}$ for all $m < n$; note that $n$ does not depend on $\eta$.
  We claim that $z_n = f^n(y)$, and for that, let $g \in G_t$ be arbitrary.
  If $g \in A_n(y)$, then $d_X(g \cdot z_n, g \cdot f^n(y)) < \eta \leq \Delta_t$.
  If $g \in A_n(x_k)$, then $d_X(g \cdot z_n, g \cdot f^n(y)) \leq d_X(g \cdot z_n, g \cdot f^n(x_k)) + d_X(0, g \cdot f^n(y)) < \eta + \Delta_t \leq \frac{3}{2} \Delta_t$ since $f^n(x_k) = 0$ and $g \notin \Theta_t(f^n(y))$.
  The same computation proves $d_X(g \cdot z_n, g \cdot f^n(y)) \leq \frac{3}{2} \Delta_t$ when $g \notin A_n(y) \cup A_n(x_k)$, and then we have $z_n = f^n(y)$ since $\frac{3}{2} \Delta_t$ is an expansivity constant of $(G_t, X_t)$.
  
  By extracting a homoclinic point close to $z_0$ and applying Lemma~\ref{lem:Periodization} to it, we can show that there exist $(t', F') \in \VVDi$ with $t' \geq t$ and a point $z \in \fix{F'}{X_{t'}}$ such that $d_X(f^m(z_0), f^m(z)) < \eta$ for all $m \leq p$.
  It is also guaranteed that $z$ is homoclinic in $(G_{t'}/F', \fix{F'}{X_{t'}})$, so that $\Theta^K_{t'}(z)$ is finite.
  We know from Corollary~\ref{cor:PeriodicC2} that $\fix{K}{\X} \in \mathfrak{C}_2$.
  Since $f$ is asymptotically nilpotent and $G/K$ is nil-rigid on $\mathfrak{C}_2$, the point $z$ is mortal.
  For $1 \leq m < k$, we have
  \begin{align*}
    {} & d_X(f^{n_m}(z), 0) \\
    \geq {} &
    d_X(f^{n_m}(x_k), 0) - d_X(f^{n_m}(z_0), f^{n_m}(x_k)) - d_X(f^{n_m}(z), f^{n_m}(z_0)) \\
    {} \geq {} & (1/2 + 2^{-k}) \epsilon - \eta - \eta > (1/2 + 2^{-k-1}) \epsilon
  \end{align*}
  Also, $d_X(f^p(z), 0) \geq d_X(f^p(z_0), 0) - d_X(f^p(z), f^p(z_0)) > \epsilon - \eta > (1/2 + 2^{-k-1}) \epsilon$.
  Thus we can choose $t_{k+1} = t'$, $x_{k+1} = z$ and $n_k = p > n > n_{k-1}$ and continue the induction.
  As stated, a limit point of the $x_k$ contradicts the asymptotic nilpotency of $f$.
\end{proof}

\begin{remark}
  \label{rem:PeriodTowers}
  The assumptions of Lemma~\ref{lem:Towers} can be slightly weakened: instead of requiring that $(X, (G_t, X_t)_{t \in \Di}, f) \in \mathfrak{C}_3$, it is enough for the ETDS to be $f$-stabilizing and weakly $0$-gluing, and the closure of the set of homoclinic points of $\fix{K}{\X}$ to contain each tier of $\X$.
  The only modification is that in the final paragraph, instead of using Lemma~\ref{lem:Periodization}, we can directly extract $z$ from $z_0$ using this new assumption.
\end{remark}

With essentially the same proof, we can also show that if all homoclinic points are mortal, the system is nilpotent.
This is our main method of actually proving the nilpotency of evolution maps when the acting group is not virtually cyclic.
The argument was already used in \cite{GuRi08} as part of the proof that asymptotically nilpotent cellular automata on one-dimensional full shifts are nilpotent.

\begin{lemma}
\label{lem:Holes}
Let $X$ be a pointed compact metric space and $f : X \to X$ an asymptotically nilpotent continuous function.
For each $\epsilon > 0$, there exists $n \in \N$ such that for all $x \in X$ we have $d_X(f^k(x), 0_X) < \epsilon$ for some $k \leq n$.
\end{lemma}

\begin{proof}
For $n \in \N$, let $X_n = \{ x \in X \;|\; d_X(f^n(x), 0_X) < \epsilon \}$.
Each $X_n$ is open, and since $f$ is asymptotically nilpotent, they form a cover of $X$.
By compactness we have a finite subcover.
\end{proof}

\begin{lemma}
\label{lem:Mortal}
Let $G$ be a group, and let $(X, (G_t, X_t)_{t \in \Di}, f) \in \mathfrak{C}_2$ be an ETDS with $G_0 = G$.
Suppose that for all $t \in \Di$, every homoclinic point $x \in X_t$ is mortal.
Then $f$ is nilpotent on each tier $X_t$.
\end{lemma}

\begin{proof}[Proof sketch]
Assume that $f$ is not nilpotent on some tier $t \in \Di$.
The proof proceeds as in Lemma~\ref{lem:Towers}, with the following main differences.
First, we consider the $t$-shadows $\Theta_t(x)$ instead of the $(K,t)$-shadows $\Theta^K_t(x)$.
The points with finite $t$-shadow are exactly the homoclinic points of $(G_t, X_t)$.

Second, the existence of the point $y$ is proved as follows, once the set $R \subseteq G_t$ has been defined.
Let $\eta > 0$ be such that $d_X(x, 0_X) < \eta$ implies $R \cap \Theta_t(x) = \emptyset$ for each $x \in X_t$.
Let $n \in \N$ be given for $\eta$ by Lemma~\ref{lem:Holes}.
Since $f$ is not nilpotent on $X_t$, there exist $m > n$ and $x \in X_t$ with $d_X(f^m(x), 0_X) > \epsilon$.
For some $k \leq n$ we have $d_X(f^k(x), 0_X) < \eta$, and we choose $y = f^k(x)$ and $p = m - k$.

Third, the existence of a mortal homoclinic point $z \in X_{t'}$ follows from the weakly dense homoclinic points of $\X$ and the assumption that all homoclinic points are mortal.
\end{proof}

\begin{remark}
  \label{rem:Mortal}
  In the case of a single-tier ETDS $\mathcal{U}(G, X, f)$, we can slightly weaken the assumptions of Lemma~\ref{lem:Mortal}: instead of $\mathcal{U}(G, X, f) \in \mathfrak{C}_2$, it is enough that $(G, X)$ is $0$-gluing and for some $p \in \N$, the image system $(G, f^p(X))$ has weakly dense homoclinic points.
  In the proof, we choose the point $y$ from $f^p(X)$, and then extract a homoclinic point $y' \in f^p(X)$ close to $y$, using it in place of $y$.
  Since $x_k$ is homoclinic and we are operating in a single tier, any point close enough to $y$ can be glued to $x_k$.
  Then the point $z_0 \in X$ is already homoclinic.
\end{remark}

\section{Finite Points and Group Extensions}
\label{sec:Finites}

In this section, we consider the set of points with a finite $(K,t)$-shadow, and show that under the conditions of Lemma~\ref{lem:Towers}, they form an expansive system under the action of $K$.
In the construction, the tiers of the subsystem are formed by those points whose $(K,t)$-shadow is contained in a given finite subset of $G/K$.
This TDS can of course be constructed with no assumption on $f$, but the asymptotic nilpotency of $f$ and Lemma~\ref{lem:Towers} guarantee that $f$ is stabilizing on these tiers, and the tiers allow us to capture the set of all points with finite shadow, even though it may not be compact, so that the resulting system is in $\mathfrak{C}_3$.
We then use this system to show that nil-rigidity is (in a restricted sense) preserved under group extensions.

Suppose we have two groups $K \trianglelefteq G$ and a TDS $\X = (X, (G_t, X_t)_{t \in \Di})$ with $G_0 = G$, and denote $H = G/K$.
Recall from Remark~\ref{rem:D} the finite sets $D^{t'}_t \subseteq G_t$ for $t \leq t' \in \Di$, which satisfy $\phi^{t'}_t(\Theta_{t'}(x)) \subseteq D^{t'}_t \cdot \Theta_t(x)$ for all $x \in X_t$.
Recall also the functions $\psi^K_t : G_t \to H$ from Definition~\ref{def:KShadow}.
Define the directed set $\VVDi = \{ (t, E) \;|\; t \in \Di, E \subseteq H \text{~finite} \}$ with the order relation given by $(t, E) \leq (t', E')$ iff $t \leq t'$ and $\psi^K_t(D^{t'}_t) E \subseteq E'$, so that $(t_0, \emptyset)$ is the least element of $\VVDi$.
Note that in order for ${\leq}$ to be reflexive, we need to choose $D^t_t = \{1_{G_t}\}$ for all $t \in \Di$, which we can safely do.
For each $(t, E) \in \VVDi$, define $X_{(t,E)} = \{ x \in X_t \;|\; \Theta^K_t(x) \subseteq E \}$, which is a closed subset of $X_t$.
Then $(X, (K_t, X_{(t,E)})_{(t,E) \in \VVDi})$ is a tiered dynamical system, with the notation $K_t = (\phi^t_{t_0})^{-1}(K)$ from Section~\ref{sec:Periods}:
If $(t, E) \leq (t', E')$, then $X_{(t,E)} \subseteq X_{(t',E')}$ by the definition of the order relation on $\VVDi$.
The homomorphisms between the groups $K_t$ are restrictions of the maps $\phi^{t'}_t$, and satisfy the required properties.
For expansivity, let $\Delta_t > 0$ be the canonical expansivity constant of $(G_t, X_t)$, and let $L_t \subseteq G_t$ be a set of representatives for the left cosets of $K_t$ in $G_t$.
Let also $1_{G_t} \in N \subseteq G_t$ be a finite set such that for $x \in X_t$, the condition $d_X(g^{-1} \cdot x, 0) \leq \Delta_t$ for all $g \in N$ implies $d_X(x, 0) \leq \Delta_t / 2$.
For a finite $E \subseteq H$, there exists $\epsilon > 0$ such that $d_X(x, y) \leq \epsilon$ implies $d_X(g \cdot x, g \cdot y) \leq \Delta_t$ for all $g \in N (\psi_t^K)^{-1}(E) \cap L_t$.
It can be verified that $\epsilon$ is an expansivity constant for $(K_t, X_{(t, E)})$.

\begin{definition}
\label{def:FiniteSystem}
Let $\X = (X, (G_t, X_t)_{t \in \Di}, (\phi^{t'}_t)_{t \leq t'})$ be a TDS and $K \trianglelefteq G_0$ a normal subgroup.
Then the TDS $(X, (K_t, X_{(t,E)})_{(t,E) \in \VVDi})$ defined above is denoted by $\fin{K}{\X}$.
\end{definition}

\begin{lemma}
  \label{lem:ShadowFinites}
  Let $(\X, f) \in \mathfrak{C}_3$ be an asymptotically nilpotent ETDS with base group $G$, and let $K \trianglelefteq G$ be such that $G/K$ is nil-rigid on $\mathfrak{C}_2$.
  Then $(\fin{K}{\X}, f) \in \mathfrak{C}_3$.
\end{lemma}

\begin{proof}
  Denote $\X = (X, (G_t, X_t)_{t \in \Di}, (\phi^{t'}_t)_{t \leq t'})$.
  We first note that as a subgroup of $G_t$, each $K_t$ is residually finite.
  For $t \in \Di$, let $L_t \subseteq G_t$ be a set of representatives for the left cosets of $K_t$ in $G_t$.
  For a set $E \subset H$, denote $L^E_t = L_t \cap /\psi_t^K)^{-1}(E)$.
  Then $L^E_t$ is finite whenever $E$ is, in particular when $(t, E) \in \VVDi$.
  Fix $(t, E) \in \VVDi$.
  
  First, we prove the weak $0$-gluing property of $\fin{K}{\X}$.
  Let $M_t \subseteq G_{w(t)}$ be a finite set containing $(D^{w(t)}_t)^{-1}$ such that for $x \in X_t$, the condition $d_X(h^{-1} \cdot x, 0) \leq \Delta_t$ for all $h \in M_t$ implies $d_X(x, 0) \leq \Delta_{w(t)} / 2$.
  We claim that $w(t, E) = (w(t), \psi_t^K(M_t) E)$ is a suitable definition of the function $w$ on $\VVDi$.
  Since $(D^{w(t)}_t)^{-1} \subseteq M_t$, we have $(w(t), \psi_t^K(M_t) E) \geq (t, E)$ in the directed set $\VVDi$.
  Let then $0 < \epsilon \leq \Delta_{w(t)} / 2$ be arbitrary, and let $\delta > 0$ and $F \subseteq G_{w(t)}$ be given by the weak $0$-gluing property of $\X$ for $t$ and $\epsilon$.
  We may assume that $1_{G_{w(t)}} \in F$.
  Let $N \subseteq H$ be a finite set such that for all $x \in X_{(t, E)}$ and $g \in G_{w(t)} \setminus \bigcap_{e \in F} e^{-1} \cdot (\psi_{w(t)}^K)^{-1}(N)$ we have $d_X(g \cdot x, 0) < \delta$.
  We can find $N$ with Lemma~\ref{lem:Nhood}.
  Let $\eta > 0$ be such that for all $x \in X_{w(t)}$ and $g \in L^N_{w(t)}$, the condition $d_X(x, 0) < \eta$ implies $d_X(g \cdot x, 0) < \delta$.
  Define $D = (L^N_{w(t)})^{-1} F L^N_{w(t)} \cap K_{w(t)}$.
  Let $\mathcal{A} = (A_i, x_i)_{i \in I}$ be an $(\eta, D)$-designation in the tier $(K_{w(t)}, X_{w(t), \psi_t^K(M_t) E})$ such that $x_i \in X_{(t, E)}$ for all $i \in I$.
  Note that $(A_i)_{i \in I}$ is a partition of $K_{w(t)}$.

  For $i \in I$, define $B_i = L_{w(t)} A_i$, so that $(B_i)_{i \in I}$ forms a partition of $G_{w(t)}$.
  We claim that $(B_i, x_i)_{i \in I}$ is a $(\delta, F)$-designation in $(G_{w(t)}, X_{w(t)})$.
  For that, let $g a \in \partial_F B_i$ with $g \in L_{w(t)}$ and $a \in A_i$.
  Our aim is to prove
  \begin{equation}
    \label{eq:gaDist}
    d_X(g a \cdot x_i, 0) < \delta
  \end{equation}
  First, if $g \notin L^N_{w(t)}$, then $g a \notin (\psi_{w(t)}^K)^{-1}(N)$, which implies~\eqref{eq:gaDist} by our choice of $N$ and because $1_{G_{w(t)}}$.
  Suppose then that $g \in L^N_{w(t)}$.
  Since $g a \in \partial_F B_i$, there exists $e \in F$ with $e g a \notin B_i$.
  If $e g a \notin (\psi_{w(t)}^K)^{-1}(N)$, then $g a \notin \bigcap_{e \in F} e^{-1} \cdot (\psi_{w(t)}^K)^{-1}(N)$, which again implies~\eqref{eq:gaDist}.
  We thus assume $e g a \in (\psi_{w(t)}^K)^{-1}(N)$.
  Let $h \in L^N_{w(t)}$ be such that $h^{-1} e g a \in K_{w(t)}$, which implies $h^{-1} e g \in D$ since $a \in K_{w(t)}$.
  Also, we have $h^{-1} e g a \notin A_i$, since $e g a \notin B_i$.
  This means that $a \in \partial_D A_i$, and since $\mathcal{A}$ is an $(\eta, D)$-designation, we obtain $d_X(a \cdot x, 0) < \eta$.
  This is turn implies~\eqref{eq:gaDist} by our choice of $\eta$.
  The above argument is visualized in Figure~\ref{fig:FiniteGluing}.
  
  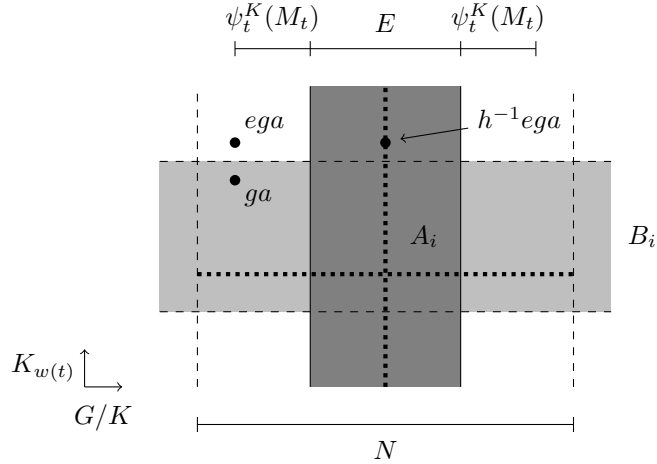
\begin{figure}[ht]
  \begin{center}
  \begin{tikzpicture}
    
    \coordinate (dE) at (-0.5,0);
    \coordinate (dW) at (4.5,0);
    \coordinate (uE) at (-0.5,4);
    \coordinate (uW) at (4.5,4);
    
    \coordinate (dL) at (0,0);
    \coordinate (dl) at ($(dL) + (1,0)$);
    \coordinate (dc) at ($(dl) + (1,0)$);
    \coordinate (dr) at ($(dc) + (1,0)$);
    \coordinate (dR) at ($(dr) + (1,0)$);
    \coordinate (uL) at (0,4);
    \coordinate (ul) at ($(uL) + (1,0)$);
    \coordinate (uc) at ($(ul) + (1,0)$);
    \coordinate (ur) at ($(uc) + (1,0)$);
    \coordinate (uR) at ($(ur) + (1,0)$);

    \fill [black!25] ($(dL) + (-1,1)$) rectangle ($(uR) + (1,-1)$);    
    \fill [black!50] (dl) rectangle (ur);
    
    \draw (dl) -- (ul);
    \draw (dr) -- (ur);
    
    \draw [dashed] (dE) -- (uE);
    \draw [ultra thick, dotted] (dc) -- (uc);
    \draw [dashed] (dW) -- (uW);
    
    \draw ($(uL) + (0,0.4)$) -- ++(0,0.2);
    \draw ($(ul) + (0,0.4)$) -- ++(0,0.2);
    \draw ($(ur) + (0,0.4)$) -- ++(0,0.2);
    \draw ($(uR) + (0,0.4)$) -- ++(0,0.2);
    \draw ($(uL) + (0,0.5)$) -- ($(uR) + (0,0.5)$);
    
    \node [above] at ($(uL) + (0.5,0.6)$) {$\psi_t^K(M_t)$};
    \node [above] at ($(uc) + (0,0.6)$) {$E$};
    \node [above] at ($(ur) + (0.5,0.6)$) {$\psi_t^K(M_t)$};
    
    
    \draw [dashed] ($(dL) + (-1,1)$) -- ($(dR) + (1,1)$);
    \draw [dashed] ($(uL) + (-1,-1)$) -- ($(uR) + (1,-1)$);
    
    \node [right] at ($(dR) + (1.1,2)$) {$B_i$};
    \node [right] at ($(dc) + (0.2,2)$) {$A_i$};
    
    \draw ($(dL) + (-0.5,-0.4)$) -- ++(0,-0.2);
    \draw ($(dR) + (0.5,-0.4)$) -- ++(0,-0.2);
    \draw ($(dL) + (-0.5,-0.5)$) -- ($(dR) + (0.5,-0.5)$);
    \draw [ultra thick,dotted] ($(dL) + (-0.5,1.5)$) -- ($(dR) + (0.5,1.5)$);
    
    \node [below] at ($(dc) + (0,-0.6)$) {$N$};
    
    \fill (0,2.75) circle (0.07);
    \node [below right] at (0,2.75) {$g a$};
    \fill (0,3.25) circle (0.07);
    \node [above right] at (0,3.25) {$e g a$};
    \fill (2,3.25) circle (0.07);
    \node (x) [above right] at (3.1,3.25) {$h^{-1} e g a$};
    \draw [->] (x) -- (2.2,3.3); 
    
    \coordinate (cmp) at (-2,0);
    \draw [->] (cmp) -- ++(0.5,0);
    \draw [->] (cmp) -- ++(0,0.5);
    \node [left] at ($(cmp) + (0,0.25)$) {$K_{w(t)}$};
    \node [below] at ($(cmp) + (0.25,-0.1)$) {$G/K$};
  
  \end{tikzpicture}
  \end{center}
  \caption{The construction of the designation $(B_i, x_i)_{i \in I}$ in Lemma~\ref{lem:ShadowFinites}. The vertical thick dotted line is $K_{w(t)}$, the horizontal thick dotted line is $L^N_{w(t)}$, and each set $A_i$ is a subset of $K_{w(t)}$. The shadows of the $x_i$ are contained in the dark gray region. The light gray region is $B_i$.}
  \label{fig:FiniteGluing}
  \end{figure}
  
  We have that $(B_i, x_i)_{i \in I}$ is a $(\delta, F)$-designation in $(G_{w(t)}, X_{w(t)})$, and since each $x_i \in X_{(t,E)} \subseteq X_t$, it has an $\epsilon$-realization $y \in X_{w(t)}$.
  We claim that $y \in X_{w(t), \psi_t^K(M_t) E}$, and for that, let $g \in G_{w(t)}$ be such that $\psi_t(g) \notin \psi_t^K(M_t) E$.
  Then for each $h \in M_t$ and $i \in I$ we have $h^{-1} g \notin (\psi_t^K)^{-1}(E)$, which implies $d_X(h^{-1} g \cdot x_i, 0) \leq \Delta_t$ since $\Theta^K_t(x_i) \subseteq E$.
  By our choice of $M_t$, this implies $d_X(g \cdot x_i, 0) < \Delta_{w(t)} / 2$.
  Then we have $d_X(g \cdot y, 0) \leq d_X(g \cdot y, g \cdot x_i) + d_X(g \cdot x_i, 0) \leq \Delta_{w(t)}$ for the $i \in I$ with $g \in B_i$.
  This implies $g \notin \Theta^K_{w(t)}(y)$.
  We have shown that $y \in X_{w(t), \psi_t^K(M_t) E}$, and this proves that $\fin{K}{\X}$ is weakly $0$-gluing.

  Next, we show that $\fin{K}{\X}$ has weakly dense homoclinic points.
  Let $(t, E) \in \VVDi$, $x \in X_{(t,E)}$ and $\epsilon > 0$.
  Since $\X$ has weakly dense homoclinic points, there exists $t' \geq t$ and a homoclinic point $y \in X_{t'}$ with $d_X(x, y) < \epsilon$.
  Since $y$ is homoclinic, it has a finite $(K,t')$-shadow $E' \subseteq H$, and thus $y \in X_{t', E'}$.
  We can choose $E'$ so that $(t, E) \leq (t', E')$.
  
  Consider finally the map $f$, which is clearly an evolution of $\fin{K}{\X}$.
  We claim that the system is $f$-stabilizing.
  Choose $\epsilon = \Delta_{s(t)}$ and apply Lemma~\ref{lem:Towers} to obtain a tier $t' \geq s(t)$ and a finite set $R \subseteq G / K$ such that for all $x \in X_{t'}$ with $\Theta^K_{t'}(x) \cap R = \emptyset$ we have $d_X(f^n(x), 0_X) \leq \Delta_{s(t)}$ for all $n \in \N$.
  We may assume $\psi_t^K(D^{s(t)}_t) \subseteq R^{-1}$.
  By Remark~\ref{rem:D}, we have $\Theta_{s(t)}(x) \subseteq \phi^{t'}_{s(t)}(\Theta_{t'}(x)) \subseteq D^{t'}_{s(t)} \cdot \Theta_{s(t)}(x)$ for $x \in X_{s(t)}$.
  Then $\Theta^K_{s(t)}(f^n(x)) \subseteq R^{-1} \psi_{s(t)}^K(D^{t'}_{s(t)}) \cdot \Theta^K_{s(t)}(x)$ for all $x \in X_t$ and $n \in \N$, which means that the tier $(t, E)$ is $f$-stabilizing with the choice $s(t, E) = (s(t), R^{-1} \psi_{s(t)}^K(D^{t'}_{s(t)}) E)$.
  The condition $\psi_t^K(D^{s(t)}_t) \subseteq R^{-1}$ guarantees $s(t, E) \geq (t, E)$ in the directed set $\VVDi$.
\end{proof}

In the $\mathrm{Fix}$-construction, gluing never takes us to a higher tier, so if each tier of $\X$ is $0$-gluing by itself, then the same holds for $\fix{K}{\X}$ for each $K \trianglelefteq G$. In the zero-dimensional case, this is also true for $\fin{K}{\X}$, but in the case of general expansive actions we do not have a symbolic representation, and thus do not know how to consistently pick a metric where gluing never takes us just above the threshold we use to define shadows. We give an example of a subshift where gluing can take us above the threshold, by deliberately picking a bad metric. We have no examples of expansive systems where weak $0$-gluing is needed for an essential reason.

\begin{example}[Need for weak $0$-gluing]
Consider the two-dimensional binary full shift $X = \{0, 1\}^{\Z^2}$.
For $n \in \N$, define $D_n = \{(i,j) \in \Z^2 \;|\; |i| + |j| = n\}$, and for $x, y \in X$, let $D(x,y) = \{n \in \N \;|\; \exists \vec v \in D_n : x_{\vec v} \neq y_{\vec v}\}$.
Then $d_X(x, y) = \sum_{n \in D(x,y)} 2^{-n/2}$ defines a metric in $X$ that induces the standard product topology.
If two configurations disagree only at the origin, then their $d_X$-distance is $1$, so the canonical expansivity constant of $X$ is $\Delta_X = 1/2$.
We define a single-tier system $\X = \mathcal{U}(\Z^2, X) \in \mathfrak{C}_3$.

Denote the vertical subgroup by $K = \{0\} \times \Z \leq \Z^2$, and consider the tiered system $\fin{K}{\X}$, where $\Z$ acts on every tier $X_E$ for $E \subset \Z^2/K \simeq \Z$ by vertical translations.
By Lemma~\ref{lem:ShadowFinites}, $\fin{K}{\X}$ is weakly $0$-gluing.
On the other hand, we can show that for $n \geq 3$, no tier $(\Z, X_{[-n, n]})$ is $0$-gluing by itself.
Let $m \geq 0$, and define configurations $x, y \in X$ by $x_{(n-2, -2 k)} = 1$ for $k \in \{0, 1, \ldots, m-1\}$, and $y_{(n-2, 2 m + 2 k + 1)} = 1$ for all $k \in \{0, 1, \ldots, m-1\}$, and put $0$ everywhere else.
Then $x, y \in X_{[-n, n]}$ since $d_X(\vec v \cdot x, 0_X) \leq \sum_{k = 2}^{m + 1} 2^{-k} < 1/2$ for all $\vec v \notin [-n, n] \times \Z$, and similarly for $y$.
Define $z \in X$ as the cellwise maximum of $x$ and $y$, which is also the $\epsilon$-realization of a suitable designation in $\fin{K}{\X}$ containing $x$ and $y$, where $\epsilon$ can be chosen arbitrarily small depending on $m$.
Then we have $d_X((0, n+1) \cdot z, 0_X) = (\sum_{k=0}^{m-1} 2^{-(2 k + 3)/2}) + (\sum_{k=0}^{m-1} 2^{-m - k - 1}) > 1/2$, which implies $(0, n+1) \in \Theta(\Delta_X, z)$, and thus $z \notin X_{[-n, n]}$.
\tqed
\end{example}

We are now able to prove that a residually finite extension of a group that is nil-rigid on $\mathfrak{C}_2$ by a group that is nil-rigid on $\mathfrak{C}_3$ is itself nil-rigid on $\mathfrak{C}_3$.
For this, we fix an exact sequence $1 \to K \to G \stackrel{\psi}{\to} H \to 1$, where $G$ is residually finite and $K$ and $H$ are nil-rigid on $\mathfrak{C}_3$ and $\mathfrak{C}_2$ respectively.
We assume that $K \trianglelefteq G$, and $H = G / K$, so that $K = \mathrm{Ker}(\psi)$.
We also fix an ETDS $(\X, f) \in \mathfrak{C}_3$ with bottom group $G$ such that $f$ is asymptotically nilpotent, and denote $\X = (X, (G_t, X_t)_{t \in \Di})$.

Consider the system $\fin{K}{\X}$.
Since the quotient group $H$ is nil-rigid on $\mathfrak{C}_2$, we can apply Lemma~\ref{lem:ShadowFinites} and obtain $\fin{K}{\X} \in \mathfrak{C}_3$.
Since the group $K$ is nil-rigid on $\mathfrak{C}_3$, we know that $f$ is nilpotent on each tier $X_{(t, E)}$ of $\fin{K}{\X}$.
Since each homoclinic point $x \in X_t$ has a finite $(K,t)$-shadow, it belongs to one of these tiers, and is thus mortal.
We can then apply Lemma~\ref{lem:Mortal} to show that $f$ is nilpotent on each tier $X_t$.
We have proved the following result.

\begin{theorem}
\label{thm:BigClosure}
  Suppose we have an exact sequence $1 \to K \to G \to H \to 1$ of groups such that $G$ is residually finite, $K$ is nil-rigid on $\mathfrak{C}_3$, and $H$ is nil-rigid on $\mathfrak{C}_2$.
  Then $G$ is nil-rigid on $\mathfrak{C}_3$.
\end{theorem}

With a similar idea, we show the nil-rigidity of potentially infinite direct sums of groups that are nil-rigid on $\mathfrak{C}_2$.

\begin{proposition} 
  \label{prop:DirectSum}
  Let $G$ be a residually finite direct sum of groups that are nil-rigid on $\mathfrak{C}_2$.
  Then $G$ is nil-rigid on $\mathfrak{C}_3$.
\end{proposition}

\begin{proof}
  Let $G = \bigoplus_{i \in I} H_i$, where each $H_i$ is nil-rigid on $\mathfrak{C}_2$.
  Let $(\X, f) \in \mathfrak{C}_3$ be an asymptotically nilpotent ETDS with $\X = (X, (G_t, X_t)_{t \in \Di})$ and $G_0 = G$.
  For each $i \in I$, let $K_i \trianglelefteq G$ be the kernel of the projection map $\pi_i : G \to H_i$, so that we have an exact sequence $1 \to K_i \to G \to H_i \to 1$.
  Consider the system $\fix{K_i}{\X} \in \mathfrak{C}_2$ whose base group is $H_i$.
  Since $H_i$ is nil-rigid on $\mathfrak{C}_2$, we see as in Lemma~\ref{lem:ShadowFinites} that for each $t \in \Di$ there exists a finite set $R^i_t \subseteq H_i$ such that $\Theta^{K_i}_{s(t)}(f^n(x)) \subseteq R^i_t \cdot \Theta^{K_i}_{s(t)}(x)$ for all $x \in X_t$ and $n \in \N$.
  
  For each $t \in \Di$, there exists a finite set $N_t \subseteq G_{s(t)}$ such that for all $x \in X_t$ and $n \in \N$ we have $\Theta_{s(t)}(f^n(x)) \subseteq N^n_t \cdot \Theta_{s(t)}(x)$.
  Fix a homoclinic point $x \in X_t$, and denote $F = \langle N_t \cdot \Theta_{s(t)}(x) \rangle \leq G_{s(t)}$.
  Then $\Theta_{s(t)}(f^n(x)) \subseteq F$ for all $n \in \N$.
  Since the image $\phi^{s(t)}_{t_0}(F) \leq G$ is finitely generated, there exists a finite set $J \subseteq I$ with $\phi^{s(t)}_{t_0}(F) \leq \bigoplus_{i \in J} H_i$.
  For each $i \in J$ and $n \in \N$, we have $\pi_i(\phi^{s(t)}_{t_0}(\Theta_{s(t)}(f^n(x)))) = \Theta^{K_i}_{s(t)}(f^n(x)) \subseteq R^i_t \cdot \Theta^{K_i}_{s(t)}(x) \subseteq H_i$.
  This implies
  \[
    \Theta_{s(t)}(f^n(x)) \subseteq (\phi^{s(t)}_{t_0})^{-1} \left( \bigoplus_{i \in J} R^i_t \cdot \Theta^{K_i}_{s(t)}(x) \right)
  \]
  The right-hand side is a finite set, since each component $R^i_t \cdot \Theta^{K_i}_{s(t)}(x)$ is finite and $\phi^{s(t)}_{t_0}$ is finite-to-one.
  Since the shadows $\Theta_{s(t)}(f^n(x))$ for $n \in \N$ are restricted to a finite set, Lemma~\ref{lem:BoundedShadow} implies $f^m(x) = f^n(x)$ for some $m < n$.
  We assumed $f$ to be asymptotically nilpotent, so $f^m(x) = 0_X$, that is, $x$ is mortal.
  Since $x$ was arbitrary, Lemma~\ref{lem:Mortal} implies that $f$ is nilpotent on each tier of $\X$.
\end{proof}

\section{Local Properties}
\label{sec:Local}

In this section, we study groups that are locally nil-rigid on some class of tiered systems, meaning that their finitely generated subgroups have this property.
Our goal is to show that such groups are nil-rigid as well.
For the class $\mathfrak{C}_4$, this follows quite easily from Lemma~\ref{lem:Nhood}.
Any finite subset of the group can be extended into a finite neighborhood of the evolution map along which the shadow of a point can spread, and let it generate a subgroup.
We consider the set of points whose shadow is contained in this subgroup, which is also a system in $\mathfrak{C}_4$.
By the nil-rigidity of the finitely generated subgroup, an asymptotically nilpotent evolution map of such a system is nilpotent.
Since every homoclinic point of the original system is contained in one of these `local' systems, it is mortal, and we apply Lemma~\ref{lem:Mortal} to conclude.

We present a slightly more general construction that can be applied to an arbitrary tiered system.
We choose a finitely generated subgroup and a finite subset of tiers, and form a new system from those.
This is used later in Section~\ref{sec:Abelian} to prove that locally virtually abelian groups are nil-rigid on $\mathfrak{C}_2$.
Like in the case of $\Z$, for a finitely generated abelian group it is enough to consider a finite number of well-behaved tiers to show that all homoclinic points are mortal, which means we can use the construction to prove the mortality of all homoclinic points of the original system.
What prevents us from extending the construction to a larger class of subgroups is that the finite neighborhood may be different on each tier, so we can only construct subsystems with finitely many tiers.

Suppose that we have an ETDS $(\X, f) \in \mathfrak{C}_2$ with $\X = (X, (G_t, X_t)_{t \in \Di})$.
We fix a finite nonempty subset $\VDi \subseteq \Di$ with a least and greatest element, the latter of which we denote by $r$, and let $H = G_r$ act on each tier $X_t$ for $t \in \VDi$ in the obvious way.
Let $\Delta$ be the canonical expansivity constant of $(H, X_r)$, and let $\Theta = \Theta_r$ be the shadow function of $(H, X_r)$, as defined in Definition~\ref{def:Shadows}.

Fix a finite set $F \subseteq H$ such that for all $x \in X_r$, the condition $d_X(g \cdot x, 0) \leq \frac{3}{2} \Delta$ for all $g \in F$ implies $d_X(x, 0) \leq \Delta$, and $\Theta(f(x)) \subseteq F \cdot \Theta(x)$ holds for all $x \in X_{s(t)} \cap f^{-1}(X_{s(t)})$ for each $t \in \VDi$ with $s(t) \in \VDi$.
If $F$ has these properties, we say that it is \emph{suitable for $\VDi$ and $f$}, or just for $\VDi$ if it satisfies the first condition.
We can extend any finite subset of $H$ into a suitable set using Lemma~\ref{lem:Nhood}, recalling that $\frac{3}{2} \Delta$ is also an expansivity constant of $(H, X_r)$.
We now construct an ETDS $(\mathcal{Y}, f)$ as follows.
Let $K_F = \langle F \rangle$ be the subgroup of $H$ generated by $F$.
For $t \in \VDi$, denote by $X_{F, t} \subseteq X_t$ the set of those $x \in X_t$ with $\Theta(x) \subseteq K_F$.
The set $X_{F,t}$ is $K_F$-invariant and closed in $X_t$.
Let $N \subseteq G_r$ be a finite set such that for $x \in X_r$, the condition $d_X(g \cdot x, 0) \leq \Delta$ for all $g \in N$ implies $d_X(x, 0) \leq \Delta/2$, and let $0 < \epsilon \leq \Delta$ be such that $d_X(x, y) < \epsilon$ implies $d_X(g^{-1} \cdot x, g^{-1} \cdot y) \leq \Delta$ for all $x, y \in X_r$ and $g \in N$.
Then $\epsilon$ is an expansivity constant for $(K_F, X_{F,t})$.
The directed set of the system is $\VDi$, its every group is $K_F$, and its tiers are $X_{F,t}$ for $t \in \VDi$.
The homomorphisms $\phi$ are identity maps on $K_F$.
It is easy to see that $f$ is an evolution of $\mathcal{Y}$.

\begin{definition}
Let $\X = (X, (G_t, X_t)_{t \in \Di})$ be a TDS and $\VDi \subseteq \Di$ a finite nonempty subset with greatest and least elements, and let $F \subseteq G$ be suitable for $\VDi$.
The TDS $\mathcal{Y}$ defined above is denoted by $\loc{\VDi, F}{\X}$.
If $\Di = \VDi$ is a singleton, we denote $\loc{\VDi, F}{\X} = \loc{F}{\X}$.
\end{definition}

\begin{lemma}
  \label{lem:LocalSystem}
  Let $\X = (X, (G_t, X_t)_{t \in \Di}, f)$ be an ETDS and $\VDi \subseteq \Di$ as above, and let $F \subseteq G$ be suitable for $\VDi$ and $f$.
  If a tier $t \in \VDi$ is $f$-stabilizing in $\X$ and $s(t) \in \VDi$, then the same holds in $\loc{\VDi, F}{\X}$.
  The analogous result holds for weakly $0$-gluing tiers and the function $w$.
\end{lemma}


\begin{proof}
  We first assume that $s(t) \in \VDi$, and show that the tier $X_{F,t}$ is $f$-stabilizing.
  Let $x \in X_{F,t} \subseteq X_t$.
  Then we have $f^k(x) \in X_{s(t)}$ for all $k \geq 0$.
  Furthermore, for all $y \in X_{F,s(t)}$ such that $f(y) \in X_{s(t)}$ we have $\Theta(f(x)) \subseteq F \cdot \Theta(x) \subseteq K_F$ by the suitability of $F$, which implies $f(y) \in X_{F,s(t)}$.
  By induction on $k$, this shows $f^k(x) \in X_{F,s(t)}$ for all $k \geq 0$.

  Next, we assume $w(t) \in \VDi$, and show that $X_{F,t}$ is weakly $0$-gluing.
  The proof resembles that of Lemma~\ref{lem:ShadowFinites}.
  Let $L \subseteq H$ be a set of representatives for the left cosets of $K_F$ in $H$, and for $B \subseteq H$, denote $L^B = L \cap B K_F$, which is finite whenever $B$ is. 
  Let $0 < \epsilon < \Delta / 2$ be arbitrary, and let $E \subseteq H$ and $\delta > 0$ be given by the weak $0$-gluing property of the tier $X_t$ for $\epsilon$.
  Note that $E$ should actually be a subset of $G_{s(t)}$, but we can lift it to $H$ via $\phi^r_{s(t)}$ without affecting the existence of realizations.
  Let $D \subseteq H$ be a finite set such that for $x \in X_t$, the condition $d_X(h^{-1} \cdot x, 0) \leq \Delta$ for all $h \in D$ implies $d_X(x, 0) < \delta$.
  Let $\eta > 0$ be such that $d_X(x, 0) < \eta$ implies $d_X(k \cdot x, 0) < \delta$ for all $k \in L^{E D}$.
  Denote $E' = K_F \cap (L^{E D})^{-1} E L^{E D}$.
  
  Let $\mathcal{A} = (A_i, x_i)_{i \in I}$ be an $(\eta, E')$-designation in $(K_F, X_{F,t})$.
  Denote $B_i = L^{E D} A_i$.
  We claim that $(B_i, x_i)_{i \in I}$ is a $(\delta, E)$-designation in $(H, X_t)$.
  First, the sets $B_i \subseteq H$ are clearly disjoint.
  Let then $i \in I$ and $g \in \partial_E B_i$ be arbitrary.
  Then there exists $e \in E$ with $e g \notin B_i$.
  Suppose first that $e g \notin E D K_F$.
  Then we have $g \notin D K_F$, which implies $h^{-1} g \notin K_F$ for all $h \in D$.
  Since $x_i \in X_{F,t}$, we obtain $d_X(h^{-1} g \cdot x, 0) \leq \Delta$ for all $h \in D$, implying $d_X(g \cdot x, 0) < \delta$.
  Suppose then that $e g \in E D K_F$, and let $h, k \in L^{E D}$ be such that $h^{-1} e g, k^{-1} g \in K_F$.
  Note that $h^{-1} e g \notin A_i$ and $k^{-1} g \in A_i$.
  Now we have $h^{-1} e k \in E'$, so that $k^{-1} g \in \partial_{E'} A_i$, which implies $d_X(k^{-1} g \cdot x, 0) < \eta$ since $\mathcal{A}$ is an $(\eta, E')$-designation.
  From this we obtain $d_X(g \cdot x, 0) < \delta$.
  
  Since $(B_i, x_i)_{i \in I}$ is a $(\delta, E)$-designation in $(H, X_t)$, it has an $\epsilon$-realization $y \in X_{w(t)}$.
  It remains to show that $y \in X_{F,w(t)}$.
  For that, let $h \in H \setminus K_F$ and $g \in F$ be arbitrary.
  There exists $x \in X_{F,t}$ (either one of the points $x_i$ or $0_X$) with $d_X(g h \cdot y, g h \cdot x) < \epsilon$.
  We also have $g h \notin K_F$, which implies $d_X(g h \cdot y, 0) \leq d_X(g h \cdot y, g h \cdot x) + d_X(g h \cdot x, 0) \leq \epsilon + \Delta < \frac{3}{2} \Delta$ since $\Theta(x) \subseteq K_F$.
  By the suitability of $F$, this implies $d_X(h \cdot y, 0) \leq \Delta$.
  Thus we have $\Theta(y) \subseteq K_F$, or equivalently, $y \in X_{F,w(t)}$.
\end{proof}

As shown by Example~\ref{ex:Loc}, in general the $\mathrm{Loc}$-construction does not preserve the property of having weakly dense homoclinic points.
For single-tier systems, the situation changes.

\begin{lemma}
\label{lem:SingleTierLocal}
  Suppose $\X = \mathcal{U}(G, X, f) \in \mathfrak{C}_4$ is a single-tier ETDS.
  Then $\loc{F}{\X} \in \mathfrak{C}_4$ for every large enough (with respect to inclusion) suitable set $F \subseteq G$.
\end{lemma}

\begin{proof}
  Denote the tier set of $\X$ by $\Di = \{t_0\}$.
  Then $\loc{F}{\X}$ consists of the single tier $X_{F, t_0}$.
  By Lemma~\ref{lem:LocalSystem}, the single tier of $\loc{F}{\X}$ is $0$-gluing, so it suffices to show that it also has dense homoclinic points.
  Let $0 < \delta \leq \Delta/4$ and $E \subseteq G$ be given by the $0$-gluing property of $\X$ for $\Delta/4$.
  Let $L \subseteq G$ be a finite set such that for $x \in X$, the condition $d_X(g \cdot x, 0) \leq \Delta$ for all $g \in L$ implies $d_X(x, 0) \leq \delta / 2$.
  Choose the set $F$ so that $E \cup L \subseteq F$.
  
  Let $x \in X_{F, t_0}$ and $\epsilon > 0$ be arbitrary.
  Let $N \subseteq G$ be a finite set such that for $y, z \in X$, the condition $d_X(h \cdot y, h \cdot z) \leq \Delta$ for all $h \in N$ implies $d_X(y, z) \leq \epsilon/2$.
  Let $0 < \eta < \epsilon/2$ be such that $d_X(y, z) < \eta$ implies $d_X(g \cdot y, g \cdot x) < \delta/2$ for all $g \in N$.
  
  Since $\X \in \mathfrak{C}_4$, there exists a homoclinic point $y \in X$ with $d_X(x, y) < \eta$.
  Denote $M = K_F \cup N$, and consider the two-element designation $\mathcal{A} = ((M, y), (G \setminus M, 0_X))$ in $\X$.
  Since $E \subseteq F$, for each $k \in K_F$ we have $E k \subseteq K_F$, so $\partial_E M \subseteq N \setminus K_F$.
  For $g \in N \setminus K_F$, we have $d_X(g \cdot y, g \cdot x) < \delta / 2$ by our choice of $\eta$, as well as $d_X(h g \cdot x, 0) \leq \Delta$ for all $h \in L$ since $L \subseteq F$.
  The latter condition implies $d_X(g \cdot x, 0) \leq \delta/2$, so we obtain $d_X(g \cdot y, 0) < \delta$.
  Thus $\mathcal{A}$ is a $(\delta, E)$-designation, and has a $\Delta/4$-realization $z \in X$.
  We have $d_X(h \cdot z, h \cdot y) < \Delta/4 < \Delta$ for all $h \in N$, which implies $d_X(z, y) \leq \epsilon/2$ by our choice of $N$, and hence $d_X(x, z) < \epsilon$.
  
  We now claim that $z \in X_{F, t_0}$, and for that, let $g \in G \setminus K_F$ be arbitrary.
  If $g \notin M$, then $d_X(z, 0) \leq \Delta$ by construction.
  If $g \in M$, then $g \in N \setminus K_F$, and we have $d_X(g \cdot z, g \cdot y) < \Delta/4$ by construction of $z$ and $d_X(g \cdot y, g \cdot x) < \delta \leq \Delta/4$ by the definition of $\eta$, as well as $d_X(g \cdot x, 0) \leq \Delta$ since $x \in X_{F, t_0}$.
  This implies $d_X(g \cdot z, 0) \leq \frac{3}{2} \Delta$.
  We now have $d_X(g \cdot z, 0) \leq \frac{3}{2} \Delta$ for all $g \in G \setminus K_F$, so the suitability of $F$ gives $d_X(g \cdot z, 0) \leq \Delta$ for all $g \in G \setminus K_F$, and hence $z \in X_{F, t_0}$.
  
  Finally, since $z$ is a realization of a designation consisting of finitely many homoclinic points, it is homoclinic in $\X$ and hence in $\loc{F}{\X}$.
  We have shown that $\loc{F}{\X}$ has dense homoclinic points.
\end{proof}

On the restricted class $\mathfrak{C}_4$ consisting of single-tier systems, local nil-rigidity is equivalent to nil-rigidity.
This follows easily from the above construction.

\begin{proposition}
  \label{prop:OneTierLocal}
  If a group $G$ is locally nil-rigid on $\mathfrak{C}_2$ or $\mathfrak{C}_4$, then it is nil-rigid on $\mathfrak{C}_4$.
  The same holds if it is locally nil-rigid on $\mathfrak{C}_3$ and locally residually finite.
\end{proposition}

\begin{proof}
  Let $\mathfrak{C} \in \{\mathfrak{C}_2, \mathfrak{C}_3, \mathfrak{C}_4\}$, let $(\X, (G, X), f) \in \mathfrak{C}_4$ be a single-tier ETDS, and denote its index set by $\Di = \{t_0\}$.
  Let $F \subseteq G$ be a large enough suitable set for $\Di$ and $f$, and consider the system $\loc{F}{\X}$ with acting group $K_F = \langle F \rangle$.
  It follows from Lemma~\ref{lem:LocalSystem} and Lemma~\ref{lem:SingleTierLocal} that $\loc{F}{\X} \in \mathfrak{C}$; in the case $\mathfrak{C} = \mathfrak{C}_3$, note that $K_F$ is residually finite.
  Since $K_F$ is nil-rigid on $\mathfrak{C}$ as a finitely generated subgroup of $G$, the restriction of $f$ to $\loc{F}{\X}$ is nilpotent.
  Every homoclinic point of $(G, X)$ belongs to $\loc{F}{\X}$ for some choice of $F$, so it is mortal, and Lemma~\ref{lem:Mortal} implies that $f$ is nilpotent on $X$.
\end{proof}

\section{Locally Virtually Abelian Groups}
\label{sec:Abelian}

We now use the results of the previous sections to show that every locally virtually abelian group is nil-rigid on $\mathfrak{C}_2$.
The main motivation behind this class is that Theorem~\ref{thm:BigClosure} and the fact that all abelian groups are nil-rigid on $\mathfrak{C}_2$ together imply that all residually finite solvable groups are nil-rigid on $\mathfrak{C}_3$, and restricting to finitely generated abelian groups is not enough to establish this result with our techniques.
We consider the less natural but larger class of locally virtually abelian groups; since tiered systems based on abelian groups generally contain nonabelian groups, this does not introduce any additional complications to the proofs.

It is again enough to show that homoclinic points are mortal, after which Lemma~\ref{lem:Mortal} provides the desired result.
We thus choose a homoclinic point $x$ from an arbitrary tier, and construct a certain local subsystem $\loc{\VDi,F}{\X}$ where $F$ contains the shadow of $x$.
We then use the fact that $\langle F \rangle$ is virtually $\Z^d$ for some $d \in \N$ to produce a system with base group $\Z^d$.
We argue as in Lemma~\ref{lem:Towers} that the shadows of homoclinic points in $\loc{\VDi,F}{\X}$ can spread only by a finite distance in each cardinal direction of $\Z^d$ under the iteration of $f$.
More specifically, we use Lemma~\ref{lem:Periodization} (periodization) to produce a tiered system with base group $\Z$, apply Lemma~\ref{lem:Z} to show that the periodized points are mortal, and glue them together to produce a point that contradicts the asymptotic nilpotency of $f$.
The main technical difficulty is that the local subsystem has only finitely many tiers, while our proof of Lemma~\ref{lem:Towers} requires a potentially infinite chain of tiers depending on the distribution of homoclinic points.
Our solution is to iteratively construct a single designation with infinitely many partition elements, which is possible since we only consider the homoclinic points of a single tier.

\begin{lemma}
  \label{lem:Zk}
  Let $(X, (\Z^d, X_t)_{t \in \Di}, (\mathrm{ID}_{\Z^d})_{t \leq t'}, f)$ be an ETDS such that the tier $r = w(s(w(t)))$ exists.
  If $f$ is asymptotically nilpotent on $X_r$, then every homoclinic point in $X_t$ is mortal.
\end{lemma}

This proof is somewhat similar to Lemma~\ref{lem:Towers}, so we reuse some of the notation and arguments used there.

\begin{proof}
  Denote the TDS by $\X = (X, (\Z^d, X_t)_{t \in \Di})$.
  Let $\Delta$ be the canonical expansivity constant of $X_r$.
  If $d \leq 1$, then the claim holds by Proposition~\ref{prop:FiniteGroups} and Lemma~\ref{lem:Z}, so we assume $d \geq 2$.
  Fix $i \in \{1, \ldots, d\}$, let $\pi_i : \Z^d \to \Z$ be the projection to the $i$th coordinate, and denote $K_i = \mathrm{Ker}(\pi_i) \simeq \Z^{d-1}$.
  Let $\Theta : X_r \to 2^{\Z^d}$ be the $r$-shadow function and $\Theta_i : X_r \to 2^\Z$ the $(K_i, r)$-shadow function of $X_r$.
  Recall that $\vec v \in \Theta(x)$ means $d_X(\vec v \cdot x, 0_X) > \Delta$ for $x \in X_r$ and $\vec v \in \Z^d$.
  
  We claim that there exists $p \in \N$ such that $\Theta_i(f^n(x)) \subseteq [-p, p] + \Theta_i(x)$ for all $n \in \N$ and all homoclinic points $x \in X_t$, which means that the $r$-shadow of a homoclinic point can only spread for a bounded number of steps along the $i$th coordinate.
  Namely, if this holds for all $i \in \{1, \ldots, d\}$, then $\Theta(f^n(x)) \subseteq P + \Theta(x)$ for all $n \in \N$ and a finite set $P \subset \Z^d$.
  By Lemma~\ref{lem:BoundedShadow}, the set $\{ f^n(x) \;|\; n \in \N \}$ is finite, so $x$ must be mortal, proving the original claim.
  We assume for contradiction that such a $p$ does not exist, so that for all $p \in \N$ there exists a homoclinic point $x^p \in X_t$ and $n(p) \in \N$ with $\Theta_i(x^p) \cap [-2 p, 2 p] = \emptyset$ and $0 \in \Theta_i(f^{n(p)}(x^p))$.
  By replacing $x^p$ with $\vec v \cdot x^p$ for a suitable vector $\vec v \in \Z^d$ with $\vec v_i = 0$, we may assume $\vec 0 \in \Theta(f^{n(p)}(x^p))$, so that $d_X(f^{n(p^p)}(x), 0_X) > \Delta$.
  
  We now periodize each $x^p$ along a subgroup of $K_i$.
  More explicitly, we consider the system of periodic points $\fix{K_i}{\X}$ with base group $\Z^d / K_i \simeq \Z$.
  Let then $p \in \N$, and let $\epsilon_p > 0$ be small enough that $d_X(x^p, y) < \epsilon_p$ implies 
  $d_X(f^n(x^p), f^n(y)) < \Delta/2$ for all $y \in X_t$ and $n \leq n(p)$.
  By Lemma~\ref{lem:Periodization}, for each $p$ there exists a finite-index subgroup $H_p \leq K_i$ and a point $y^p \in \fix{H_p}{X_{w(t)}}$ that is homoclinic in $\fix{K_i}{\X}$ and satisfies $d_X(x^p, y^p) < \epsilon_p$.
  For large enough $p$, we also have $\Theta_i(y) \cap [-p, p] = \emptyset$, since $\Theta_i(y) \subseteq N + \Theta_i(x)$ for some finite set $N \subset \Z$ that does not depend on $p$.
  In particular, we have $d_X(f^{n(p)}(y^p), 0_X) \geq \Delta / 2$.
  
  We claim that each of the points $y^p$ is mortal.
  By Lemma~\ref{lem:PeriodicProps}, each tier $(w(t), H_p)$ of $\fix{K_i}{\X}$ is $f$-stabilizing with $s(w(t), H_p) = (s(w(t_0)), H_p)$, and the latter is $0$-gluing with $w(s(w(t)), H_p) = (r, H_p)$.
  Thus Lemma~\ref{lem:Z} applied to the system $\fix{K_i}{\X}$ implies that $f$ is nilpotent on $\fix{H_p}{X_{w(t)}}$.
  Then $y^p \in \fix{H_p}{X_{w(t)}}$ is mortal, as claimed.
  Let $\ell(p) \in \N$ be such that $f^{\ell(p)}(y^p) = 0_X$.
  Clearly, we have $n(p) < \ell(p)$.
  
  Let $0 < \epsilon \leq \Delta / 4$ be small.
  Let $N \subset \Z^d$ be a finite set with $\Theta(f(x)) \subseteq -N + \Theta(x)$ for all $x \in X_r$.
  Let $\delta > 0$ and $E \subset \Z^d$ be given for $\epsilon$ by the weak $0$-gluing property of $s(w(t))$.
  Let $M \subset \Z^d$ be such that $d_X(\vec v \cdot x, 0_X) \leq \Delta$ for all $\vec v \in M$ implies $d_X(x, 0_X) < \delta$.
  We may assume $\vec 0 \in M$.
  For $x \in X_{w(t)}$ and $m \in \N$, we denote $A_m(x) = \bigcup_{q \leq m} E - M + \Theta(f^q(x))$.
  As in Lemma~\ref{lem:Towers}, for all $\vec v \in \partial_E A_m(x)$ and $q \leq m$ we have $d_X(\vec v \cdot f^q(x), 0_X) < \delta$.
  This fact is used below in the construction of several designations.
  
  We inductively construct a $(\delta, E)$-designation $\mathcal{B} = (B_k, x_k)_{k \geq 1}$ in $X_{s(w(t))}$, maintaining the following conditions for all $k \geq 1$.
  \begin{enumerate}
    \item There exists $p(k) \in \N$ with $x_k = y^{p(k)}$, and $\Theta(x_k) \subseteq B_k \subseteq A_{n(p(k))}(x_k)$.
    \item For $1 \leq j < k$ we have $A_{\ell(p(j))}(x_j) \cap A_{\ell(p(j))}(x_k) = \emptyset$.
  \end{enumerate}
  We also guarantee that $n(p(k)) < \ell(p(k)) < n(p(k+1))$ holds for all $k$.
  The idea is that each $y^{p(k)}$ is not close to $0_X$ at time $n(p(k))$, and dies off at time $\ell(p(k))$, making way for $y^{p(k+1)}$ to be far from the origin at time $n(p(k+1))$, and so on.
  
  Let $k \geq 1$.
  We assume that the partial designation $(B_j, x_j)_{1 \leq j < k}$ has been constructed (it is empty in the case $k = 1$), and proceed to choose $B_k$ and $x_k$.
  Let $V \subset \Z^d$ be a finite set with $d_X(x, 0_X) \geq \Delta / 2$ implying $d_X(\vec v \cdot x, 0_X) > \Delta$ for some $\vec v \in V$, when $x \in X_r$.
  Denote $n = \ell(p(k-1))$, and define
  \[
    R = \bigcup_{\substack{0 \leq m \leq k \\ 1 \leq j < k}} \pi_i(N \cdot m + ((M - E + A_n(x_j)) \cup V)) \subset \Z
  \]
  Since $R$ is finite, there exists $p(k) \in \N$ with $R \cap \Theta_i(y^{p(k)}) = \emptyset$.
  As in Lemma~\ref{lem:Towers}, we have $A_n(x_j) \cap A_n(y^{p(k)}) = \emptyset$ for all $j < k$, as well as $p(k) > n$.
  We choose $x_k = y^{p(k)}$ and $B_k = A_0(x_k)$.
  Then $k$ satisfies conditions 1 and 2.
  
  Since $\mathcal{B}$ is a $(\delta, E)$-designation in $X_{s(w(t))}$, it has an $\epsilon$-realization $z \in X_r$.
  We now consider the trajectory of $z$.
  For $q \geq 0$, let $k_q \geq 0$ be such that $\ell(p(k_q)) \leq q < \ell(p(k_q+1))$ (with the convention $\ell(p(0)) = 0$), and denote by $\mathcal{B}_q$ the designation $(A_{\ell(p(k_q))}(x_k), f^q(x_k))_{k \geq k_q}$.
  We know that the sets $A_{\ell(p(k_q))}(x_k)$ are disjoint for $k \geq k_q$ and $\vec v \in \partial_E A_{\ell(p(k_q))}(x_k)$ implies $d_X(\vec v \cdot f^q(x_k), 0_X) < \delta$, so $\mathcal{B}_q$ is a valid $(\delta, E)$-designation in $X_{s(w(t))}$.
  Denote its $\epsilon$-realization by $z_q \in X_r$, and note that $z = z_0$.
  
  \begin{claim}
    \label{cl:OrbitOfz}
    If $\epsilon$ is small enough, then for each $q \in \N$ we have $f(z_q) = z_{q+1}$.
  \end{claim}
  
  \begin{proof}
    If $q + 1 < \ell(p(k_q+1))$, then the designations $\mathcal{B}_q$ and $\mathcal{B}_{q+1}$ have the same subsets of $\Z^d$, and each point of $\mathcal{B}_{q+1}$ is the $f$-image of the corresponding point of $\mathcal{B}_q$.
    Then Lemma~\ref{lem:FuncGlue} implies $f(z_q) = z_{q+1}$.
    
    Consider then the case $q + 1 = \ell(p(k_q + 1))$, where $k_{q+1} = k_q + 1$.
    Intuitively, $\mathcal{B}_{q+1}$ is obtained from $\mathcal{B}_q$ by applying $f$ to each point, enlarging the subsets of $\Z^d$ to accommodate them, and removing one of the subsets and the associated point, which has become $0_X$.
    More formally, consider the $(\delta, E)$-designation $\mathcal{B}' = (A_{\ell(p(k_q))}(x_k), f^{q+1}(x_k))_{k \geq k_q}$, which is obtained from $\mathcal{B}_q$ by applying $f$ to each point.
    If $\epsilon$ is small enough, Lemma~\ref{lem:FuncGlue} implies that $f(z_q)$ is an $\epsilon$-realization of $\mathcal{B}'$.
    Consider then the $(\delta, E)$-designation $\mathcal{B}'' = (A_{\ell(p(k_{q+1}))}(x_k), f^{q+1}(x_k))_{k \geq k_q}$, which is obtained from $\mathcal{B}'$ by enlarging the subsets of $\Z^d$.
    By Lemma~\ref{lem:RegionsGlue}, $f(z_q)$ is the $\epsilon$-realization of $\mathcal{B}''$ as well.
    Finally, since $f^{q+1}(x_{k_q}) = 0_X$, it is clear that $\mathcal{B}''$ and $\mathcal{B}_{q+1}$ have the same realizations, so $f(z_q)$ is the $\epsilon$-realization of $\mathcal{B}_{q+1}$.
    Lemma~\ref{lem:UniqGlue} implies $f(z_q) = z_{q+1}$.
  \end{proof}
  
  For all $k \geq 1$, denoting $p = p(k)$, the point $f^{n(p)}(y^p)$ is $\epsilon$-close to $z_{n(p)}$, since $\vec 0 \in A_{\ell(p)}$.
  This implies $d_X(z_{n(p)}, 0_X) \geq d_X(f^{n(p)}(y^p), 0_X) - d_X(z_{n(p)}, f^{n(p)}(y^p)) \geq \Delta / 4$, which contradicts the asymptotic nilpotency of $f$ on $X_r$.
\end{proof}

\begin{lemma}
  \label{lem:AbelByFinite}
  Every finitely generated finite-by-(virtually abelian) group is virtually $\Z^d$ for some $d \in \N$.
\end{lemma}

\begin{proof}
  Suppose we have a short exact sequence $1 \to K \to G \stackrel{\psi}{\to} H \to 1$, where $K$ is finite and $A \leq H$ is abelian of finite index.
  Since $G$ is finitely generated, so is $A$, so it has a finite-index subgroup isomorphic to $\Z^d = \langle \vec e_1, \ldots, \vec e_d \rangle$.
  Let $g_i \in \psi^{-1}(\vec e_i)$ for each $i$, and let $C \trianglelefteq \psi^{-1}(A)$ be the centralizer of $\{g_1, \ldots, g_d\}$ in $\psi^{-1}(A)$.
  Since $g^{-1} g_i g \in \psi^{-1}(\vec e_i) \subseteq g_i K$ for all $g \in \psi^{-1}(A)$, $C$ has finite index in $\psi^{-1}(A)$.
  For some $m \geq 1$ we have $g_i^m \in C$ for all $i$, and then $\langle g_1^m, \ldots, g_d^m \rangle \simeq \Z^d$ has finite index in $G$.
\end{proof}

\begin{proposition}
  \label{prop:LocallyVirtuallyAbelian}
  Every locally virtually abelian group is nil-rigid on $\mathfrak{C}_2$.
\end{proposition}

\begin{proof}
  Let $G$ be such a group, and let $(\X, f) \in \mathfrak{C}_2$ be an asymptotically nilpotent ETDS with $\X = (X, (G_t, X_t)_{t \in \Di})$ and $G_0 = G$.
  We claim that for every tier $t \in \Di$, every homoclinic point $x \in X_t$ is mortal.
  Denote $\VDi = \{t, w(t), s(w(t)), w(s(w(t)))\}$ and $r = w(s(w(t)))$.
  Let $F \subset G_r$ be a finite set containing $\Theta_r(x)$ and suitable for $\VDi$ and $f$, and construct the system $\loc{\VDi,F}{\X}$.
  As a finitely generated subgroup of $G_r$, $K_F = \langle F \rangle$ is finite-by-(virtually abelian), hence virtually $\Z^d$ for some $d \in \N$ by Lemma~\ref{lem:AbelByFinite}.
  We use Lemma~\ref{lem:VirtualProp} to pass to into a TDS that has the same tiers, but every acting group is $\Z^d$. Then the tier $w(s(w(t)))$ still exists, so by Lemma~\ref{lem:Zk}, every homoclinic point in $X_{F,t}$ is mortal, so in particular $x$ is.
  Since $x$ was arbitrary, Lemma~\ref{lem:Mortal} implies that $f$ is nilpotent on each tier.
\end{proof}

\begin{corollary}
  Every abelian group is nil-rigid on $\mathfrak{C}_2$.
\end{corollary}

\section{The Class of Nil-Rigid Groups}
\label{sec:GroupsWeGet}

If a group $G$ is nil-rigid on $\mathfrak{C}_4$, then all expansive $0$-gluing pointed dynamical systems $(G, X)$ with dense homoclinic points are nil-rigid, meaning that each asymptotically nilpotent endomorphism $f : X \to X$ is nilpotent.
We can now show this property for a large class of groups.

\begin{theorem}
  \label{thm:BigClass}
  Let $G$ be a residually finite group with a subnormal series $\{1_G\} = H_0 \triangleleft H_1 \triangleleft \cdots \triangleleft H_n = G$, where each factor $H_{k+1} / H_k$ is locally virtually abelian.
  Then $G$ is nil-rigid on $\mathfrak{C}_3$.
\end{theorem}

\begin{proof}
  For $n = 0$, the result follows from Proposition~\ref{prop:FiniteGroups}.
  For $n \geq 1$, we have an exact sequence $1 \to H_{n-1} \to G \to H_n / H_{n-1} \to 1$.
  The factor $H_n / H_{n-1}$ is nil-rigid on $\mathfrak{C}_2$ by Proposition~\ref{prop:LocallyVirtuallyAbelian}, and $H_{n-1}$ is residually finite as a subgroup of $G$, and thus nil-rigid on $\mathfrak{C}_3$ by the induction hypothesis.
  Theorem~\ref{thm:BigClosure} implies the desired result.
\end{proof}

\begin{corollary}
  \label{cor:Solvables}
  Every residually finite solvable group is nil-rigid on $\mathfrak{C}_3$.
\end{corollary}

Write $\mathfrak{G}$ for the class of groups defined in Theorem~\ref{thm:BigClass}, and $\mathfrak{G}^*$ for the class of groups that are locally $\mathfrak{G}$.
By Proposition~\ref{prop:OneTierLocal}, all groups in $\mathfrak{G}^*$ are nil-rigid on $\mathfrak{C}_4$.

\begin{corollary}
  The following groups are nil-rigid on $\mathfrak{C}_4$: 
  \begin{itemize}
  \item every linear group (that is, subgroup of some $GL(n,F)$ with $n \geq 1$ and $F$ any field) that does not contain a free group,
  \item every (abelian-by-polycyclic)-by-finite group,
  \item every finitely generated group of polynomial growth,
  \item every virtually nilpotent group,
  \item every wreath product $A \wr P$ where $A$ is abelian and $P$ is locally virtually polycyclic.
  \end{itemize}
\end{corollary}

\begin{proof}
Every linear group is locally residually finite by Malcev's theorem \cite{Ma40}. A linear group that does not contain a free group is locally virtually solvable by Tits' alternative \cite{Ti72}. Every (abelian-by-polycyclic)-by-finite group is locally residually finite by \cite{Ja74,Ro76}, and these groups are virtually solvable by form. Every finitely generated group of polynomial growth is virtually nilpotent by \cite{Gr81}.

To see that all virtually nilpotent groups are nil-rigid on $\mathfrak{C}_4$, note that by Lemma~\ref{lem:virtual} we only need to show this for nilpotent groups. Every nilpotent group is locally polycyclic (see e.g. Theorem~2.13 of \cite{We09}), so the claim follows from Proposition~\ref{prop:OneTierLocal}.

Consider then the wreath product $A \wr P$.
By Proposition~\ref{prop:OneTierLocal}, it suffices to prove that each finitely generated subgroup $G \leq A \wr P$ is nil-rigid on $\mathfrak{C}_4$.
Such a $G$ embeds into $A \wr P'$ for some finitely generated, and hence virtually polycyclic, subgroup $P' \leq P$.
The group $A \wr P'$ is abelian-by-(virtually polycyclic), hence (abelian-by-polycyclic)-by-finite and nil-rigid on $\mathfrak{C}_4$ by the second claim.
The nil-rigidity of $G$ follows from Lemma~\ref{lem:Subgroups}.
\end{proof}

In particular, the discrete Heisenberg group and the lamplighter group $\Z_2 \wr \Z$ are nil-rigid on $\mathfrak{C}_4$. Note that (abelian-by-polycyclic)-by-finite also covers every metabelian group.


A free group on two or more generators is not in the class $\mathfrak{G}^*$.
To see this, suppose that $A = H_0 \triangleleft H_1 \triangleleft \cdots \triangleleft H_n = H \leq F$ is the shortest subnormal series that witnesses $H \in \mathfrak{G}$ for a finitely generated non-abelian free subgroup $H$ of such a free group $F$.
Since subgroups of free groups are free, each $H_i$ is a free group on one generator, i.e. $\Z$, which implies that $H$ is solvable, a contradiction.
In spite of this, the two-generator free group $F_2$ seems like a natural candidate for a nil-rigid group -- indeed, one might think nil-rigidity of all virtually free groups could be proved similarly to the nil-rigidity of $\Z$, but we are unable to do this.

\begin{question}
  Is the free group $F_2$ nil-rigid on $\mathfrak{C}_4$?
\end{question}

Another interesting question is whether Corollary~\ref{cor:Solvables} can be extended to all solvable groups.
We do not know the answer even on the weaker class $\mathfrak{C}_4$.
Our proof techniques are strongly based on periodization, which requires the existence of many finite-index subgroups.
It is not clear how to proceed in the absence of residual finiteness.

\begin{question}
  Is every solvable group nil-rigid on $\mathfrak{C}_4$?
\end{question}

On the other hand, on a general residually finite group, such as a nonabelian free group, periodization is possible.
It seems plausible that periodization, together with some new ideas that strengthen or replace the arguments of Section~\ref{sec:Finites}, could be used to prove the nil-rigidity of some well-behaved subclass of residually finite groups.

\begin{question}
  Is every residually finite group nil-rigid on $\mathfrak{C}_4$?
\end{question}

We also do not have counterexamples for general groups, but here our techniques do not apply at all.
We are not even aware of any countable groups $G$ such that the full shift $\Sigma^G$ is not nil-rigid, and this subcase is of particular interest.

\begin{question}
\label{q:Every}
  Is every group nil-rigid on $\mathfrak{C}_4$?
\end{question}

One natural class of examples of independent interest are cellular automata on group sets \cite{Mo11,Wa17}.

\begin{question}
Let $G \curvearrowright A$ be a transitive group action. Let $\Sigma \ni 0$ be a finite alphabet and let $G$ act on $\Sigma^A$ by $(g \cdot x)_a = x_{g^{-1} \cdot a}$. When is the system $(G, \Sigma^A)$ nil-rigid?
\end{question}

Of particular interest are actions arising from automorphism groups of locally finite graphs. A theorem of Trofimov \cite[Theorem~1]{Tr85} generalizes Gromov's theorem from groups to vertex-transitive graphs of polynomial growth, saying that after quotienting out finite blocks of imprimitivity, any such graph admits a vertex-transitive action by a finitely generated virtually nilpotent group with finite stabilizers. However, we have not been able to prove nil-rigidity in this case. See Example~\ref{ex:Trofimov} for a special case.

\section{Examples and Counterexamples}
\label{sec:FinalExamples}

In this section, we list some examples of dynamical systems (most of which are zero-dimensional) for which our results imply something meaningful, as well as some counterexamples for potential strengthenings of our results, and some connections with other notions. Our main application of interest are cellular automata on full shifts on groups, and indeed we consider the following to be the most important corollary.

\begin{example}
\label{ex:IrreducibleSFT}
Let $G$ be a countable group and let $X \subset \Sigma^G$ be a strongly irreducible SFT (for example, $X = \Sigma^G$). If $G \in \mathfrak{G}^*$, then all asymptotically nilpotent cellular automata on $X$ are nilpotent. To see this, observe that if $f : X \to X$ is asymptotically nilpotent, then in particular $0^G \in X$. By strong irreducibility, finite points are dense in $X$. Since $X$ is an SFT, it has the shadowing property, which implies $0$-gluing by Proposition~\ref{prop:ShadowingIsGluing}.
\tqed
\end{example}

In the case $G = \Z^2$, the weaker assumption of block-gluing is enough. See Section~\ref{sec:BlockGluing}.

\begin{example}
Let $M$ be a countable monoid that appears as a submonoid of a countable group $G \in \mathfrak{G}^*$, suppose $\Sigma \ni 0$ is a finite alphabet and consider the space $\Sigma^M$ with the product topology. For $m \in M, A \subset M$, define $D_{m,A}(x) : \Sigma^M \to \Sigma^A$ by $D_{m,A}(x)_a = x_{ma}$, and let $f : \Sigma^M \to \Sigma^M$ be a function satisfying $f(x)_m = F(D_{m,A}(x))$ for some finite $A \subset M$ and function $F : \Sigma^A \to \Sigma$. If $f$ is asymptotically nilpotent on $\Sigma^M$ (towards $0^M$), then it is nilpotent. To see this, observe that $f$ naturally induces an asymptotically nilpotent endomorphism of $\Sigma^G$. From this, we obtain in particular the nil-rigidity of cellular automata on $\N^d$.
\tqed
\end{example}

The next example highlights a relatively concrete class of groups in $\mathfrak{G}^*$ that are not virtually nilpotent.
It generalizes a construction used in \cite{Sa17a} to produce counterexamples to the Tits alternative in the group $\Aut(\Sigma^\Z)$, and is also somewhat similar to the constructions presented in \cite{dCoMa07}.

\begin{example}
\label{ex:WeirdGroup}
Let $K$ and $H$ be residually finite countable groups with $|K| \geq 2$.
We construct a new residually finite group $G$, which can be thought of as the residually finite version of the restricted wreath product $K \wr H$, as follows.
Denote
\[
  X = \bigcup_{\substack{A \trianglelefteq_f K \\ B \trianglelefteq_f H}} (K / A)^{H / B}
\]
where $A$ ranges over the finite-index normal subgroups of $K$, and $B$ the finite-index normal subgroups of $H$.
For fixed $A$ and $B$, elements of the set $(K / A)^{H / B}$ are functions from $H/B$ to $K/A$.
To each element $g \in K \cup H$ we associate a permutation $\psi_g$ of $X$ as follows.
If $g \in K$ and $f \in (K / A)^{H / B}$, then $\psi_g(f)(H) = g \cdot f(H)$ and $\psi_g(f)(b H) = f(b H)$ whenever $b H \neq H$; intuitively, $\psi_g$ acts on the element at the origin of $H/B$ using multiplication by $g$, and leaves the other elements unchanged.
If $g \in H$, then $\psi_g(f)(b H) = f(g^{-1} b H)$ for all $b H \in H/B$; this means that $\phi_g$ permutes the coordinates of $f$ using multiplication by $g$.
Let $G \leq \mathrm{Sym}(X)$ be the group of permutations generated by $\psi(K \cup H)$.

We make some observations about this group.
First, it is residually finite, since it is defined by its action on the finite sets $(K/A)^{H/B}$.
Second, if $K$ and $H$ are finitely generated, so is $G$.
Third, the map $\psi$ restricts to an embedding of both $K$ and $H$ into $G$.
This last property is not as trivial, but it follows from the residual finiteness of the groups.

Define a function $\phi : K \cup H \to H$ by $\phi(h) = h$ for $h \in H$ and $\phi(k) = 1_H$ for $k \in K$.
Then $\phi$ extends into a homomorphism from $G$ to $H$, and its kernel $\ker(\phi) \trianglelefteq G$ is isomorphic to a subgroup of the infinite direct sum $K^\omega$ and contains copies of $K^n$ for all finite $n$.
If $K$ has a property P that is preserved under finite products and subgroups, then $\ker(\phi)$ is locally P, and $G$ is locally-P-by-$H$.
In particular, if $K = \Z^5 \rtimes A_5$ is the semidirect product of $\Z^5$ and the alternating group $A_5$ that permutes its axes, and $H = \Z$, then $G$ is (locally virtually abelian)-by-abelian, hence $G \in \mathfrak{G}$.
We could also choose $K = A_5$ and $G = \Z$ to obtain a (locally finite)-by-abelian group in $\mathfrak{G}$.

On the other hand, we claim that if $K$ has a nonsolvable finite factor and $H$ is infinite, as is the case in both examples above, then $G$ is not virtually solvable.
Suppose it were, and take a solvable subgroup $L \leq G$ of finite index $n \geq 1$.
Let $A \trianglelefteq K$ be a subgroup of finite index such that $F = K/A$ is not solvable, and let $B \trianglelefteq H$ have finite index at least $n$.
Then the action of $G$ on $F^{H/B}$ induces a surjective homomorphism from $G$ to $F \wr (H/B)$, which sends $L$ to a solvable subgroup $E$ of index at most $n$.
Consider the subgroup $F^n \leq F \wr (H/B)$, and note that $E \cap F^n$ has index at most $n$ in $F^n$ as well, and that this group is solvable since $E$ is.
Now, for each $i \in \{1, \ldots, n\}$ the projection map $\pi_i : E \cap F^n \to F$ is nonsurjective, since $F$ is not solvable.
But then the index of $E \cap F^n$ in $F^n$ is at least $2^n$, a contradiction.
\tqed
\end{example}

\begin{example}
\label{ex:VirtuallyZ}
Consider the sofic shift $X \subset \{0, 1, 2\}^\Z$ defined by the forbidden patterns $1 0^n 1$ and $2 0^n 2$ for all $n \in \N$.
Let $h : X \to X$ be the automorphism that changes every $1$ into $2$ and vice versa.
Let $G = \Z \times \Z_2$, which is nil-rigid on $\mathfrak{C}_1$, and let $(G, X)$ be the dynamical system whose action is defined by $(1, \overline{0}) \cdot x = \sigma(x)$ and $(0, \overline{1}) \cdot x = h(x)$.
It is not hard to see that any family of homoclinic points of $X$ can be glued together, possibly after applying the function $h$ to some of them.
This means that $(G, X)$ has the $f$-variable homoclinic recurrence property, so it is nil-rigid by Corollary~\ref{cor:Z}.
We also note that the system has dense periodic points and dense homoclinic points.

Let $f : X \to X$ be the cellular automaton that shifts every $1$ to the right and every $2$ to the left, and turns to $0$ those $1$-$2$ pairs that would collide or pass through each other.
Since $1$s and $2$s alternate in all configurations of $X$, the CA $f$ is asymptotically nilpotent, but it is not nilpotent.
This does not contradict the nil-rigidity of the system, since $f$ does not commute with $h$ and thus is not an endomorphism of $(G, X)$.
\tqed
\end{example}

\begin{example}
Consider the ternary full shift $X = \{0,1,2,3\}^\Z$, and the following cellular automaton $f : X \to X$: $2$-symbols move to the right, and $3$-symbols move left, destroying both on collision, like in Example~\ref{ex:VirtuallyZ}. For $n \geq 2$, the pattern $01^n0$ is turned into $020^{n-2}30$, while lone $1$s are turned into $0$s. Consider the tiered system $(\{0,1,2,3\}^\Z, (\Z,\{0,1\}^\Z))$ (with just one tier $\{0,1\}^\Z$). This system is weakly $0$-gluing and has weakly dense homoclinic points. The map $f$ is asymptotically nilpotent on $\{0,1\}^\Z$, but it is not uniformly so. Indeed, while $f$ is an evolution of the TDS, it is not stabilizing, because $\{0,1\}^\Z$ is not mapped to any tier. If we add a second tier and consider the tiered system $(\{0,1,2,3\}^\Z, (\Z, [0,n]^\Z)_{n \in \{1,3\}})$ instead, then $f$ is a stabilizing evolution, and indeed the system is in $\mathfrak{C}_3$. We have that $f$ is non-uniformly asymptotically nilpotent on the first tier. Since $f$ is not asymptotically nilpotent on the second tier (which is the ambient space $X$), this example does not contradict our results.
\tqed
\end{example}

\begin{example}
\label{ex:EvenShift}
Consider the \emph{even shift} $X \subset \{0, 1\}^\Z$ defined by the forbidden patterns $1 0^{2 n + 1} 1$ for all $n \in \N$.
It has the $f$-variable homoclinic recurrence property for any endomorphism $f$, since every collection of homoclinic points can be glued together, possibly after translating some of them one step to the left.
Theorem~\ref{thm:Z} implies that it is nil-rigid.
This result was also claimed in \cite{Sa17}.

In the same vein, consider any \emph{coded system} $Y \subseteq \Sigma^\Z$, which is the closure of all bi-infinite concatenations of words drawn from a set $W \subseteq \Sigma^+$. Suppose that $0^k \in W$ for some $k \geq 1$, and there exists $n \geq 0$ such that every $w \in W \setminus 0^*$ satisfies $0^n \not\sqsubset w$.
For the same reason as above, Theorem~\ref{thm:Z} shows that $Y$ is nil-rigid. The assumption on $0$-symbols is necessary, since the sofic subshift in Example~\ref{ex:VirtuallyZ} is the coded system defined by $W = 0^* \cup 10^*2$.
\tqed
\end{example}

To showcase the power of Lemma~\ref{lem:Z}, we give examples of $\Z$-subshifts which are not $f$-variably $0$-gluing, but have the $f$-variable homoclinic recurrence property.

\begin{example}
Consider the subsystem $X \subset \Sigma^\Z$ defined by forbidding each word $w \in \Sigma^{k^2}$ containing more than $k$ nonzero symbols.
We claim that $X$ has the $f$-variable homoclinic recurrence property with respect to any endomorphism $f$.
Namely, if $v, w \in \Sigma^{n^2}$ are two words of the same length $n^2$ occurring in $X$, then $x = {}^\infty 0 v 0^{4 n^2} w 0^\infty \in X$, since any word occurring in $x$ that contains symbols from both $v$ and $w$ has length greater than $4 n^2$ and contains at most $2 n$ nonzero symbols.
We can iterate this construction to glue together any collection of homoclinic points.
For an example of a system that requires the use of nontrivial functions in $\Lambda_f$, take $\Sigma = \{0, 1\}$ and consider the intersection of $X$ with the even shift of Example~\ref{ex:EvenShift}, or take $\Sigma = \{0, 1, 2\}$ and consider the intersection of $X$ with the shift of Example~\ref{ex:VirtuallyZ} (replacing the group $\Z$ with $\Z \times \Z_2$ in the latter case).
By Corollary~\ref{cor:Z}, these systems are nil-rigid.
\tqed
\end{example}

\begin{example}
Consider the subshift $Y \subset \{0, 1\}^\Z$ defined as the shift orbit closure of the configurations $x^n = {}^\infty 0 . 1 0^n 1 0^{n+1} 1 0^{n+2} 1 \cdots$ for all $n \in \N$.
This subshift does not have the $f$-variable homoclinic recurrence property for any endomorphism $f$, since every configuration $y \in Y$ satisfies $y_i = 0$ for all negative $i \in \Z$ with large enough absolute value.
The right shift is an asymptotically nilpotent automorphism of $Y$ that is not nilpotent.

However, the union of $Y$ with its reversal does have the $f$-variable homoclinic recurrence property for every endomorphism $f$.
Up to shifting, the only homoclinic point of this system is $x = {}^\infty 0 1 0^\infty$, and depending on the choice of $\epsilon > 0$ and $c \in \{1, -1\}$, we can choose either $x^n$ or its reversal for some large enough $n$ as the realization of a suitable designation.
By Lemma~\ref{lem:Z}, this system is nil-rigid for any endomorphism.
The intuition is that if an endomorphism $f$ is not nilpotent, then $f(x) = q \cdot x$ for some $q \in \Z$, and if $n$ is larger than the radius of $f$, then $f(x^n) = q \cdot x^n$ holds as well, and analogously for the reversal of $x^n$.
Depending on the sign of $q$, one of these configurations will have infinitely many $1$-symbols crossing the origin under the repeated application of $f$.
\tqed
\end{example}

\begin{example}
\label{ex:Trofimov}
Consider the graph product $K_5 \boxtimes \Z$ where $K_5$ is the five-vertex clique and $\Z$ denotes the infinite line graph $(\Z, \{(a,b) \;|\; |a-b| = 1\})$. The wreath product $A_5 \wr \Z$ acts on this graph, and thus its colorings, in an obvious way (fix an origin copy of $K_5$, which $A_5$ permutes, and let $\Z$ shift along $\Z$). Let $\Sigma$ be a finite alphabet and consider the dynamical system $(\Sigma^{K_5 \boxtimes \Z}, A_5 \wr \Z)$.

Our results do not directly apply to this system for several reasons -- $A_5 \wr \Z$ is not residually finite, and there are essentially no homoclinic points since stabilizers are infinite. From infinite stabilizers one also obtains that the system is not $0$-gluing. 
Nevertheless, this system is nil-rigid: As Trofimov's theorem predicts, there are finite blocks of imprimitivity (namely the copies of $K_5$) such that quotienting them out yields a graph whose automorphism group is finitely generated and virtually nilpotent, and acts with finite stabilizers. In our case the graph is $\Z$, the homomorphism from $\Aut (K_5 \boxtimes \Z)$ to $\Aut (\Z)$ splits and we obtain an expansive action of $\Z$ on $\Sigma^{K_5 \boxtimes \Z}$. This action has dense homoclinic points and is $0$-gluing. Any endomorphism of $(\Sigma^{K_5 \boxtimes \Z}, A_5 \wr \Z)$ is of course also an endomorphism of this subaction, so the system is nil-rigid. \tqed
\end{example}

\begin{example}
\label{ex:UnitCircle}
Let $\mathbb{S}$ be the unit circle, which we identify with the interval $[0, 1)$ by the map $x \mapsto e^{2 \pi i x}$. The map $x \mapsto x^2$ is asymptotically nilpotent on $\mathbb{S}$ but not uniformly so. Let $G$ be an infinite group, denote $X = \mathbb{S}^G$, and consider the pointed dynamical system $(G, X)$ with the shift action and $0_X = 0^G$.
This system is $0$-gluing and has dense homoclinic points, but it is not expansive.
Define $f : X \to X$ by $f(x)_g = x_g^2$ for all $g \in G$.
Then $f$ is an asymptotically nilpotent endomorphism of $(G, X)$ which is not nilpotent.
This shows that expansiveness plays a major role in our results.
\tqed
\end{example}

\begin{example}
\label{ex:OnePoint}
Let $\dot \N = \N \cup \{\infty\}$ be the one-point compactification of $\N$. It is not hard to show that $\dot \N$ has no non-uniformly asymptotically nilpotent continuous self-maps. However, we claim that $\dot \N^\Z$ with the shift dynamics admits such an endomorphism, namely the function defined by
\[
  f(x)_i =
  \begin{cases}
    \max(0, \min(x_{i+1}, m)-1), & \text{if $x_{i-m} > 0$ for some $m \geq 1$,} \\
    \max(0, x_{i+1}-1), & \text{if such an $m$ does not exist,}
  \end{cases}
\]
with the interpretation that $m$ is always chosen minimal, and $\infty - 1 = \infty$.
It is easy to see that $f$ is continuous and commutes with the shift action of $\Z$.

On points of the form $\cdots 0 0 0 \infty 0 0 0 \cdots$, the function $f$ behaves as the shift map.
Since the $\infty$-symbol can pass over the origin arbitrarily late, $f$ is not uniformly asymptotically nilpotent.
Let then $x \in \dot \N^\Z$ be an arbitrary point, and suppose $x_i > 0$ for some $i \in \Z$.
For any $j \geq i$, we then have $f(x)_j \leq j - i$, which implies $f^{1+n}(x)_{j-n} \leq \max(0, j-i-n)$ for all $n \in \N$.
In the case $n = j-i$, we obtain $f^{1+j-i}(x)_i = 0$.
Thus for each $i \in \N$ there is at most one time step $n \in \N$ with $f^n(x)_i > 0$, so $f$ is asymptotically nilpotent.

This is another example of a system $(G, X)$ that is $0$-gluing and has dense homoclinic points, but is not nil-rigid.
As with Example~\ref{ex:UnitCircle}, it is not expansive.
However, for all $x \neq 0_X$ there exists $g \in G$ with $d_X(g \cdot x, 0_X) \geq 1$ (with a suitable choice of the metric $d_X$), which is a weaker version of expansivity for pointed systems.
The example shows that it is indeed too weak for our purposes, even in the case $G = \Z$.
\tqed
\end{example}

\begin{example}
While tiered dynamical systems were only defined as a technical tool, there are also some direct applications for the nilpotency of ETDS. Denote by $\ddot\Z$ the two-point compactification $\Z \cup \{-\infty,\infty\}$ of $\Z$ and consider the dynamical system $\ddot\Z^\Z$ under the shift dynamics $\sigma(x)_i = x_{i-1}$ and zero point $0^\Z$. This system is compact, but not expansive. It has, however, natural expansive subsystems $[a,b]^\Z$, which form the tiers of a system
\[ \X = (\ddot\Z^\Z, (\Z, [a,b]^\Z)_{a, b \in \Z, a \leq b}) \]
where the homomorphisms between the groups are identity maps, and the ordering is the inclusion order between the intervals $[a,b]$. Any endomorphism $f : \ddot\Z^\Z \to \ddot\Z^\Z$ under the dynamics of $\sigma$ induces an evolution of the tiered system $\X$. It induces a stabilizing one if and only if for all $a \in \N$ there exists $b \in \N$ such that $f^n([-a,a]^\Z) \subseteq [-b,b]^\Z$ for all $n \in \N$. It is clear that $\X$ is weakly $0$-gluing and has weakly dense homoclinic points because each tier has these properties separately. Thus from our results it follows that if $f$ is a stabilizing endomorphism of $(\ddot\Z^\Z, \sigma)$, and $f^n(x) \longrightarrow 0^\Z$ for all $x \in \ddot\Z^\Z$, then for all $a \in \N$ we have $f^n([-a,a]^\Z) = \{0^\Z\}$ for some $n \in \N$.

If we do not assume the stabilization property, then evolution maps need not even be uniformly asymptotically nilpotent on the subsystems $[-a,a]^\Z$.
Namely, let $g : \dot \N^\Z \to \dot \N^\Z$ be the non-uniformly asymptotically nilpotent endomorphism from Example~\ref{ex:OnePoint}.
Define $h : \ddot\Z^\Z \to \dot \N^\Z$ by
\[
  h(x)_i =
  \begin{cases}
    g(|x|)_i, & \text{if~} x_i \geq 0, \\
    \infty, & \text{if~} x_i < 0
  \end{cases}
\]
where the absolute value $|x|$ is defined cell-wise.
This function is clearly continuous and commutes with shifts, so the codomain extension $f = g \circ h : \ddot\Z^\Z \to \ddot\Z^\Z$ is an endomorphism of $\ddot\Z^\Z$.
Then $f$ is non-uniformly asymptotically nilpotent. \tqed
\end{example}

The class of endomorphisms of $\ddot\Z^\Z$ includes the \emph{sand automata} \cite{DeGuMa08}, which are those endomorphisms that preserve the positions of $\infty$ and $-\infty$-symbols, and commute with the vertical shift $\sigma^v(x)_i = x_i + 1$. 
However, the above claim does not directly translate to a nontrivial statement for sand automata, as they cannot be asymptotically nilpotent towards $0^\Z$ on any subsystem $[a, b]^\Z$ with $a < b$. 



We also mention that our results of course hold for expansive manifold homeomorphisms, e.g. toral automorphisms and the multiplication map on a solenoid. However, for examples of expansive manifold homeomorphisms that we know, our theorem says nothing nontrivial, as such systems tend to be very `endomorphism-rigid', and one can give direct algebraic and topological proofs of nil-rigidity.
For example, given any expansive homeomorphism on a compact space, any endomorphism is determined by a finite amount of information, namely any good enough approximation of the an endomorphism determines it uniquely. We do not know examples of expansive manifold homeomorphisms where nilpotency is undecidable when local rules are given in this way. On the other hand, nilpotency is algorithmically undecidable for cellular automata \cite{AaLe74,Ka92}.

\section{Block-gluing Two-dimensional SFTs are Nil-rigid}
\label{sec:BlockGluing}

It is a common phenomenon in multidimensional symbolic dynamics (i.e. symbolic dynamics on finitely generated abelian groups) that there is an immense jump in the complexity of dynamical properties between dimension one and two, and a curious smaller gap between dimensions two and three. We mention the famous question of whether a strongly irreducible subshift of finite type in dimension three always has dense periodic points. This is the case in dimension two \cite{Li03}, and in dimension one, even the SFT assumption can be removed \cite{Be88}.

We show in Example~\ref{ex:IrreducibleSFT} that in any dimension (in fact on any group), strongly irreducible SFTs are nil-rigid. In this section, we show that in dimension two, the weaker notion of block-gluing is enough, and leave open whether it suffices in higher dimensions. See \cite{BoRoSc10} for an overview of block-gluing and related mixing notions of $\Z^2$ subshifts.


Let $G$ be a group and $\mathcal{F}$ any set of subsets of $G$ which contains $G$ and is closed under (left) translations, pointwise limits and arbitrary intersections. For $R \in \N$, we say a subshift $X \subset \Sigma^G$ is $(\mathcal{F},R)$-gluing, if whenever $P \in \Sigma^A$ and $Q \in \Sigma^B$ are valid patterns in $X$, $A, B \in \mathcal{F}$ and
\[ d_G(A, B) = \min_{a \in A, b \in B} d_G(a, b) \geq R \]
then the union pattern $P \sqcup Q$ defined by $(P \sqcup Q)|_A = P$ and $(P \sqcup Q)|_B = Q$ is also valid in $X$. We say $X$ is \emph{$\mathcal{F}$-gluing}, if it is $(\mathcal{F},R)$-gluing for some $R \in \N$, which we call the \emph{gluing constant}. Here $d_G$ can be any left-invariant proper metric on $G$, as the choice only impacts the choice of the constant $R$; in the case of general groups we usually use the word metric, while on $\Z^2$ it is customary to use the $\ell_\infty$ metric.
For $A \subseteq G$, write $C_{\mathcal{F}}(A)$ for the unique smallest element of $\mathcal{F}$ containing $A$ as a subset, which exists since $G \in \mathcal{F}$ and $\mathcal{F}$ is closed under intersections.

Let $B_{m,n} = [0,m-1] \times [0,n-1] \subset \Z^2$ and $B_n = B_{n,n}$. These sets and their translates (i.e. all rectangles) are called \emph{blocks}. Let $\mathcal{B}$ be the closure of translates of the sets $B_{m,n}$ with respect to pointwise convergence. A $\Z^2$-SFT is called \emph{block-gluing} if it is $(\mathcal{B}, R)$-gluing for some $R \in \N$. Typically, the infinite blocks are not included in the definition, but a simple compactness argument shows that our definition is equivalent. We also use the following lemma (up to $\lambda = \omega+2$) to avoid cluttering our proof with compactness arguments; its proof is a simple compactness argument.

\begin{lemma}
\label{lem:OrdinalGluing}
 Let $X \subset \Sigma^G$ be a $(\mathcal{F}, R)$-gluing subshift. Let $\lambda$ be an ordinal, and for all $\alpha < \lambda$ suppose $B_\alpha \in \mathcal{F}$ and $d_G ( C_{\mathcal{F}} ( \bigcup_{\beta < \alpha} B_\beta ), B_\alpha ) \geq R$.
Let $P_\alpha \in \Sigma^{B_\alpha}$ be a valid pattern for each $\alpha < \lambda$. Then there exists $x \in X$ such that $x|_{B_\alpha} = P_\alpha$ for all $\alpha < \lambda$.
\end{lemma}

\begin{theorem}
\label{thm:Block}
Let $X \subset \Sigma^{\Z^2}$ be a block-gluing subshift of finite type. Then $X$ is nil-rigid.
\end{theorem}


\begin{proof}
Let $R$ be a gluing constant, and assume without loss of generality that $R$ is also the maximal size of a forbidden pattern of $X$. Let $f : X \to X$ be an asymptotically nilpotent cellular automaton. Let $\X = \mathcal{U}(G, X)$ and let $K = \{0\} \times \Z \leq \Z^2$ be the vertical subgroup.

\begin{claim}
\label{cl:HomoclDenseBlock}
The homoclinic points of $\fix{K}{\X}$ are dense in $X$.
\end{claim}

\begin{proof}
Observe that the tiers of $\fix{K}{\X}$ correspond to nontrivial subgroups of $K$, and thus each vertically periodic point of $X$ is on some tier. Thus we need to show that every pattern of $X$ appears in some vertically periodic point whose support intersects only finitely many columns.

Let $n \in \N$ and let $P \in \Sigma^{B_n}$ be a valid pattern. We show that there is a point in $X$ which is vertically periodic, homoclinic in $\fix{K}{\X}$ and contains a copy of $P$ at the origin. For this, we apply Lemma~\ref{lem:OrdinalGluing} to obtain a point $x \in X$ containing a translate of $P$ at the coordinate $(0, (n+R-1)k)$ for all $k \in \Z$, and containing only zeroes outside the strip $[-R+1, n+R-2] \times \Z$. See Figure~\ref{fig:GluingInstructions}(a) for the gluing instructions.

Since $X$ is an SFT, we can use the pigeonhole principle to extract a periodic point from $x$: Pick $h_1, h_2 \in \Z$ such that $h_2 = h_1 + m(n+R-1)$ for some $m \geq 2$ and $x_{(-R+1, h_1) + B_{n + 2(R-1), n+R-2}} = x_{(-R+1, h_2) + B_{n + 2(R-1), n + R-2}}$. Define $y \in \Sigma^{\Z^2}$ by $y_{(a, h_1 + b (h_2 - h_1) + k)} = x_{(a, h_1 + k)}$ for all $a, b \in \Z$ and $0 \leq k < h_2 - h_1$.
Then $y \in X$ since all patterns of shape $B_R$ that occur in $y$ also occur on $x$, and $y$ contains a translate of $P$ since $h_2 - h_1 \geq 2(n+R-1)$. Finally, we can translate $y$ to force $P$ to occur at the origin. Since $P$ was arbitrary, the homoclinic points of $\fix{K}{\X}$ are dense in $X$.
\end{proof}

By Proposition~\ref{prop:ShadowingIsGluing}, $X$ is $0$-gluing. By Remark~\ref{rem:PeriodTowers} after Lemma~\ref{lem:Towers}, the $K$-shadow $\Theta_K(x)$ of a configuration $x \in X$ can only spread a bounded distance under the iteration of $f$. In the subshift $X$ we can interpret this concretely, as the shadow of a point is always at a bounded distance from its support in the Hausdorff metric: for any configuration $x \in X$ with support $A \subset \Z^2$, the support of $f^n(x)$ is contained in $A + [-C, C] \times \Z$ for all $n \in \N$. It follows that $\fin{K}{\X}$ is a $0$-gluing tiered $\Z$-system and $f$ is its asymptotically nilpotent evolution map. By Lemma~\ref{lem:Z}, $f$ is nilpotent on each tier of $\fin{K}{\X}$. In particular, there exists $m \in \N$ such that for all $x \in X$ whose support is contained in $[0, R-2] \times \Z$, we have $f^m(x) = 0^{\Z^2}$.

\begin{claim}
\label{cl:MortalDenseBlock}
Mortal homoclinic configurations are dense in $f^m(X)$.
\end{claim}

\begin{proof}
Let $[-r, r]^2 \subset \Z^2$ be a neighborhood of $f$, let $n \geq 3 r m$, and let $P \in \Sigma^{B_n}$ be any pattern that appears in a configuration of $f^m(X)$. Let $Q \in \Sigma^{B_n + [-r m, r m]^2}$ be a pre-image of $P$ under $f^m$, meaning that if $y \in X$ contains $Q$ at $\vec v$, then $f^m(y)$ contains $P$ at $\vec v + (rm, rm)$.

We glue $Q$ into a configuration whose support is shaped like the letter `H', with all four arms having thickness $R-1$ and the central bar containing a copy of $Q$. More explicitly, we construct a configuration $x \in X$ such that $x|_{B_n + [-r m, r m]^2} = Q$ and $x_{\vec v} = 0$ for all $v \in H \cup V$ where
\begin{align*}
  V & {} = ((-\infty, -r m - R] \cup [n - 1 + r m + R, \infty)) \times \Z \\
  H & {} = [-r m, n - 1 + r m] \times ((-\infty, -r m - R] \cup [n - 1 + r m + R, \infty))
\end{align*}
See Figure~\ref{fig:GluingInstructions}(b) for the gluing instructions.

Let $Y \subset X$ be the set of those configurations whose support is contained in $[t, t+R-2] \times \Z$ for some $t \in \Z$. We claim that there exists a finite set $T \subset \Z^2$ such that for all $\vec v \in \Z^2 \setminus T$, the pattern $x|_{\vec v + B_n + [-r m, r m]^2}$ occurs in some configuration of $Y$. Namely, if this is not the case, then we can shift $x$ vertically and extract a limit configuration $y \in X$ such that $y|_{\vec v + B_n + [-r m, r m]^2}$ does not occur in $Y$ for some $\vec v \in \Z^2$, and the support of $y$ is contained in the set $([-r m - R + 1, -r m - 1] \cup [n + r m, n + r m + R - 2]) \times \Z$. Since $X$ is defined by forbidden patterns of size at most $R \leq n$, we can erase one of the vertical stripes and obtain a configuration $z \in Y$ whose support is contained in either $[-r m - R + 1, -r m - 1] \times \Z$ or $[n + r m, n + r m + R - 2] \times \Z$. Since $n \geq 3 r m$, one of these choices satisfies $z|_{\vec v + B_n + [-r m, r m]^2} = y|_{\vec v + B_n + [-r m, r m]^2}$, a contradiction.

Recall that $f^m(Y) = \{ 0^{\Z^2} \}$.
Since $[-r m, r m]^2$ is a neighborhood of $f^m$, we have $f^m(x)_{\vec v} = 0$ for all $\vec v \in \Z^2 \setminus T$. In particular, $f^m(x)$ is homoclinic. It is also mortal, since it lies on some tier of $\fin{K}{\X}$, on which $f$ is nilpotent. By the definition of $Q$, the configuration $f^m(x)$ contains $P$. Since $P$ was an arbitrary (large enough) pattern occurring in $f^m(X)$, the claim follows.
\end{proof}

Remark~\ref{rem:Mortal} after Lemma~\ref{lem:Mortal} now shows that $f$ is nilpotent on $X$.
\end{proof}

\begin{figure}[ht]
\begin{center}
\begin{tikzpicture}[scale = 1.5]


\node at (-0.8,5.8) {(a)};

\fill [black!15] (-0.2,-0.1) rectangle (1.2,5.9);

\draw [fill=white] (0, 0)   rectangle (1, 1);
\draw [fill=white] (0, 1.2) rectangle (1, 2.2);
\draw [fill=white] (0, 2.4) rectangle (1, 3.4);
\draw [fill=white] (0, 3.6) rectangle (1, 4.6);
\draw [fill=white] (0, 4.8) rectangle (1, 5.8);
\draw (-0.2,-0.1) -- (-0.2,5.9);
\draw (1.2,-0.1) -- (1.2,5.9);

\node () at (0.5,2.9) {$P_0 = P$};
\node () at (0.5,4.1) {$P_1 = P$};
\node () at (0.5,1.7) {$P_2 = P$};
\node () at (0.5,5.3) {$P_3 = P$};
\node () at (0.5,0.5) {$P_4 = P$};
\node () at (-0.7, 2.9) {$P_{\omega} = 0$};
\node () at (1.7, 2.9) {$P_{\omega + 1} = 0$};


\begin{scope}[xshift=4cm]
\node at (-0.8,5.8) {(b)};

\fill [black!15] (-0.2,-0.1) rectangle (1.2,5.9);

\draw [fill=white] (0, 2.4) rectangle (1, 3.4);
\node () at (0.5, 2.9) {$P_0 = Q$};

\draw (-0.2,-0.1) -- (-0.2,5.9);
\draw (1.2,-0.1) -- (1.2,5.9);

\node () at (-0.7, 2.9) {$P_3 = 0$};
\node () at (1.7, 2.9) {$P_4 = 0$};

\fill [white] (0,-0.1) rectangle (1,2.2);
\draw (0,-0.1) -- (0,2.2) -- (1,2.2) -- (1,-0.1);
\node () at (0.5, 1.4) {$P_2 = 0$};

\fill [white] (0,3.6) rectangle (1,5.9);
\draw (0,5.9) -- (0,3.6) -- (1,3.6) -- (1,5.9);
\node () at (0.5, 4.4) {$P_1 = 0$};
\end{scope}
\end{tikzpicture}

\caption{Gluing instructions for applying Lemma~\ref{lem:OrdinalGluing} in the proof of Theorem~\ref{thm:Block}: (a) the construction of a homoclinic point in $\fix{K}{\X}$ and (b) the construction of an `H'-shape. Here, $P_\alpha = P$ and $P_\alpha = Q$ mean that at ordinal $\alpha$, the pattern is a suitably translated copy of the pattern $P$ or $Q$ respectively, and $P_\alpha = 0$ means that we glue zeroes in the area, i.e. $P_\alpha = 0^B$ for a suitable $B \subset \Z^2$. The gray areas have width $R$, and we have no control over their content.}
\label{fig:GluingInstructions}
\end{center}
\end{figure}
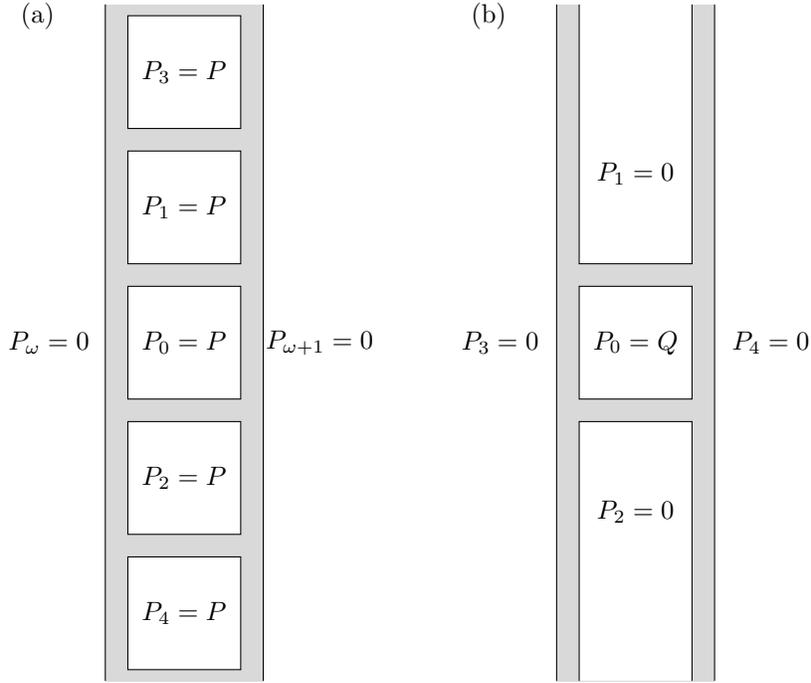

\begin{example}
Consider the SFT $X \subset \{0, 1\}^{\Z^2}$ where each occurrence of $1$ must have another $1$ as its north or east neighbor.
It is block-gluing, so Theorem~\ref{thm:Block} implies that it is nil-rigid.
It does not contain any homoclinic points apart from $0^{\Z^2}$, so our earlier results are not applicable to it.
\tqed
\end{example}

Recall from Section~\ref{sec:Gluing} that a subshift $X \subset \Sigma^{\Z^2}$ is \emph{$0$-to-$0$ sofic} if there exists a subshift of finite type $0^{\Z^2} \in Y \subset \Gamma^{\Z^2}$ and a block code $\pi : Y \to X$ such that $\pi^{-1}(0^{\Z^2}) = \{0^{\Z^2}\}$.
Such subshifts are $0$-gluing by Lemma~\ref{lem:GluingImage}.
A minor modification of the proof of Theorem~\ref{thm:Block} shows that block-gluing $0$-to-$0$ sofic shifts are also nil-rigid.
Note that the SFT $Y$ may not be block-gluing even if $X$ is.

\begin{theorem}
Let $X \subset \Sigma^{\Z^2}$ be a block-gluing $0$-to-$0$ sofic shift. Then $X$ is nil-rigid.
\end{theorem}

\begin{proof}[Proof sketch]
  The proof follows closely that of Theorem~\ref{thm:Block}.
  In the proof of Claim~\ref{cl:HomoclDenseBlock}, instead of $x \in X$, we work with a configuration $y \in Y$ of the SFT cover such that the pattern $P$ occurs in $\pi(y)$.
  In the proof of Claim~\ref{cl:MortalDenseBlock} we use the same trick and work with a $\pi$-preimage of the configuration $y$.
  Note that the distance $n$ needs to be slightly larger to take into account the radius of $\pi$.
\end{proof}


\section*{Acknowledgement}

The authors are thankful to Yves de Cornulier for reviewing the construction in Example~\ref{ex:WeirdGroup} and mentioning the link to the results in \cite{dCoMa07} on MathOverflow.

\bibliographystyle{ws-ijac}
\bibliography{../../../bib/bib}{}

\def\ocirc#1{\ifmmode\setbox0=\hbox{$#1$}\dimen0=\ht0 \advance\dimen0
  by1pt\rlap{\hbox to\wd0{\hss\raise\dimen0
  \hbox{\hskip.2em$\scriptscriptstyle\circ$}\hss}}#1\else {\accent"17 #1}\fi}
\begin{thebibliography}{10}

\bibitem{AaLe74}
S.~O. Aanderaa and H.~R. Lewis, Linear sampling and the $\forall \exists
  \forall$ case of the decision problem, {\em J. Symbolic Logic} {\bf 39} (09
  1974)  519--548.

\bibitem{BaKi16}
L.~Bartholdi and D.~Kielak, {Amenability of groups is characterized by Myhill's
  Theorem}, {\em ArXiv e-prints}  (May 2016) To appear in \emph{Journal of the
  European Mathematical Society}.

\bibitem{Be88}
A.~Bertrand, Specification, synchronisation, average length,  in G.~Cohen and
  P.~Godlewski (eds.), {\em Coding Theory and Applications} Springer Berlin
  Heidelberg,  (1988)  86--95.

\bibitem{Bo71}
R.~Bowen, Periodic points and measures for {Axiom A} diffeomorphisms, {\em
  Transactions of the American Mathematical Society} {\bf 154}  (1971)
  377--397.

\bibitem{BoLiRu88}
M.~Boyle, D.~Lind and D.~Rudolph, The automorphism group of a shift of finite
  type, {\em Transactions of the American Mathematical Society} {\bf 306}(1)
  (1988)  pp. 71--114.

\bibitem{BoRoSc10}
M.~Boyle, R.~Pavlov and M.~Schraudner, Multidimensional sofic shifts without
  separation and their factors, {\em Trans. Amer. Math. Soc.} {\bf 362}(9)
  (2010)  4617--4653.

\bibitem{CaKaTa16}
S.~Capobianco, J.~Kari and S.~Taati, An ``almost dual'' to {G}ottschalk's
  conjecture,  in {\em International Workshop on Cellular Automata and Discrete
  Complex Systems}  (2016)  77--89.

\bibitem{CeCo10}
T.~Ceccherini-Silberstein and M.~Coornaert, {\em Cellular Automata and Groups},
  Springer Monographs in Mathematics (Springer-Verlag, 2010).

\bibitem{CeCo11}
T.~Ceccherini-Silberstein and M.~Coornaert, On algebraic cellular automata,
  {\em Journal of the London Mathematical Society} {\bf 84}(3)  (2011)
  541--558.

\bibitem{dCoMa07}
Y.~de~Cornulier and A.~Mann, Some residually finite groups satisfying laws,  in
  G.~N. Arzhantseva, J.~Burillo, L.~Bartholdi and E.~Ventura (eds.), {\em
  Geometric Group Theory} Birkh{\"a}user Basel,  (2007)  45--50.

\bibitem{DeGuMa08}
A.~Dennunzio, P.~Guillon and B.~Masson, Stable dynamics of sand automata,  in
  G.~Ausiello, J.~Karhum{\"a}ki, G.~Mauri and L.~Ong (eds.), {\em Fifth Ifip
  International Conference On Theoretical Computer Science -- Tcs 2008}
  Springer US,  (2008)  157--169.

\bibitem{GaKuLe78}
P.~G{\'a}cs, G.~L. Kurdyumov and L.~A. Levin, One-dimensional uniform arrays
  that wash out finite islands, {\em Problems Inform. Transmission} {\bf 14}(3)
   (1978)  223--226.

\bibitem{Go73}
W.~Gottschalk, Some general dynamical notions,  in {\em Recent advances in
  topological dynamics} (Springer, 1973) pp. 120--125.

\bibitem{Gr81}
M.~Gromov, Groups of polynomial growth and expanding maps, {\em Publications
  Math{\'e}matiques de l'Institut des Hautes {\'E}tudes Scientifiques} {\bf
  53}(1)  (1981)  53--78.

\bibitem{Gr99}
M.~Gromov, Endomorphisms of symbolic algebraic varieties, {\em Journal of the
  European Mathematical Society} {\bf 1} (Apr 1999)  109--197.

\bibitem{GuRi08}
P.~Guillon and G.~Richard, Nilpotency and limit sets of cellular automata,  in
  {\em Mathematical foundations of computer science 2008}, {\em Lecture Notes
  in Comput. Sci.} {\bf 5162} (Springer, Berlin, 2008) pp. 375--386.

\bibitem{GuRi10}
P.~Guillon and G.~Richard, {Asymptotic behavior of dynamical systems and
  cellular automata}, {\em ArXiv e-prints}  (April 2010).

\bibitem{Ho10}
M.~Hochman, On the automorphism groups of multidimensional shifts of finite
  type, {\em Ergodic Theory Dynam. Systems} {\bf 30}(3)  (2010)  809--840.

\bibitem{Ja74}
A.~V. Jategaonkar, Integral group rings of polycyclic-by-finite groups, {\em
  Journal of Pure and Applied Algebra} {\bf 4}(3)  (1974)  337 -- 343.

\bibitem{Ka92}
J.~Kari, The nilpotency problem of one-dimensional cellular automata, {\em SIAM
  J. Comput.} {\bf 21}(3)  (1992)  571--586.

\bibitem{Li03}
S.~J. Lightwood, Morphisms from non-periodic $\mathbb{Z}^{2}$ subshifts {I}:
  constructing embeddings from homomorphisms, {\em Ergodic Theory and Dynamical
  Systems} {\bf 23}(2)  (2003) p. 587–609.

\bibitem{LiMa95}
D.~Lind and B.~Marcus, {\em An introduction to symbolic dynamics and coding}
  (Cambridge University Press, Cambridge, 1995).

\bibitem{Ma40}
A.~Malcev, On isomorphic matrix representations of infinite groups, {\em
  Matematicheskii Sbornik} {\bf 50}(3)  (1940)  405--422.

\bibitem{MeSa17}
T.~{Meyerovitch} and V.~{Salo}, {On pointwise periodicity in tilings, cellular
  automata and subshifts}, {\em ArXiv e-prints}  (March 2017) To appear in
  \emph{Groups, Geometry and Dynamics}.

\bibitem{Mo62}
E.~F. Moore, {\em {Machine Models of Self-Reproduction}}
\newblock Mathematical problems in the biological sciences, (American
  Mathematical Society, 1962), pp. 17--33.

\bibitem{Mo11}
S.~Moriceau, Cellular automata on a {$G$}-set, {\em J. Cell. Autom.} {\bf 6}(6)
   (2011)  461--486.

\bibitem{My63}
J.~Myhill, Shorter note: The converse of {Moore's} {Garden-of-Eden} theorem,
  {\em Proceedings of the American Mathematical Society} {\bf 14}(4)  (1963)
  685--686.

\bibitem{OsTi14}
A.~V. Osipov and S.~B. Tikhomirov, Shadowing for actions of some finitely
  generated groups, {\em Dyn. Syst.} {\bf 29}(3)  (2014)  337--351.

\bibitem{Pa16}
R.~Pavlov, On intrinsic ergodicity and weakenings of the specification
  property, {\em Advances in Mathematics} {\bf 295}(Supplement C)  (2016)  250
  -- 270.

\bibitem{Ro76}
J.~E. Roseblade, Applications of the {Artin-Rees} lemma to group rings,  in
  {\em Symposia Mathematica}, ~{\bf 17}  (1976)  471--478.

\bibitem{Sa12c}
V.~Salo, On nilpotency and asymptotic nilpotency of cellular automata,  in
  E.~Formenti (ed.), {\em Proceedings 18th international workshop on Cellular
  Automata and Discrete Complex Systems and 3rd international symposium
  Journ{\'e}es Automates Cellulaires} Open Publishing Association,  (2012)
  86--96.

\bibitem{Sa16}
V.~{Salo}, {Subshifts with sparse projective subdynamics}, {\em ArXiv e-prints}
   (May 2016).

\bibitem{Sa17a}
V.~{Salo}, {No Tits alternative for cellular automata}, {\em ArXiv e-prints}
  (September 2017).

\bibitem{Sa17}
V.~Salo, Strict asymptotic nilpotency in cellular automata,  in {\em Cellular
  automata and discrete complex systems}, {\em Lecture Notes in Comput. Sci.}
  {\bf 10248} (Springer, Cham, 2017) pp. 3--15.

\bibitem{SaTo17}
V.~Salo and I.~T{\"o}rm{\"a}, Independent finite automata on {C}ayley graphs,
  {\em Natural Computing} {\bf 16}(3)  (2017)  411--426.

\bibitem{Si74}
K.~Sigmund, On dynamical systems with the specification property, {\em
  Transactions of the American Mathematical Society} {\bf 190}  (1974)
  285--299.

\bibitem{Ti72}
J.~Tits, Free subgroups in linear groups, {\em Journal of Algebra} {\bf 20}(2)
  (1972)  250 -- 270.

\bibitem{To15}
I.~T\"{o}rm\"{a}, A uniquely ergodic cellular automaton, {\em Journal of
  Computer and System Sciences} {\bf 81}(2)  (2015)  415 -- 442.

\bibitem{Tr85}
V.~I. Trofimov, Graphs with polynomial growth, {\em Sbornik: Mathematics} {\bf
  51}(2)  (1985)  405--417.

\bibitem{Wa17}
S.~Wacker, Cellular automata on group sets and the uniform
  {Curtis--Hedlund--Lyndon} theorem, {\em Natural Computing}  (Oct 2017).

\bibitem{We09}
B.~A.~F. Wehrfritz, {\em Group and ring theoretic properties of polycyclic
  groups}, ~{\bf 10} of {\em Algebra and Applications} (Springer-Verlag London,
  Ltd., London, 2009).

\bibitem{We00}
B.~Weiss, Sofic groups and dynamical systems, {\em Sankhyā: The Indian Journal
  of Statistics, Series A (1961-2002)} {\bf 62}(3)  (2000)  350--359.

\end{thebibliography}

\end{document}